\documentclass[11pt,reqno]{amsart}

\usepackage{lipsum}
\usepackage{amsmath}
 \usepackage{amsaddr}

\usepackage[utf8]{inputenc}
\usepackage{amssymb}

\usepackage{ragged2e}
\usepackage{xfrac}
\usepackage{fancyvrb}

\usepackage{amsthm}
\usepackage{amsfonts}
\usepackage{amssymb}
\usepackage{tikz}
\usepackage{color}
\usepackage{float}
\usepackage[makeroom]{cancel}
\usepackage{ytableau}
\usepackage{mathdots}
\usepackage{multicol}
\usepackage{mathtools}
\usepackage{stmaryrd}
\usepackage{multirow}
\usepackage{pdflscape}

\usepackage[ruled,vlined]{algorithm2e}
\usepackage[margin=3cm]{geometry}

\newtheorem{teo}{Theorem}[section]
\newtheorem{cor}[teo]{Corollary}
\newtheorem{prop}[teo]{Proposition}
\newtheorem{lema}[teo]{Lemma}

\theoremstyle{definition}

\newtheorem{defin}[teo]{Definition}
\newtheorem{obs}[teo]{Remark}
\theoremstyle{remark}
\newtheorem{ex}[teo]{Example}

\DeclareMathSymbol{\shortminus}{\mathbin}{AMSa}{"39}

\numberwithin{equation}{section}

\title{An action of the cactus group on shifted tableau crystals}
\author{Inês Rodrigues} 
\email{imarrodrigues@fc.ul.pt}
\thanks{The author is partially supported by the Lisbon Mathematics PhD program (funded by the Portuguese Science Foundation). This research was made within the activities of the Group for Linear, Algebraic and Combinatorial Structures of the Center for Functional Analysis, Linear Structures and Applications (University of Lisbon), and was partially supported by FCT - Fundação para a Ciência e Tecnologia, under the project UIDB/04721/2020.}

\usepackage{hyperref}
\hypersetup{colorlinks=true,citecolor=blue,linkcolor=magenta,breaklinks=true}
\definecolor{lgray}{rgb}{0.65, 0.65, 0.65}

\begin{document}
\begin{abstract}
Recently, Gillespie, Levinson and Purbhoo introduced a crystal-like structure for shifted tableaux, called the shifted tableau crystal. We introduce, on this structure, a shifted version of the crystal reflection operators, which coincide with the restrictions of the shifted Schützenberger involution to any primed interval of two adjacent letters. Unlike type $A$ Young tableau crystals, these operators do not realize an action of the symmetric group on the shifted tableau crystal since the braid relations do not need to hold. Following a similar approach as Halacheva, we exhibit a natural internal action of the cactus group on this crystal, realized by the restrictions of the shifted Schützenberger involution to all primed intervals of the underlying crystal alphabet, containing, in particular, the aforesaid action of the shifted crystal reflection operator analogues. 
\end{abstract}

\maketitle

\section{Introduction}
Young tableaux, as well as shifted tableaux, arise in many areas of mathematics, such as algebraic combinatorics, geometry or representation theory \cite{Gi19,Pr91,Sag90,Stem89}. While the first have their original role in the representation theory of symmetric groups \cite{JK81,Yo86}, the latter have their origin in projective representations \cite{HoHu92,Joz89}, due to I. Schur \cite{Sch1911}, and their connection to geometry was shown by Hiller and Boe \cite{HB86} and Pragacz \cite{Pr91}. One important tool for the study of Young tableaux are Kashiwara crystals \cite{BumpSchi17,Kash95}. A \emph{Kashiwara crystal} of type $A$ (for $GL_n$) is a non-empty set $\mathcal{B}$ together with partial maps $e_i, f_i : \mathcal{B} \longrightarrow \mathcal{B} \sqcup \{\emptyset\}$, length functions $\varepsilon_i, \varphi_i : \mathcal{B} \longrightarrow \mathbb{Z}$, for $i \in I = [n-1]$, and weight function $\mathsf{wt}: \mathcal{B} \longrightarrow \mathbb{Z}^n$ satisfying certain axioms (see, for example,  \cite[Definition 2.13]{BumpSchi17}). This crystal may be regarded as a coloured and weighted directed acyclic graph, with vertices in $\mathcal{B}$ and $i$-coloured edges $y \xrightarrow{i} x$ if and only if $f_i (y) = x$, for $i \in I$. The set of semistandard Young tableaux of a given shape, in the alphabet $[n]$, is known to provide a model for Kashiwara type $A$ crystals \cite[Chapter 3]{BumpSchi17}, with coplactic operators $e_i$ and $f_i$ defined in terms of reading words. This crystal is isomorphic to the crystal basis of an irreducible $U_q(\mathfrak{gl}_n)$-module. The Schützenberger involution \cite{Schu76}, also known as Lusztig involution \cite{Lusz91}, is defined on the type $A$ Young tableau crystals, as a set map on  $\mathcal{B}$, and acts on its graph structure by ``flipping" it upside down, while reverting the orientation of arrows and its colours. This involution is realized by the evacuation, for straight shapes \cite{Schu76}, or its coplactic extension, often called reversal, for skew shapes \cite{BSS96,Haim92}.

Recently, Gillespie, Levinson and Purbhoo \cite{GLP17} and Gillespie, Levinson \cite{GL19} introduced a crystal-like structure on shifted tableaux. This structure has vertices the skew shifted tableaux, for a given shape $\lambda/\mu$, on the primed alphabet $ [n]' $, and double edges, corresponding to the action of the primed and unprimed lowering and raising operators which commute with the shifted \textit{jeu de taquin}. Each connected component has an unique highest weight element, a shifted skew tableau where each primed and unprimed raising operator is equal to $\emptyset$, which is a Littlewood-Richardson-Stembridge (LRS) tableau of shape $\lambda/\mu$ \cite{Stem89}. Similarly, it has a unique lowest weight element, a shifted skew tableau such that each primed and unprimed lowering operator is equal to $\emptyset$, which is the reversal of the highest weight element. The primed and unprimed operators considered separately yield a type $A$ Kashiwara crystal. We remark that this structure is not a queer crystal\footnote{Hence the terminology ``crystal-like structure". However, we will henceforth use the term \emph{crystal} to refer to this structure, whenever there is no risk of ambiguity.} and differs from the one in \cite{AsOg18,GHPS18}, which is indeed a crystal for the quantum queer Lie superalgebra.

Crystal reflection operators were originally defined by Lascoux and Schützenberger \cite{LaSchu81} in type $A$ tableau crystals and they have been shown to define an action of the symmetric group on those crystals. Kashiwara 
\cite[Theorem 7.2.2]{Kash94}, \cite[Theorem 11.1]{Kash95} defined the Weyl group action on arbitrary normal crystals. 
Halacheva \cite{Hala16,Hala20} has shown that there is an internal action of the cactus group in any normal crystal via partial Schützenberger involutions. When considering subintervals of adjacent letters, the action of the cactus group agrees with the action of the corresponding Weyl group generators. Indeed, the internal action factors through the quotient of this group by the braid relations of the corresponding Weyl group \cite{Hala16,Hala20,HaKaRyWe20}.

We introduce a shifted version of the crystal reflection operators in type $A$, Definition \ref{def:crystalreflection}. In Theorem \ref{sigmarever}, we show that, similarly to type $A$, they coincide with the restrictions of the shifted Schützenberger involution to the primed interval of adjacent letters $[i,i+1]'=\{i',i,(i+1)',i+1\}\subseteq  [n]'$, for any $i \in I$. They act on the $\{i',i\}$-coloured components of the shifted tableau crystal by a double reflection through vertical and horizontal axes, rather than a simple reflection as in the Young tableau crystal.

Unlike type $A$ crystals, they do not define a natural action of the symmetric group $\mathfrak{S}_n$ on the shifted tableau crystal, since the braid relations do not need to hold, as shown in Example \ref{exbraidnot}. Following a similar approach as Halacheva \cite{Hala16, Hala20}, we then show in Theorem \ref{teo:cactusaction} that the restrictions of the shifted Schützenberger involution on the primed subintervals of $[n]$ yield an internal action of the cactus group $J_n$ on that crystal. We note that this internal action on the shifted tableau crystal, unlike the one on type $A$ crystals, does not factor through the braid relations of the symmetric group. When the shifted Schützenberger involution is restricted to primed subintervals of two adjacent letters, the cactus group action agrees with the action of the shifted crystal reflection operators on the shifted crystal. This means that both actions agree as permutations of the vertices within each $\{i',i\}$-coloured component of the shifted crystal. 

The cactus group $J_n$ first appeared in the works of Devadoss \cite{Dev99} and Davis, Januszkiewicz and Scott \cite{DJS03}, as the fundamental group of the quotient orbifold of $\overline{M}_{0}^{n+1} (\mathbb{R})$, the Deligne-Mumford moduli space of stable curves of genus 0 with $n+1$ marked points, by the action of $\mathfrak{S}_n$ that permutes the first $n$ of those points. It is expected, although we have not attempted to explore it, that this combinatorial internal action of the cactus group on the shifted crystal carries some geometrical meaning, as this crystal has its origin in the orthogonal Grassmannian \cite{GLP19}. Indeed, this is the case for $\mathfrak{g}$-crystals, for $\mathfrak{g}$ a semi-simple Lie algebra \cite{Hala16,Hala20,HaKaRyWe20}. Moreover, the tensor product of shifted tableau crystals is not known, and consequently, nor an external action of the cactus group.
 
This paper has the following structure: Section \ref{sec2} provides the basic notions on shifted tableaux. We recall the definition of words and tableaux, following \cite{GLP17}, as well as the notion of shifted \textit{jeu de taquin} and shifted evacuation. Section \ref{sec3} recalls the main concepts on the shifted tableau crystal of \cite{GLP17}. We provide more details on the $i$-string decomposition of such crystals (which corresponds to the axiom (B1) in \cite{GL19} and (A2) in \cite{GLP17}). We then introduce, in Section \ref{sec4}, the shifted crystal reflection operators and prove their coincidence with the restriction of the shifted Schützenberger involution to a marked alphabet of two adjacent letters, Theorem \ref{sigmarever}. In Section \ref{sec5}, we recall the definition of the cactus group \cite{HenKam06} and then we prove Theorem \ref{teo:cactusaction}, the main result, which presents an action of the cactus group on the shifted tableau crystal. Some additional examples are discussed in Appendix \ref{appendix}. 

An extended abstract \cite{Ro20} of this paper was accepted in the Proceedings of the 32nd Conference on Formal Power
Series and Algebraic Combinatorics.

\section{Background}\label{sec2}

This section is intended to provide the basic definitions and results on shifted tableaux, words, and involutions and algorithms among them. We mainly follow the notations in \cite{GL19, GLP17}.

A \emph{strict partition} is a sequence $\lambda = (\lambda_1, \ldots, \lambda_k)$ of positive integers such that $\lambda_1 > \ldots > \lambda_k$. The \emph{size} of $\lambda$ is $|\lambda| = \sum \lambda_i$. The entries $\lambda_i$ are called the \emph{parts} of $\lambda$ and the \emph{length} of $\lambda$, denoted $\ell(\lambda)$, is the number of non-zero parts of $\lambda$. A strict partition $\lambda$ is identified with its \emph{shifted shape} $S(\lambda)$ which consists of $|\lambda|$ boxes placed in $\ell(\lambda)$ rows, with the $i$-th row having $\lambda_i$ boxes and being shifted $i-1$ units to the right. We use the English (or matrix) notation. The boxes in $\{(1,j), (2,j+1), (3,j+2), \ldots\}$ form a \emph{diagonal}, for $j \geq 1$. If $j=1$ it is called the \emph{main diagonal}. Given strict partitions $\lambda$ and $\mu$ such that $S(\mu) \subseteq S(\lambda)$, we write $\mu \subseteq \lambda$ and define the \emph{skew shifted shape} of $\lambda/\mu$ as $ S(\lambda/\mu) = S(\lambda) \setminus S(\mu)$ (see Figure \ref{fig:lambdamu}). Shapes of the form $\lambda/\emptyset$ are called \emph{straight} (or \emph{normal}). Note that the shifted shape $\lambda$ lies naturally in the ambient triangle of the \emph{shifted staircase shape} $\delta = (\lambda_1, \lambda_1-1, \ldots, 1)$. We define the \emph{complement} of $\lambda$ to be the strict partition $\lambda^{\vee}$ whose set of parts is the complement of the set of parts of $\lambda$ in $\{\lambda_1, \lambda_1-1, \ldots, 1\}$. In particular, $\emptyset^{\vee} = \delta$ (see Figure \ref{fig:lambdamu}).

\begin{figure}[h]
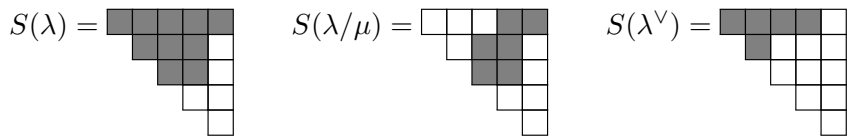

\begin{center}
$S(\lambda)=
	\ytableausetup{smalltableaux}	
	\begin{ytableau}
    *(gray)  & *(gray)  &*(gray)  &*(gray)  &*(gray)  \\
    \none & *(gray)  & *(gray)  & *(gray) &   \\
    \none & \none & *(gray) & *(gray)  & \\
    \none & \none & \none & & \\
    \none & \none & \none & \none &
  \end{ytableau}
  \qquad
  S(\lambda/\mu)=
	\ytableausetup{smalltableaux}	
	\begin{ytableau}
    {}  & {} & {} &*(gray)  &*(gray)  \\
    \none & {}  & *(gray)  & *(gray) &   \\
    \none & \none & *(gray) & *(gray) & \\
    \none & \none & \none & & \\
    \none & \none & \none & \none &
  \end{ytableau}
  \qquad
 S(\lambda^{\vee})=
	\ytableausetup{smalltableaux}	
	\begin{ytableau}
    *(gray)  & *(gray)  &*(gray)  &*(gray)  &   \\
    \none & *(gray)  &    &  &   \\
    \none & \none &   &  & \\
    \none & \none & \none & & \\
    \none & \none & \none & \none &
  \end{ytableau}
  $
 \end{center}
 \caption{The shapes of $\lambda$, $\lambda/\mu$ and $\lambda^{\vee}$, shaded in gray, for $\lambda= (5,3,2)$ and $\mu = (3,1)$.}
 \label{fig:lambdamu}
 \end{figure}

We set $[n] := \{1 < \ldots < n\}$ and define the \emph{primed} (or marked) alphabet  $[n]':=\{1' < 1 < \ldots n' < n\}$. When referring to the letters $i$ and $i'$ without specifying whether they are primed, we write $\mathbf{i}$. Given a string $w = w_1 \ldots w_m$ in the alphabet $[n]'$, the \emph{first} $i$ or $i'$ will be the leftmost entry of $w$ equal to $i$ or $i'$, for $1 \leq i \leq n$. The \emph{canonical form} of $w$ is the string obtained from $w$ by replacing the first $i$ or $i'$ (if it exists) with $i$, for all $1 \leq i \leq n$. Two strings $w$ and $v$ are said to be \emph{equivalent}, denoted by $w \simeq v$, if they have the same canonical form (this is an equivalence relation).

\begin{defin}[\cite{GLP17} Definition 2.2]
A \emph{word} $\hat{w}$ is an equivalence class of strings. The representative in canonical form is called the \emph{canonical representative}. The \emph{weight} of a word $\hat{w}$ is $\mathsf{wt}(\hat{w}) = (wt_1, \ldots, wt_n)$, where $wt_i$ is equal to the total number of $i$ and $i'$ in $\hat{w}$.
\end{defin}

\begin{ex}
The string $w=12'2'1123'2'2$ is equivalent to $\hat{w} = 122'11232'2$, the latter being the canonical form of the former. The weight of $\hat{w}$ is $(3,5,1)$.
\end{ex}

A \emph{partial operator} on a set $S$ is a map $A: T \longrightarrow S$, where $T \subseteq S$, such that $A(s) = \emptyset$ for $s \not\in T$. In this case, it is said that $A$ is \emph{undefined} on $s$. Otherwise, it is said to be \emph{defined}. Given a partial operator $A$ on the set of finite strings in the alphabet $[n]'$, we say that $A$ is \emph{defined on words} \cite[Definition 2.4]{GLP17} if $A(v)$ is defined on some representative $v$ of $\hat{w}$ and if $A(u) \simeq A(v)$ for any $u$ and $v$ representatives of $\hat{w}$. The \emph{induced operator} $\hat{A}$ on words is then defined as $\hat{A}(w) = \widehat{A(v)}$ if such a representative $v$ exists, and $\hat{A}(w) = \emptyset$ otherwise. For simplicity, from now on we refer to a word by $w$ instead of $\hat{w}$ and to an operator defined on a word by $A$ instead of $\hat{A}$.

\begin{defin}
Let $\lambda$ and $\mu$ be strict partitions such that $\mu \subseteq \lambda$. A \emph{shifted semistandard (Young) tableau} $T$ of shape $\lambda / \mu$ is a filling of $S(\lambda/\mu)$ with letters in $\{1' < 1 < \ldots\}$ such that:
	\begin{enumerate}
	\item The entries are weakly increasing in each row and in each column.
	\item There is at most one $i$ per column, for any $i \geq 1$.
	\item There is at most one $i'$ per row, for any $i \geq 1$.
	\end{enumerate}
\end{defin}

The \emph{(row) reading word} $w(T)$ of such a tableau is formed by reading the entries of $T$ from left to right, going bottom to top. The \emph{weight} of $T$ is defined as $\mathsf{wt}(T)=\mathsf{wt}(w(T))$. A shifted tableau is said to be \emph{standard} if its weight is $(1, \ldots, 1)$. We say that a tableau $T$ is in \emph{canonical form} if so it is $w(T)$. A tableau $T$ in canonical form is identified with its set of \emph{representatives}, that are obtained by possibly priming the entry corresponding to the first $i$ in $w(T)$, for all $i$. We denote by $\mathsf{ShST}(\lambda/\mu,n)$ the set of shifted semistandard tableaux of shape $\lambda/\mu$, on the alphabet $[n]'$, in canonical form.

\begin{ex}\label{semistandard}
The following is a shifted semistandard tableau, with $\lambda=(6,5,3)$, $\mu=(2,1,0)$ and $n=3$, in the canonical form, with its word and weight.

$$T=\begin{ytableau}
{} & {} & {} & 1 &1 & 2'\\
\none & {} & 2 & 3' & 3\\
\none & \none & 3 & 3
\end{ytableau}\qquad
w(T)=3323'3112' \qquad
\mathsf{wt}(T)=(2,2,4)$$

The following shifted semistandard tableaux are the representatives of $T$ that are not in canonical form:
$$\begin{ytableau}
{} & {} & {} & 1' &1 & 2'\\
\none & {} & 2 & 3' & 3\\
\none & \none & 3 & 3
\end{ytableau}
\quad
\begin{ytableau}
{} & {} & {} & 1 &1 & 2'\\
\none & {} & 2' & 3' & 3\\
\none & \none & 3 & 3
\end{ytableau}
\quad
\begin{ytableau}
{} & {} & {} & 1 &1 & 2'\\
\none & {} & 2 & 3' & 3\\
\none & \none & 3' & 3
\end{ytableau}
\quad
\begin{ytableau}
{} & {} & {} & 1' &1 & 2'\\
\none & {} & 2' & 3' & 3\\
\none & \none & 3 & 3
\end{ytableau}
$$
$$
\begin{ytableau}
{} & {} & {} & 1' &1 & 2'\\
\none & {} & 2 & 3' & 3\\
\none & \none & 3' & 3
\end{ytableau}
\quad
\begin{ytableau}
{} & {} & {} & 1 &1 & 2'\\
\none & {} & 2' & 3' & 3\\
\none & \none & 3' & 3
\end{ytableau}
\quad
\begin{ytableau}
{} & {} & {} & 1' &1 & 2'\\
\none & {} & 2' & 3' & 3\\
\none & \none & 3' & 3
\end{ytableau}$$

\end{ex}

A \emph{diagonally-shaped tableau} is a skew shifted tableau of shape $(2n-1, 2n-3, \ldots,1) / (2n-2, 2n-4, \ldots, 2)$. Every word $w=w_1 \ldots w_n$ may be regarded as a shifted tableau $D_w$ of such shape.

\begin{ex}\label{diagonal}
The word $w=2311'$ is the reading word of
$$D_w=\begin{ytableau}
{} & {} & {} & {} & {} & {} & {1'}\\
\none & {} & {} & {} & {} & {1}\\
\none & \none & {} & {} & 3\\
\none & \none & \none & 2
\end{ytableau}$$
\end{ex}

\subsection{The shifted \textit{jeu de taquin}, Knuth equivalence and dual equivalence}

A skew shape $S(\lambda/\mu)$ is said to be a \emph{border strip} if it contains no subset of the form $\{(i,j),(i+1,j+1)\}$ and a \emph{double border strip} if it contains no subset of the form $\{(i,j),(i+1,j+1),(i+2,j+2)\}$. A shifted semistandard tableau $T$ is a disjoint union of border strips $\bigsqcup\limits_{i} T^i$, where $T^{i}$ is the tableau obtained from $T$ considering only the entries filled with $\mathbf{i}$. Given strict partitions $\nu \subseteq \mu \subseteq \lambda$, we say that $\lambda/\mu$ \emph{extends} $\mu/\nu$, and, in this case, we define
$$(\lambda/\mu) \sqcup (\mu/\nu) := \lambda/\nu.$$

If $\lambda/\mu$ extends a shape $B$ consisting of a single box, then $B$ is said to be an \emph{inner corner} of $\lambda/\mu$. Similarly, if $B$, consisting of a single box, extends $\lambda/\mu$, we say that $B$ is an \emph{outer corner}.

\begin{defin}[\cite{Wor84}, Section 6.4]
Let $T\in \mathsf{ShST}(\lambda/\mu,n)$.  An \emph{inner jeu de taquin slide} is the process in which an empty inner corner of the skew shape of $T$ is chosen and then either the entry to its right or the one below it is chosen to slide into the empty square, in such way that the tableau is still semistandard, and then repeating the process with the obtained empty square until it is an outer corner. An \emph{outer jeu de taquin slide} is the reverse process, starting with an outer corner. This process has an exception to the sliding rules when the empty box of an inner or outer slide enters in the diagonal. If an inner slide moves a box with $a'$ to the left into the diagonal and then moves a box with $a$ up from the diagonal, to the right of it, the former becomes unprimed (and vice versa for the corresponding outer slide), as illustrated by the following slide:

\begin{center}
\includegraphics[scale=0.4]{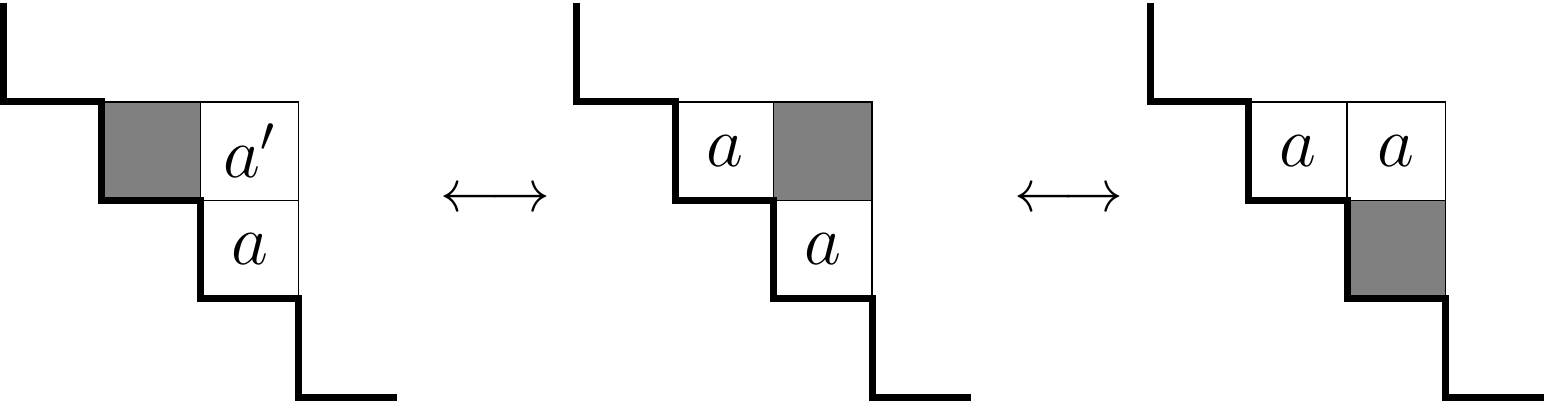}
\end{center}
\end{defin}

If $T$ is not in the canonical form, there is another exception to consider (observe that this illustration
is in the same canonical class of the former):
\begin{center}
\includegraphics[scale=0.4]{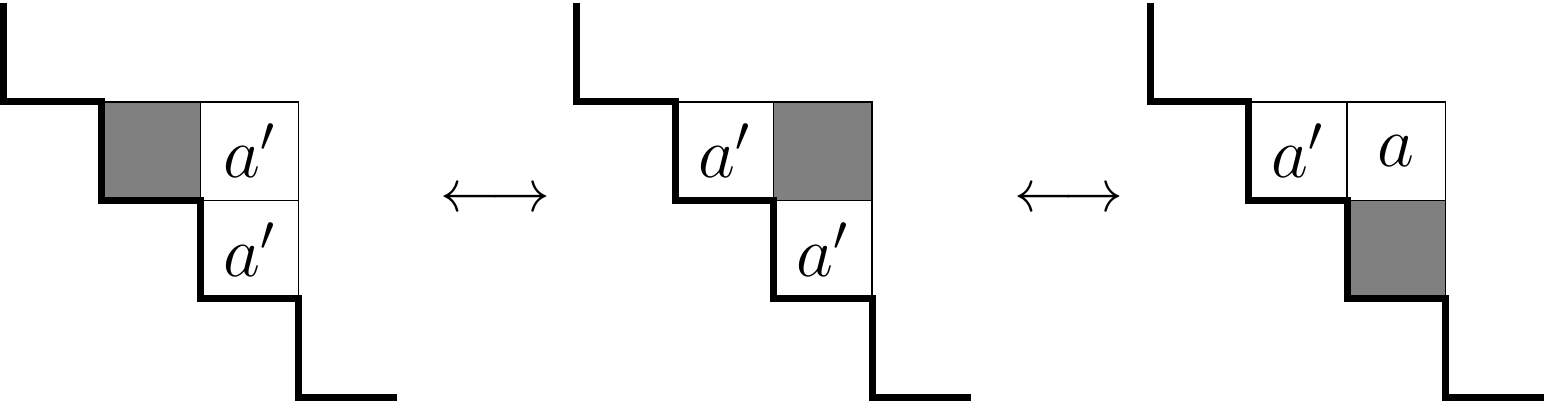}
\end{center}

The \emph{rectification} $\mathsf{rect}(T)$ of $T$ is the tableau obtained by applying a sequence of inner slides until a straight shape is obtained. This is well defined as the rectification process does not depend on the chosen sequence of slides \cite{Wor84}, \cite[Theorem 11.1]{Sag87}. The \emph{rectification} of a word $w$ is the rectification of any tableau with reading word $w$ (in particular, the diagonally-shaped tableau $D_w$ with word $w$). Two tableaux are said to be \emph{shifted jeu de taquin equivalent} (or plactic equivalent) if they have the same rectification. In Example \ref{diagonal}, the rectification of $w=2311'$ is $1123$.
The \emph{standardization} of a word $w$, denoted $\mathsf{std}(w)$, is obtained by replacing in order the letters in any representative of $w$ with $1, \ldots, n$, from least to greatest, reading right to left for primed entries, and left to right for unprimed entries. This does not depend on the choice of representative. The \emph{standardization} of a shifted tableau $T$ is the tableau of the same shape as $T$ with reading word $\mathsf{std}(w(T))$.
In both cases, the standardization is independent of the choice of representative. For example, the standardization of the word $w=3323'3112'$ is $\mathsf{std}(w) = 67458123$.

\begin{lema}[\cite{GLP17}, Lemma 3.5]\label{standard} 
If $s$ is a word in $[n]$, with $n =a_1+\ldots+a_k$, and such that $\mathsf{wt}(s) = (1,\ldots,1)$, then there is at most one word $w$ of weight $(a_1, \dots , a_k)$ with standardization $std(w) = s$.
\end{lema}

Given $\nu$ a strict partition, the \emph{Yamanouchi tableau} of shape $\nu$, denoted $Y_{\nu}$, is the shifted semistandard tableau whose $i$-th row is filled only with unprimed $i$'s, for all $i= 1, \ldots, \ell(\nu)$.

\begin{prop}\label{yamunique}
If $\nu$ is a strict partition, there is a unique shifted tableau of shape and weight $\nu$, up to canonical form, which is $Y_{\nu}$.
\end{prop}

\begin{proof}
We construct such a tableau of shape and weight $\nu$, on the primed alphabet $\{1,\ldots,\ell(\nu)\}'$, starting on the first row. Since $\mathsf{wt}_1 = \nu_1$ and there are $\nu_1$ boxes on the first row, to ensure semistandardness each box must be filled with $1$'s, except for the first one which may be filled with $1'$ or $1$. If it is $1'$, it is identified with $1$ in canonical form, as this is the first occurrence. The process for the remaining rows is the same.
\end{proof}

\begin{ex}
If $\nu = (4,3,1)$, then,
$$Y_{\nu} = \begin{ytableau}
1 & 1 & 1 & 1\\
\none & 2 & 2 & 2\\
\none & \none & 3
\end{ytableau}$$
\end{ex}

\begin{defin}\label{ballot}
A word $w$ on the alphabet $[n]'$ with weight $\nu$, a strict partition, is said to be \emph{ballot} (or \emph{lattice}, or \emph{Yamanouchi}) if its rectification is $w(Y_\nu)$.
\end{defin}

\begin{defin}[\cite{Stem89,Wor84}]
A shifted semistandard tableau $T$ of weight $\nu$, a strict partition, is said to be \emph{Littlewood-Richardson-Stembridge} (LRS) if $\mathsf{rect}(T)=Y_{\nu}$. Equivalently, the reading word of such a tableau is a \emph{ballot} word of weight $\nu$.

\end{defin}

There is another formulation for LRS tableaux due to Stembridge, using some statistics on its word (for details, see \cite[Theorem 8.3]{Stem89}). Given strict partitions $\lambda$, $\mu$ and $\nu$, such that $|\lambda| = |\mu| + |\nu|$, the \emph{shifted Littlewood-Richardson coefficient} $f_{\mu,\nu}^{\lambda}$ is defined to be the number of LRS tableaux of shape $\lambda/\mu$ and weight $\nu$.

\begin{ex}\label{lrs}
The following tableau of shape $(6,5,2,1)/(4,2)$ and weight $(4,3,1)$ rectifies to $Y_{(4,3,1)}$, thus it is a LRS tableau. Note that its word $322112'1'1$ is a ballot word with weight $(4,3,1)$.
$$\begin{ytableau}
 {} &  & {} & {} & {1'} & 1\\
\none & {} & {} & 1 & 1 & 2'\\
\none & \none & 2 & 2\\
\none & \none & \none & 3
\end{ytableau}$$

\begin{align*}
\begin{ytableau}
 {} &  & {} & {} & {1'} & 1\\
\none & {} & *(gray){} & 1 & 1 & 2'\\
\none & \none & 2 & 2\\
\none & \none & \none & 3
\end{ytableau}
&\longrightarrow
\begin{ytableau}
 {} &  & {} & {} & {1'} & 1\\
\none & *(gray){} & 1 & 1 & 2'\\
\none & \none & 2 & 2\\
\none & \none & \none & 3
\end{ytableau}
\longrightarrow
\begin{ytableau}
 {} &  & {} & *(gray){} & {1'} & 1\\
\none & 1 & 1 & 2'\\
\none & \none & 2 & 2\\
\none & \none & \none & 3
\end{ytableau}
\longrightarrow
\begin{ytableau}
 {} & {}& *(gray){} & {1'} & 1\\
\none & 1 & 1 & 2'\\
\none & \none & 2 & 2\\
\none & \none & \none & 3
\end{ytableau}\\
&\longrightarrow
\begin{ytableau}
 {} & *(gray){} & {1'} & 1\\
\none & 1 & 1 & 2'\\
\none & \none & 2 & 2\\
\none & \none & \none & 3
\end{ytableau}
\longrightarrow
\begin{ytableau}
*(gray){} & {1'} & 1 & 1\\
\none & 1 & 2' & 2\\
\none & \none & 2 & 3
\end{ytableau}
\longrightarrow
\begin{ytableau}
1 & 1 & 1 & 1\\
\none & 2 & 2 & 2\\
\none & \none & 3
\end{ytableau}
\end{align*}
\end{ex}

\begin{defin}[\cite{Sag87}]
Two words $w$ and $v$ on an alphabet $[n]'$ are said to be \emph{shifted Knuth equivalent}, denoted $w \equiv_k v$, if one can be obtained from the other by applying a sequence of the following Knuth moves on adjacent letters
	\begin{description}
	\item[(K1)] $bac \longleftrightarrow bca$ if, under the standardization ordering, $a < b < c$.
	\item[(K2)] $acb \longleftrightarrow cab$ if, under the standardization ordering, $a < b < c$.
	\item[(S1)] $ab \longleftrightarrow ba$ if these are the first two letters.
	\item[(S2)] $aa \longleftrightarrow aa'$ if these are the first two letters.
	\end{description}
\end{defin}

\begin{ex}
Let $w=212'21$. We have $\mathsf{std}(w) = 41352$, and then the last 1 is less than 2', which is less that the last 2, in standardization ordering. Thus, $w \equiv_k 212'12$.
\end{ex}

Observe that the shifted Knuth moves above may be performed via (inner or outer) \textit{jeu de taquin} slides. If $a < b < c$ in standardization order, then the Knuth moves (K1) and (K2) are illustrated by:

$$\begin{ytableau}
{} & {} & a\\
\none & b & c
\end{ytableau}
\longrightarrow
\begin{ytableau}
{} & a & c \\
\none & b
\end{ytableau}
\qquad \qquad
\begin{ytableau}
{} & {} & b\\
\none & a & c
\end{ytableau}
\longrightarrow
\begin{ytableau}
{} & a & b \\
\none & c
\end{ytableau}
$$

For the Knuth move (S1), assume, without loss of generality, that $a < b$ in standardization ordering. Then,
$$\begin{ytableau}
{} & {} & a\\
\none & b
\end{ytableau}
\longrightarrow
\begin{ytableau}
a & b
\end{ytableau}$$

Finally the Knuth move (S2) is illustrated by the exception slide

$$\begin{ytableau}
{} & {a'}\\
\none & a
\end{ytableau} \longleftrightarrow
\begin{ytableau}
a & a
\end{ytableau}$$

If $w$ and $v$ are shifted Knuth equivalent words, the diagonally-shaped tableaux $D_w$ and $D_v$ have the same rectification. Thus $D_w$ can be transformed into $D_v$ via some sequence of \textit{jeu de taquin} slides.

\begin{teo}[\cite{Sag87} Theorem 12.2, \cite{Wor84} Theorem 4.4.4]\label{jdtknuth}
Two shifted semistandard tableaux are \textit{jeu de taquin} equivalent if and only if their reading words are shifted Knuth equivalent.
\end{teo}

Therefore, two tableaux in $\mathsf{ShST}(\lambda/\mu,n)$ are said to be shifted Knuth equivalent if so are their reading words. Shifted Knuth equivalence classes and \textit{jeu de taquin} classes on words coincide and are in one-to-one correspondence with shifted semistandard tableaux of straight shape, via rectification (or shifted Schensted insertion \cite{Sag87}). Unlike the classic Knuth relations for unprimed alphabets, the shifted Knuth equivalence is not a congruence, due to rules (S1) and (S2), since $w \equiv_k v$ does not necessarily imply that $tw \equiv_k tw$ for any letter $t \in [n]'$. For instance, $22'1 \equiv_k 221$ but $322'1 \not\equiv_k 3221$. However, under certain conditions we have the following results.

\begin{lema}\label{congresq}
Let $w$ and $v$ be two words in $[n]'$ such that $w \equiv_k v$. Let $t \in [n]'$. Then,
$$wt \equiv_k vt$$
\end{lema}

\begin{proof}
Let $D_w$ and $D_v$ be diagonally-shaped shifted tableaux with words $w$ and $v$, respectively. By Theorem \ref{jdtknuth}, $\mathsf{rect}(D_w) = \mathsf{rect}(D_v)$. Let $T$ be this straight-shaped tableau, of shape $\lambda$. Hence, we may consider the tableau $T^0$ of shape $(\lambda_1+2, \lambda_1, \ldots, \lambda_k)/(\lambda_1+1)$ consisting of $t$ on the entry $(1,\lambda_1+2)$ and $T$ on the remaining part. Clearly, $\mathsf{rect}(T^0)=\mathsf{rect}(D^0_w)=\mathsf{rect}(D^0_v)$ where $D^0_w$ and $D^0_v$ are the diagonally-shaped shifted tableaux  with  words $wt$ and $vt$ respectively.
\end{proof}

\begin{lema}\label{congrdir}
Let $w$ and $v$ be two words in $[n]'$ such that $w \equiv_k v$ and such that there exists a sequence of Knuth relations turning $w$ into $v$ using only $(K1)$ and $(K2)$. Let $t \in [n]'$. Then,
$$tw \equiv_k tv$$
\end{lema}

\begin{proof}
If the rules $(S1)$ and $(S2)$ are not used, then $w$ and $v$ are Knuth equivalent as Young tableau words, considering the standardization to avoid primed entries.
\end{proof}

We define now shifted dual equivalence on words and tableaux. Lemma \ref{standard} ensures that it is compatible with standardization, i.e., that shifted \textit{jeu de taquin} commutes with standardization.

\begin{defin}[\cite{Haim92}]\label{def:dualstd}
Two standard shifted tableaux are \emph{shifted dual equivalent} (or coplactic equivalent) if they have the same shape after applying any sequence (including the empty sequence) of inner or outer \textit{jeu de taquin} slides to both. Two shifted semistandard tableaux are shifted dual equivalent if so are their standardizations.
\end{defin}

In particular, considering the empty sequence of jeu de taquin slides, we have that shifted tableaux that are dual equivalent must have the same shape. In terms of mixed insertion \cite{Haim89}, two shifted semistandard tableaux of the same shape are dual equivalent if and only if they have the same mixed insertion recording tableau \cite{Haim92,Sag87}. This notion is extended to words, with two words being \emph{shifted dual equivalent} if their corresponding diagonally-shaped tableaux are shifted dual equivalent.

The following characterizes dual equivalence on straight-shaped shifted tableaux, in which the dual equivalence classes are determined by the (straight) shapes. Considering mixed-insertion, the recording tableau of a dual class is the recording tableau of the unique shifted Yamanouchi tableau in that class.

\begin{prop}[\cite{Haim92}, Corollary 2.5]\label{haimdual}
Two tableaux of the same straight shape are dual equivalent.
\end{prop}

An operator on shifted tableaux (in canonical form) of the same shape is said to be \emph{coplactic} if it commutes with all sequences of shifted \textit{jeu de taquin} slides.

\subsection{The shifted evacuation and reversal}\label{subsectevacrev}

In this section we recall the \emph{shifted evacuation}, an involution on shifted semistandard tableaux of straight shape that preserves the shape and reverses the weight. This involution was firstly introduced by Worley \cite{Wor84}, as an analogue of the Schützenberger involution \cite{Schu76} on ordinary Young tableaux. Choi, Nam, and Oh \cite{CNO17} recently gave a new formulation using the shifted switching process and showed that the two definitions coincide.

We define a complementation within the alphabet $[n]'$ as:

\begin{align}\label{opcn}
\mathsf{c}_n:&[n]' \to [n]' \nonumber\\
&k \mapsto (n-k+1)'\\
&k' \mapsto n-k+1 \nonumber
\end{align}

This is extended to shifted semistandard tableaux as follows. Given $T \in \mathsf{ShST}(\lambda/\mu,n)$, we define $\mathsf{c}_n (T)$ as the tableau of shape $\mu^{\vee}/\lambda^{\vee}$ obtained by taking each box $(i,j) \in S (\lambda/\mu)$, filled with $\mathbf{k}$, to a box $(\lambda_1 - j +1, \lambda_1 - i +1) \in S(\mu^{\vee}/\lambda^{\vee})$, filled with $\mathsf{c}_n (\mathbf{k})$. In other words, the operator $\mathsf{c}_n$ flips the shape of $T$ across the anti-diagonal of its staircase ambient shape, while complementing the entries using \eqref{opcn}. Then, if $\mathsf{wt}(T) = (wt_1, \ldots, wt_n)$, we have $\mathsf{wt}(\mathsf{c}_n (T)) = \mathsf{wt}(T)^{\mathsf{rev}} := (wt_n, \ldots, wt_1)$. Thus, the operator $\mathsf{c}_n$ is a weight-reversing and shape-``flipping" bijection between $\mathsf{ShST}(\lambda/\mu,n)$ and $\mathsf{ShST}(\mu^{\vee}/\lambda^{\vee},n)$, and consequently an involution in $\mathsf{ShST}(\lambda/\mu,n) \sqcup \mathsf{ShST}(\mu^{\vee}/\lambda^{\vee},n)$.

If $w=w_1 \ldots w_l$ is a word on $[n]'$, then $\mathsf{c}_n (w)$ is defined as the word, in canonical form, of  $\mathsf{c}_n (D_w)$. Then, we have $\mathsf{c}_n (w_1) \ldots \mathsf{c}_n (w_l)$, after canonicalizing. The operator $\mathsf{c}_n$ on words and diagonally-shaped shifted tableaux is a weight-reversing involution.

\begin{ex}
Consider the following shifted semistandard tableau, with (row) reading word $w = 33223'3112'$ and $n=3$, 
$$T= \begin{ytableau}
{} & {} & 1 & 1 & 2'\\
\none & 2 & 2 & 3' & 3\\
\none & \none & 3 & 3
\end{ytableau}$$

We then have $\mathsf{c}_n (w) = 11'22'11'33'2$ and
$$\mathsf{c}_n (T) = \begin{ytableau}
{} & {} & {} & 1' & 2\\
\none & {} & 1' & 1 & 3'\\
\none & \none & 1 & 2' & 3\\
\none & \none & \none & 2
\end{ytableau} $$

\end{ex}

Additionally, we have 
\begin{equation}\label{wcolwrow}
w_{\mathsf{col}}(\mathsf{c}_n (T)) = \mathsf{c}_n (w(T))
\end{equation}
where $w_{\mathsf{col}} (T)$ denotes the \emph{column reading word} of $T$, which is read along columns from bottom to top, going left to right. For diagonally-shaped tableaux, it is clear that the row and column reading words coincide. More generally, for $T$ a shifted semistandard tableau of any shape, we have that $w(T) \equiv_k w_{\mathsf{col}} (T)$ (\cite[Lemma 6.4.12]{Wor84}), and thus

\begin{equation}\label{eq:eqkwcol}
w(\mathsf{c}_n (T)) \equiv_k \mathsf{c}_n (w(T)).
\end{equation}

By construction, the operator $\mathsf{c}_n$ is coplactic. In particular, it preserves shifted Knuth and dual equivalences. Moreover, it commutes with standardization.

\begin{defin}[\cite{Wor84}, Definition 7.1.5]
Let $T$ be a shifted semistandard tableau of straight shape. The \emph{(shifted) evacuation} $\mathsf{evac}$ is defined as $\mathsf{evac}(T) := \mathsf{rect} (\mathsf{c}_n (T))$.
\end{defin}

The coplacity of $\mathsf{c}_n$ ensures that $\mathsf{evac} (T)$ may also be obtained by first ``rectifying" $T$ south-eastward, until obtaining a tableau of shape $\delta/\lambda^{\vee}$, also known as \emph{anti-straight shape}, and then applying $\mathsf{c}_n$.

\begin{ex}\label{exevac1}
Consider the following tableau.
$$T = \begin{ytableau}
1 & 1 & 2' & 2\\
\none & 2 & 2\\
\none & \none & 3
\end{ytableau}$$

Assuming the underlying alphabet to be $\{1,2,3\}'$, to obtain $\mathsf{evac}(T)$ we first compute $\mathsf{c}_n (T)$ and then rectify it.

\begin{align*}
T=\begin{ytableau}
1 & 1 & 2' & 2\\
\none & 2 & 2\\
\none & \none & 3
\end{ytableau} &\xrightarrow{\mathsf{c}_3}
\mathsf{c}_3 (T) = \begin{ytableau}
{} &  & *(gray) & 2'\\
\none & 1 & 2' & 2\\
\none & \none & 2 & 3'\\
\none & \none & \none & 3
\end{ytableau} \rightarrow
\begin{ytableau}
{} & *(gray) & 2' & 2\\
\none & 1 & 2' & 3'\\
\none & \none & 2 & 3
\end{ytableau}
\rightarrow
\begin{ytableau}
*(gray) & 1 & 2' & 2\\
\none & 2 & 2 & 3'\\
\none & \none & 3
\end{ytableau} \rightarrow
\begin{ytableau}
1 & 2' & 2 & 2\\
\none & 2 & 3'\\
\none & \none & 3
\end{ytableau} = \mathsf{evac}(T).
\end{align*}

If the underlying alphabet was $\{1,2,3,4\}'$, then
$\mathsf{evac}(T) =
\begin{ytableau}
2 & 3' & 3 & 3\\
\none & 3 & 4'\\
\none & \none & 4
\end{ytableau}.$

\end{ex}

\begin{prop}[\cite{Wor84}, Lemma 7.1.6]\label{Jshape}
Let $T$ be a shifted semistandard tableau of straight shape. Then, $\mathsf{evac}(T)$ and $T$ have the same shape and $\mathsf{evac}^2(T) = T$.
\end{prop}

Moreover, since rectification leaves the weight unchanged, we have $\mathsf{wt}(\mathsf{evac}(T)) = \mathsf{wt}(T)^\mathsf{rev}$. We remark that $\mathsf{evac}(T)$ is shifted Knuth equivalent to $\mathsf{c}_n (T)$ by construction. And since $\mathsf{evac}(T)$ has the same straight shape as $T$, Proposition \ref{haimdual} ensures that they are shifted dual equivalent. Thus, $\mathsf{evac} (T)$ is the unique tableau that is shifted dual equivalent to $T$ and shifted Knuth equivalent to $\mathsf{c}_n (T)$.

\begin{prop}\label{Yevanunique}
Let $\nu$ be a strict partition. Then, $\mathsf{evac}(Y_{\nu})$ is the unique shifted tableau of shape $\nu$ and weight $\nu^{\mathsf{rev}}$.
\end{prop}

\begin{proof}
The tableau $\mathsf{evac}(Y_{\nu})$ has shape $\nu$ and weight $\nu^{\mathsf{rev}}$. Let $Q$ be another shifted tableau in the same conditions. Then, $\mathsf{evac}(Q)$ has shape and weight equal to $\nu$, thus by Proposition \ref{yamunique}, $\mathsf{evac}(Q) = Y_{\nu}$. Since $\mathsf{evac}$ is an involution, we have $\mathsf{evac}^2(Q) = Q = \mathsf{evac}(Y_{\nu})$.
\end{proof}

The following result provides a direct way to compute the evacuation of $Y_{\nu}$.

\begin{prop}\label{evacyam}
Let $\nu = (\nu_1, \ldots, \nu_n)$ be a strict partition, with $n > 1$. Considering $[n]'$ to be the underlying alphabet, $\mathsf{evac}(Y_{\nu})$ is the tableau of shape $\nu$ such that its $n$-th row is filled with $n^{\nu_n}$, and its $i$-th row is filled with $i^{\nu_n}(i+1)' (i+1)^{\nu_{n \shortminus 1}-\nu_{n}-1} \ldots n' n^{\nu_i - \nu_{i+1}-1}$, reading from left to right, for $i < n$.
\end{prop}

\begin{proof}
This filling clearly defines a shifted semistandard tableau. Let $T_0$ be the tableau in those conditions. By construction, $T_0$ has shape $\nu$ and it is clear that its weight is given by $(\nu_m, \ldots, \nu_1) = \nu^{\mathsf{rev}}$. Hence, by Proposition \ref{Yevanunique}, $T_0 = \mathsf{evac}(Y_{\nu})$.
\end{proof}

\begin{ex}\label{exYevac}
Let $\mu = (4,3,1)$ and $n=3$. Then,
$$Y_{\mu} = \begin{ytableau}
1 & 1 & 1 & 1\\
\none & 2 & 2 & 2\\
\none & \none & 3
\end{ytableau}
\longrightarrow
\mathsf{evac}(Y_{\mu})=
\begin{ytableau}
1 & 2' & 2 & 3'\\
\none & 2 & 3' & 3\\
\none & \none & 3
\end{ytableau}.$$
\end{ex}

\begin{defin}
A word $v$ is said to be \emph{anti-ballot} if there exists a tableau $T \in \mathsf{ShST}(\lambda/\mu,n)$ such that $\mathsf{rect}(T) = \mathsf{evac}(Y_{\nu})$, for $\nu$ a strict partition.
\end{defin}

It is due to Haiman \cite[Theorem 2.13]{Haim92} that, given a shifted semistandard tableau $T$, there is a unique tableau $T^e$, the \emph{reversal} of $T$, that is shifted Knuth equivalent to $\mathsf{c}_n (T)$ and dual equivalent to $T$. Since the operator $\mathsf{c}_n$ preserves shifted Knuth equivalence \cite[Lemma 7.1.4]{Wor84}, the reversal operator is the coplactic extension of evacuation, in the sense that, we may first rectify $T$, then apply the evacuation operator, and then perform outer \textit{jeu de taquin} slides, in the reverse order defined by the previous rectification, to get a tableau $T^e$ with the same shape of $T$. From \cite[Corollary 2.5, 2.8 and 2.9]{Haim92}, this tableau $T^e$ is shifted dual equivalent to $T$, besides being shifted Knuth equivalent to $\mathsf{c}_n (T)$. In particular, $\mathsf{evac}(T) = T^e$ for tableaux of straight shape.

\begin{ex}\label{ex:reversal}
Let $\lambda = (6,5,3,1)$ and $\mu=(4,2)$ and consider the following tableau in the alphabet $\{1,2,3\}'$

$$
T= \begin{ytableau}
{} & {} & {} & {} & 1' & 1\\
\none & {} & {} & {} & 1 & 2'\\
\none & \none & 1 & 2 & 2\\
\none & \none & \none & 3
\end{ytableau}.
$$

To compute the reversal $T^e$, we first rectify $T$, recording the outer corners resulting from the sequence of inner \textit{jeu de taquin} slides:

\begin{align*}
&\begin{ytableau}
{} & {} & {} & {} & 1' & 1\\
\none & {} & {} & *(gray){} & 1 & 2'\\
\none & \none & 1 & 2 & 2\\
\none & \none & \none & 3
\end{ytableau}
\longrightarrow
\begin{ytableau}
{} & {} & {} & {} & 1' & 1\\
\none & {} & *(gray){} & 1 & 2' & {\color{lgray} \bullet_7}\\
\none & \none & 1 & 2 & 2\\
\none & \none & \none & 3
\end{ytableau}
\longrightarrow
\begin{ytableau}
{} & {} & {} & {} & 1' & 1\\
\none & *(gray){}& 1 & 1 & 2' & {\color{lgray} \bullet_7}\\
\none & \none & 2 & 2 & {\color{lgray} \bullet_6}\\
\none & \none & \none & 3
\end{ytableau}
\longrightarrow
\begin{ytableau}
{} & {} & {} & *(gray){} & 1' & 1\\
\none & 1 & 1 & 2' & {\color{lgray} \bullet_5} & {\color{lgray} \bullet_7}\\
\none & \none & 2 & 2 & {\color{lgray} \bullet_6}\\
\none & \none & \none & 3
\end{ytableau}
\longrightarrow\\
&\begin{ytableau}
{} & {} & *(gray){} & 1' & 1 & {\color{lgray} \bullet_4}\\
\none & 1 & 1 & 2' & {\color{lgray} \bullet_5} & {\color{lgray} \bullet_7}\\
\none & \none & 2 & 2 & {\color{lgray} \bullet_6}\\
\none & \none & \none & 3
\end{ytableau}
\longrightarrow
\begin{ytableau}
{} & *(gray){} & 1' & 1 & {\color{lgray} \bullet_3} & {\color{lgray} \bullet_4}\\
\none & 1 & 1 & 2' & {\color{lgray} \bullet_5} & {\color{lgray} \bullet_7}\\
\none & \none & 2 & 2 & {\color{lgray} \bullet_6}\\
\none & \none & \none & 3
\end{ytableau}
\longrightarrow
\begin{ytableau}
*(gray){} & 1' & 1 & 1 & {\color{lgray} \bullet_3} & {\color{lgray} \bullet_4}\\
\none & 1 & 2' & 2 & {\color{lgray} \bullet_5} & {\color{lgray} \bullet_7}\\
\none & \none & 2 & 3 & {\color{lgray} \bullet_6}\\
\none & \none & \none & {\color{lgray} \bullet_2}
\end{ytableau}
\longrightarrow
\begin{ytableau}
1 & 1 & 1 & 1 & {\color{lgray} \bullet_3} & {\color{lgray} \bullet_4}\\
\none & 2 & 2 & 2 & {\color{lgray} \bullet_5} & {\color{lgray} \bullet_7}\\
\none & \none & 3 & {\color{lgray} \bullet_1} & {\color{lgray} \bullet_6}\\
\none & \none & \none & {\color{lgray} \bullet_2}
\end{ytableau}.
\end{align*}

Then, we compute the evacuation of the obtained straight-shaped tableau (see Example \ref{exYevac}):

$$\begin{ytableau}
1 & 1 & 1 & 1 & {\color{lgray} \bullet_3} & {\color{lgray} \bullet_4}\\
\none & 2 & 2 & 2 & {\color{lgray} \bullet_5} & {\color{lgray} \bullet_7}\\
\none & \none & 3 & {\color{lgray} \bullet_1} & {\color{lgray} \bullet_6}\\
\none & \none & \none & {\color{lgray} \bullet_2}
\end{ytableau} \xrightarrow{\mathsf{evac}} 
\begin{ytableau}
1 & 2' & 2 & 3' & {\color{lgray} \bullet_3} & {\color{lgray} \bullet_4}\\
\none & 2 & 3' & 3 & {\color{lgray} \bullet_5} & {\color{lgray} \bullet_7}\\
\none & \none & 3 & {\color{lgray} \bullet_1} & {\color{lgray} \bullet_6}\\
\none & \none & \none & {\color{lgray} \bullet_2}
\end{ytableau}.$$

Finally, we perform outer \textit{jeu de taquin} slides defined by the outer corners of the previous sequence:

\begin{align*}
\begin{ytableau}
1 & 2' & 2 & 3' & {\color{lgray} \bullet_3} & {\color{lgray} \bullet_4}\\
\none & 2 & 3' & 3 & {\color{lgray} \bullet_5} & {\color{lgray} \bullet_7}\\
\none & \none & 3 & {\color{lgray} \bullet_1} & {\color{lgray} \bullet_6}\\
\none & \none & \none & {\color{lgray} \bullet_2}
\end{ytableau}
&\longrightarrow
\begin{ytableau}
{} & 1 & 2' & 3' & {\color{lgray} \bullet_3} & {\color{lgray} \bullet_4}\\
\none & 2 & 2 & 3' & {\color{lgray} \bullet_5} & {\color{lgray} \bullet_7}\\
\none & \none & 3 & 3 & {\color{lgray} \bullet_6}\\
\none & \none & \none & {\color{lgray} \bullet_2}
\end{ytableau}
\longrightarrow
\begin{ytableau}
{} & {} & 2' & 3' & {\color{lgray} \bullet_3} & {\color{lgray} \bullet_4}\\
\none & 1 & 2' & 3' & {\color{lgray} \bullet_5} & {\color{lgray} \bullet_7}\\
\none & \none & 2 & 3' & {\color{lgray} \bullet_6}\\
\none & \none & \none & 3
\end{ytableau}
\longrightarrow
\begin{ytableau}
{} & {} & {} & {} & 2' & 3'\\
\none & 1 & 2' & 3' & {\color{lgray} \bullet_5} & {\color{lgray} \bullet_7}\\
\none & \none & 2 & 3' & {\color{lgray} \bullet_6}\\
\none & \none & \none & 3
\end{ytableau}\\
&\longrightarrow
\begin{ytableau}
{} & {} & {} & {} & 2' & 3'\\
\none & {} & 1 & 2' & 3' & {\color{lgray} \bullet_7}\\
\none & \none & 2 & 3' & {\color{lgray} \bullet_6}\\
\none & \none & \none & 3
\end{ytableau}
\longrightarrow
\begin{ytableau}
{} & {} & {} & {} & 2' & 3'\\
\none & {} & {} & 2' & 3' & {\color{lgray} \bullet_7}\\
\none & \none & 1 & 2 & 3'\\
\none & \none & \none & 3
\end{ytableau}
\longrightarrow
\begin{ytableau}
{} & {} & {} & {} & 2' & 3'\\
\none & {} & {} & {} & 2' & 3'\\
\none & \none & 1 & 2 & 3'\\
\none & \none & \none & 3
\end{ytableau} = T^e.
\end{align*}
\end{ex}

By construction, the reversal is an involution on $\mathsf{ShST}(\lambda/\mu,n)$ which permutes the tableaux within
each shifted dual class while reversing the weight. It coincides with evacuation on straight-shaped tableaux. In
particular, since each dual class has a unique LRS tableau, it yields a bijection $T \longmapsto \mathsf{c}_n (T^e)$ between the set of LRS tableaux of shape $\lambda/\mu$ and weight $\nu$ and the set of LRS tableaux of shape $\mu^{\vee}/\lambda^{\vee}$ and weight $\nu$. Hence, we have the symmetry $f_{\mu\nu}^{\lambda} = f_{\lambda^{\vee}\nu}^{\mu^{\vee}}$. The symmetry $f_{\mu\nu}^{\lambda} = f_{\nu\mu}^{\lambda}$, as well as the others resulting from these two, may be obtained using the \emph{shifted tableau switching} (see \cite{CNO17}).

Henceforth, we will denote by $\eta$ either the involution $\mathsf{evac}$ or its coplactic extension $e$, and we will refer to it by the \emph{Schützenberger involution}\footnote{We remark that our formulation of $\eta$, which coincides with the one in \cite[Remark 5.7]{GL19}, consists of $\mathsf{c}_n$ followed by some \textit{jeu de taquin} slides, to preserve shape. This differs from the $\eta$ presented in \cite{GLP17}, which corresponds to $\mathsf{c}_n$ here.}. Later, in Section \ref{sec4}, the involution $\eta$ will be formulated in the language of a shifted tableau crystal, and it will become clear how the tableaux are permuted within each
shifted dual class.

\section{A crystal-like structure on shifted tableaux}\label{sec3}
	
We recall the main results on the shifted tableau crystal $\mathcal{B}(\lambda/\mu,n)$, the crystal-like structure on $\mathsf{ShST}(\lambda/\mu,n)$ introduced in \cite{GL19, GLP17}. Let $\{e_1, \ldots,e_n\}$ be the canonical basis of $\mathbb{R}^n$, and let $\alpha_i=e_i-e_{i+1}$ be the weight vectors, for $i \in I := [n-1]$. We first recall the primed raising and lowering operators on words.

\begin{defin}[\cite{GLP17}, Definition 3.3] 
Given a word $w$ on $[n]'$ and  $i \in I$, the \emph{primed raising operator} $E_i'(w)$ is defined as the unique word such that
	\begin{enumerate}
	\item $\mathsf{std}(E_i' (w)) = \mathsf{std}(w)$,
	\item $\mathsf{wt}(E_i'(w)) = \mathsf{wt}(w) + \alpha_i$,
	\end{enumerate}
if such word exists. Otherwise, $E_i'(w) = \emptyset$, and we say that $E_i'$ is undefined on $w$. The \emph{primed lowering operator} $F_i'(w)$ is defined in analogous way using $-\alpha_i$.
\end{defin}
This notion is well defined due to Lemma \ref{standard}, and as a direct consequence we have that $E_i'(w) = v$ if and only if $w=F_i' (v)$, for any words $w$ and $v$ \cite[Proposition 3.4]{GLP17}. This definition is extended to a shifted semistandard tableau $T$, putting $E_i' (T)$ as the shifted semistandard tableau with the same shape as $T$ and with (row) reading word $E_i' (w(T))$. The primed operators preserve semistandardness \cite[Proposition 3.6]{GLP17} and they are coplactic \cite[Proposition 3.7]{GLP17}. Moreover, the tableaux $T$, $E_i' (T)$ and $F_i' (T)$ are dual equivalent, since their standardization is unchanged (Definition \ref{def:dualstd}), whenever $E_i'$ and $E_i$ are not undefined on $T$.

In order to simplify the notation, from now on we consider the alphabet $\{1,2\}'$, but the results hold for any primed alphabet $\{i,i+1\}'$ of two adjacent letters. The following propositions provide a simple way to compute the primed operators both on words and on shifted tableaux of straight shape.

\begin{prop}[\cite{GLP17}, Proposition 3.9]\label{computEF}
To compute $F_1' (w)$ consider all representatives of $w$. If all representatives have
the property that the last $1$ is left of the last $2'$ then $F_1' (w) = \emptyset$. If there exists a representative such that the last $1$ is right of the last $2'$ then $F_1'(w)$ is obtained by changing the last $1$ to $2'$ in that representative. The word $E_1'(w)$ is defined similarly reverting the roles of $1$ and $2'$.
\end{prop}

\begin{prop}[\cite{GLP17}, Proposition 3.11]
Let $T\in \mathsf{ShST}(\lambda,n)$  a shifted semistandard tableau of straight shape. If $T$ has one row, then $E_1'(T)$ (respectively $F_1'(T)$) is obtained by changing the leftmost $2$ to $1$ (respectively, $1$ to $2$), if possible, and it is $\emptyset$ otherwise. If $T$ has two rows and the first row contains a $2'$, then $E'_1(T)$ is obtained by changing that $2'$ to $1$ and $F_1'(T) = \emptyset$. If the first row does not contain a $2'$, then $E_1'(T) = \emptyset$ and $F_1'(T)$ is obtained by changing the rightmost $1$ to $2'$.
\end{prop}

\begin{ex}
Let $T = \begin{ytableau}
1&1&1&2'&2\\
\none & 2&2&2\\
\none & \none &3
\end{ytableau}$. Then, $F'_2(T) = \begin{ytableau}
1&1&1&2'&3'\\
\none & 2&2&2\\
\none & \none &3
\end{ytableau}.$

Observe that both tableaux have the same standardization and that $\mathsf{wt}(F_2'(T)) = (3,4,2) = \mathsf{wt}(T)-(0,1,-1)$.
\end{ex}

Given a word $w$ on the alphabet $[n]'$ and $i \in I$, the $i$-th \emph{lattice walk} of $w$ is obtained by considering the subword $w_i$, consisting of the letters $\{i,i+1\}'$, and replacing each letter according to the following table, starting at the origin $(0,0)$.

\begin{center}
\begin{tabular}{|c|cccc|}
\hline
$x_k y_k = 0$ & $\xrightarrow{\text{1'}}$ & $\xrightarrow{\text{1}}$ & $\uparrow 2'$ & $\uparrow 2$ \\
\hline
$x_k y_k \neq 0$ & $\xrightarrow{\text{1'}}$ & $\downarrow 1$ & $\xleftarrow{\text{2'}}$ & $\uparrow 2$ \\
\hline
\end{tabular}
\end{center}

The lattice walks of a word may be used to provide another criterion for balotness. 

\begin{prop}[\cite{GLP17}, Corollary 4.5]\label{cor4.5}
A word $w$ is ballot if and only if the $i$-th lattice walk of the subword $w_i$ consisting of $\mathbf{i}$ and $\mathbf{i+1}$ ends on the $x$ axis, for all $i\in I$.
\end{prop}

\begin{ex}\label{exlatwalk}
Let $w=3211221'11$. To obtain the $1$st and $2$nd lattice walks of $w$, consider the subwords $w_1 = 211221'11$  and $w_2 = 3222$ (which corresponds to $2111$, using the alphabet $\{1,2\}'$). Replacing each letter accordingly, we obtain
	
	\begin{center}
	\includegraphics[scale=1]{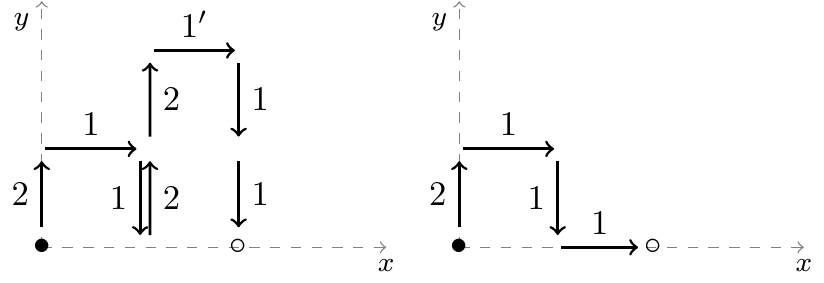}
	\end{center}
Since both lattice walks end on the $x$ axis, the word $w$ is ballot.
\end{ex}

If $w$ is a word on the alphabet $\{1', 1, 2',2\}$ and $u= w_k w_{k+1}$ $ \ldots w_l$ is a substring of some representative of $w$, then $u$ is called a \emph{substring} of $w$. The coordinates $(x,y)$ of the point of the 1-lattice walk before the start of $u$ is called the \emph{location} of $u$.

\begin{defin}[\cite{GLP17}, Definition 5.3]
A substring $u$ is said to be a \emph{$F_1$-critical substring} if any of these conditions on $u$ and its location are satisfied (as well as the adequate transformations to apply), where $ab^* c$ means any string of the form $ab \cdots bc$, including $ac$:

\begin{center}
\begin{table}[h]
\begin{tabular}{|c|ccc|c|}
\hline
Type & Substring &  Condition steps & Location & Transformation \\
\hline
\multirow{2}{*}{1F} & \multirow{2}{*}{$u=1(1')^* 2'$} & $\rightarrow \rightarrow \uparrow$ & $y=0$ & \multirow{2}{*}{$u\rightarrow 2'(1')^* 2$} \\

  &   & $\downarrow\rightarrow\uparrow$ & $y=1, x\geq 1$ &   \\
\hline
\multirow{2}{*}{2F} & \multirow{2}{*}{$u=1(2)^* 1'$} & $\rightarrow\uparrow\rightarrow$ & $x=0$ & \multirow{2}{*}{$u \rightarrow 2'(1')^* 1$} \\

  &   & $\downarrow\uparrow\rightarrow$ & $x=1, y \geq 1$ &  \\
\hline
3F & $u=1$ & $\rightarrow$ & $y=0$ & $u \rightarrow 2$ \\
\hline
4F & $u=1'$ & $\rightarrow$ & $x=0$ & $u \rightarrow 2'$ \\
\hline
\multirow{2}{*}{5F} & $u=1$ & $\downarrow$ & \multirow{2}{*}{$x=1, y\geq 1$} & \multirow{2}{*}{$\emptyset$} \\

  & $u=2'$  & $\leftarrow$ &  &   \\
\hline
\end{tabular}
\end{table}
\end{center}
\end{defin}

The \emph{final} $F_1$-critical substring $u$ of the word $w$ is the $F_1$-critical substring $u$ with the highest starting index, taking the longest in the case of a tie. If there is still a tie (due to different representatives), take any such $u$. The unprimed raising and lowering operators may now be recalled.

\begin{defin}[\cite{GLP17}, Definition 5.4]
Let $w$ be a word. The word $F_1(w)$ is obtained by taking a representative of $w$ containing a final $F_1$-critical substring and transforming it according to the previous table. If there is no $F_1$-critical substring or it is of type 5F, then put $F_1 (w) = \emptyset$, and in this case, $F_1$ is said to be undefined on $w$.
\end{defin}

\begin{lema}[\cite{GLP17}, Proposition 5.14 (i)]\label{lemaxindef}
Let $w$ be a word on the alphabet $\{i',i,(i+1)',(i+1)\}$ and let $(x,y)$ be the endpoint of the $i$-lattice walk of $w$. If $x=0$, then $F_i (w) = \emptyset$.
\end{lema}

The operators $F_i$ are called the \emph{unprimed raising operators}. The \emph{unprimed lowering operators} $E_i$ are defined on words by $E_i (w) := \mathsf{c}_n F_{n-i} \mathsf{c}_n (w) $, for $i \in I$. In particular, for $n=2$, we have $E_1 (w) = \mathsf{c}_2 F_1 (\mathsf{c}_2 (w))$, thus $E_1$ may be obtained in similar way using the following table of $E_1$-critical substrings:

\begin{center}
\begin{table}[h]
\begin{tabular}{|c|ccc|c|}
\hline
Type & Substring &  Condition steps & Location & Transformation \\
\hline
\multirow{2}{*}{1E} & \multirow{2}{*}{$u=2'(2)^*1$} & $\uparrow \uparrow \rightarrow$ & $x=0$ & \multirow{2}{*}{$u\rightarrow 1(2)^*1'$} \\

  &   & $\leftarrow \uparrow \rightarrow$ & $x=1, y\geq 1$ &   \\
\hline
\multirow{2}{*}{2E} & \multirow{2}{*}{$u=2'(1')^*2$} & $\uparrow \rightarrow \uparrow$ & $y=0$ & \multirow{2}{*}{$u \rightarrow 1(2)^*2'$} \\

  &   & $\leftarrow \rightarrow \uparrow$ & $y=1, x \geq 1$ &  \\
\hline
3E & $u=2'$ & $\uparrow$ & $x=0$ & $u \rightarrow 1'$ \\
\hline
4E & $u=2$ & $\uparrow$ & $y=0$ & $u \rightarrow 1$ \\
\hline
\multirow{2}{*}{5E} & $u=1$ & $\downarrow$ & \multirow{2}{*}{$y=1, x\geq 1$} & \multirow{2}{*}{$\emptyset$} \\

  & $u=2'$  & $\leftarrow$ &  &   \\
\hline
\end{tabular}
\end{table}
\end{center}

These definitions can be extended to tableaux: given $T \in \mathsf{ShST} (\lambda/\mu,n)$, $F_i (T)$ is the tableau with the same shape as $T$ with (row) reading word $F_i (w(T))$, for $i \in I$. The definition of $E_i (T)$ is analogous. In both definitions, the row reading word of $T$ may be replaced by the column reading word \cite[Proposition 5.21]{GLP17}. This notions are well defined, since $E_i (T)$ and $F_i (T)$ are shifted semistandard tableaux, for all $i \in I$, whenever these operators are not undefined on $T$ \cite[Theorem 5.18]{GLP17}. Moreover, the primed and uprimed operators $E_i'$, $E_i$, $F_i'$ and $F_i$ commute whenever the compositions among them are defined, for each $i \in I$ \cite[Proposition 5.36]{GLP17}.

\begin{ex}
Let $T= \begin{ytableau}
1 & 1 & 1 & 1 & 2'\\
\none & 2 & 3' & 3\\
\none & \none & 3
\end{ytableau}$. To compute $E_2$ and $F_2$, we consider the subword of $w(T)$ in the alphabet $\{2,3\}'$, which is $323'32'$, with the following lattice walk (in the alphabet $\{1,2\}'$):

\begin{center}
	\begin{tikzpicture}
	\draw[gray,dashed,->] (0,0) -- (0,2.5);\draw[gray,dashed,->] (0,0) -- (2,0);
	\node[below left] at (0,2.5) {{\footnotesize $y$}};\node[below] at (2,0) {{\footnotesize $x$}};
	
	\node at (0,0) (1) {$\bullet$};
	\node at (0,0.9) (2) {};
	\node at (1,0.9) (3) {};
	\node at (1,1.1) (4) {};
	\node at (0,1.1) (5) {};
	\node at (0,2) (6) {};
	\node at (1,2) (7) {$\circ$};

	\draw[->] (1) to node[left] {$2$} (5);
	\draw[->] (2) to node[below] {$1$} (3);
	\draw[->] (4) to node[above] {$2'$} (5);
	\draw[->] (5) to node[left] {$2$} (6);
	\draw[->] (6) to node[above] {$1'$} (7);
	\end{tikzpicture}	
\end{center}

Since $T$ has a final $E_2$-critical substring of type 2E, and a final $F_2$-critical substring of type 4F. Thus, applying the correct substitution we obtain
$$E_2 (T) = \begin{ytableau}
1 & 1 & 1 & 1 & 2'\\
\none & 2 & 2 & 3'\\
\none & \none & 3
\end{ytableau} \qquad
F_2 (T) = \begin{ytableau}
1 & 1 & 1 & 1 & 3'\\
\none & 2 & 3' & 3\\
\none & \none & 3
\end{ytableau}. $$
\end{ex}

The primed and unprimed operators may be used to give an alternative formulation for ballot (and anti-ballot) words. Indeed, a word $w$ on the alphabet $\{i,i+1\}'$ is ballot (respectively anti-ballot) if and only if $E_i (w) = E_i' (w) = \emptyset$ (respectively $F_i (w) = F_i' (w) = \emptyset$) \cite[Proposition 5.17]{GLP17}. Hence, a word on the alphabet $[n]'$ is ballot (respectively anti-ballot) if and only if $E_{i}(w) = E_i' (w)=\emptyset$ (respectively $E_{i}(w) = E_i' (w)=\emptyset$) for all $i \in I$.

Whenever they are defined, $E_i(T)$ and $F_i(T)$ are dual equivalent to $T$ \cite[Corollary 5.3]{GLP17}. Since this is also true for the primed operators, as the standardization is unchanged, then any two tableaux that differ by a sequence of any lowering or raising operators are dual equivalent. Moreover, the unprimed operators are coplactic, whenever defined \cite[Theorem 5.35]{GLP17}.

An \emph{highest weight element} (respectively \emph{lowest weight element}) of $\mathsf{ShST}(\lambda/\mu,n)$ is a tableau $T$ such that $E_i(T) = E_i' (T) = \emptyset$ (respectively $F_i (T) = F_i' (T) = \emptyset$), for any $i \in I$. This means that the reading word of $T$ is ballot (respectively anti-ballot).

\begin{prop}[\cite{GLP17}, Proposition 6.4]\label{highuni}
Let $\nu$ be a strict partition. We have that $Y_{\nu}$ is the unique $T\in \mathsf{ShST}(\nu,n)$ for which $E_{i}(T) = E_i' (T)=\emptyset$, for all $i \in I$. Then, every $T\in \mathsf{ShST}(\nu,n)$ may be obtained from every other by a sequence of primed and unprimed lowering and raising operators.
\end{prop}

The set $\mathsf{ShST}(\lambda/\mu,n)$ is closed under the operators $E_i, E_i', F_i, F_i'$, for  $i \in I$.
Moreover, we also have the \emph{partial length functions} \cite{GL19} given by:
\begin{align*}
\varepsilon_i'(T) &:= max\{k:\, E_i'^k(T) \neq \emptyset\}
&\widehat{\varepsilon_i} (T) := max\{k:\, E_i^k(T) \neq \emptyset\}\\
\varphi_i'(T) &:= max\{k:\, F_i'^k(T) \neq \emptyset\}
&\widehat{\varphi_i} (T) := max\{k:\, F_i^k(T) \neq \emptyset\},
\end{align*}
and \emph{total length functions} $\varepsilon_i (T)$ and $\varphi_i(T)$, defined as the $y$-coordinate and $x$-coordinate, respectively, of the endpoints of the $i$-th lattice walk of $T$, for $i \in I$ \cite[Section 5.1]{GLP17}.

\begin{figure}[h!]
\begin{center}
\includegraphics[scale=0.5]{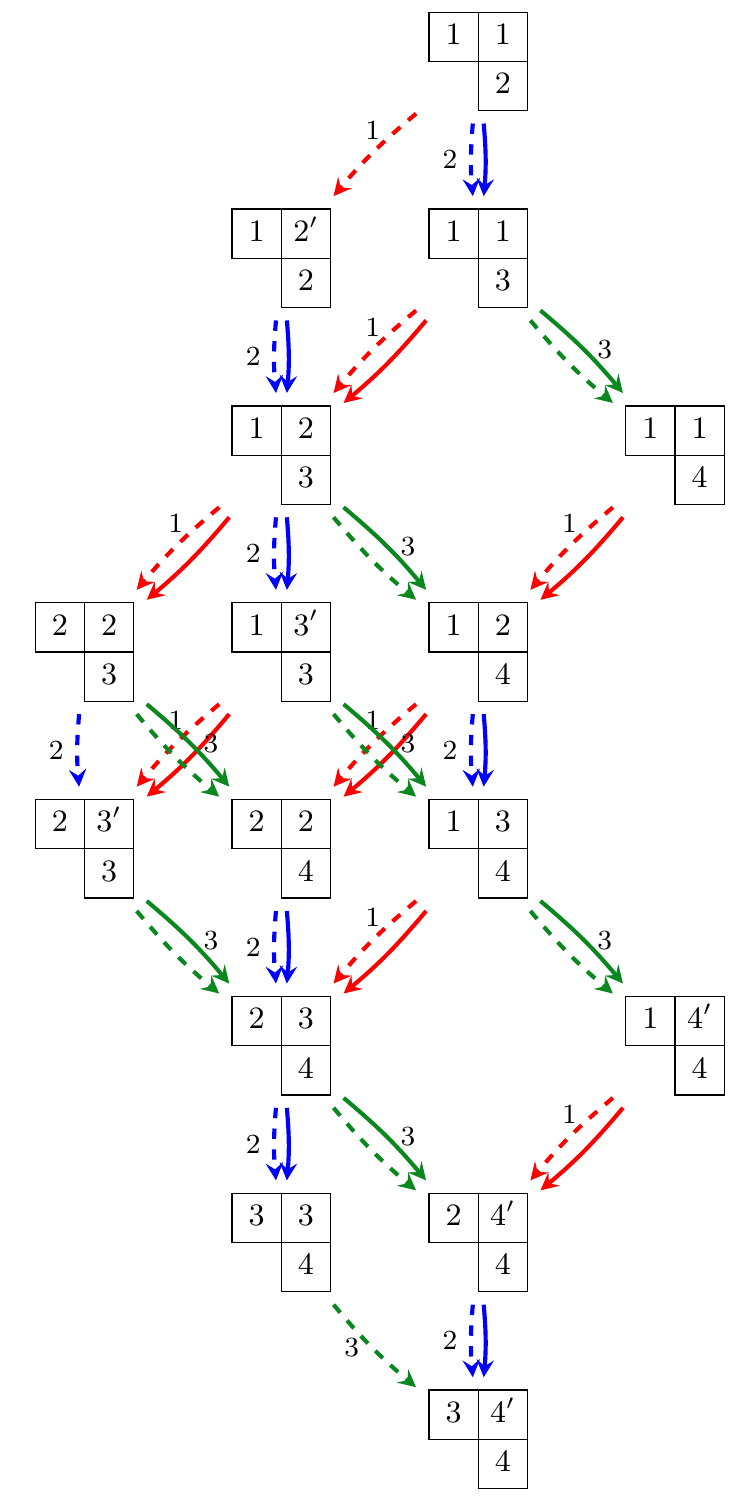}
\hspace{1cm}
\vrule{}
\hspace{1cm}
\includegraphics[scale=0.5]{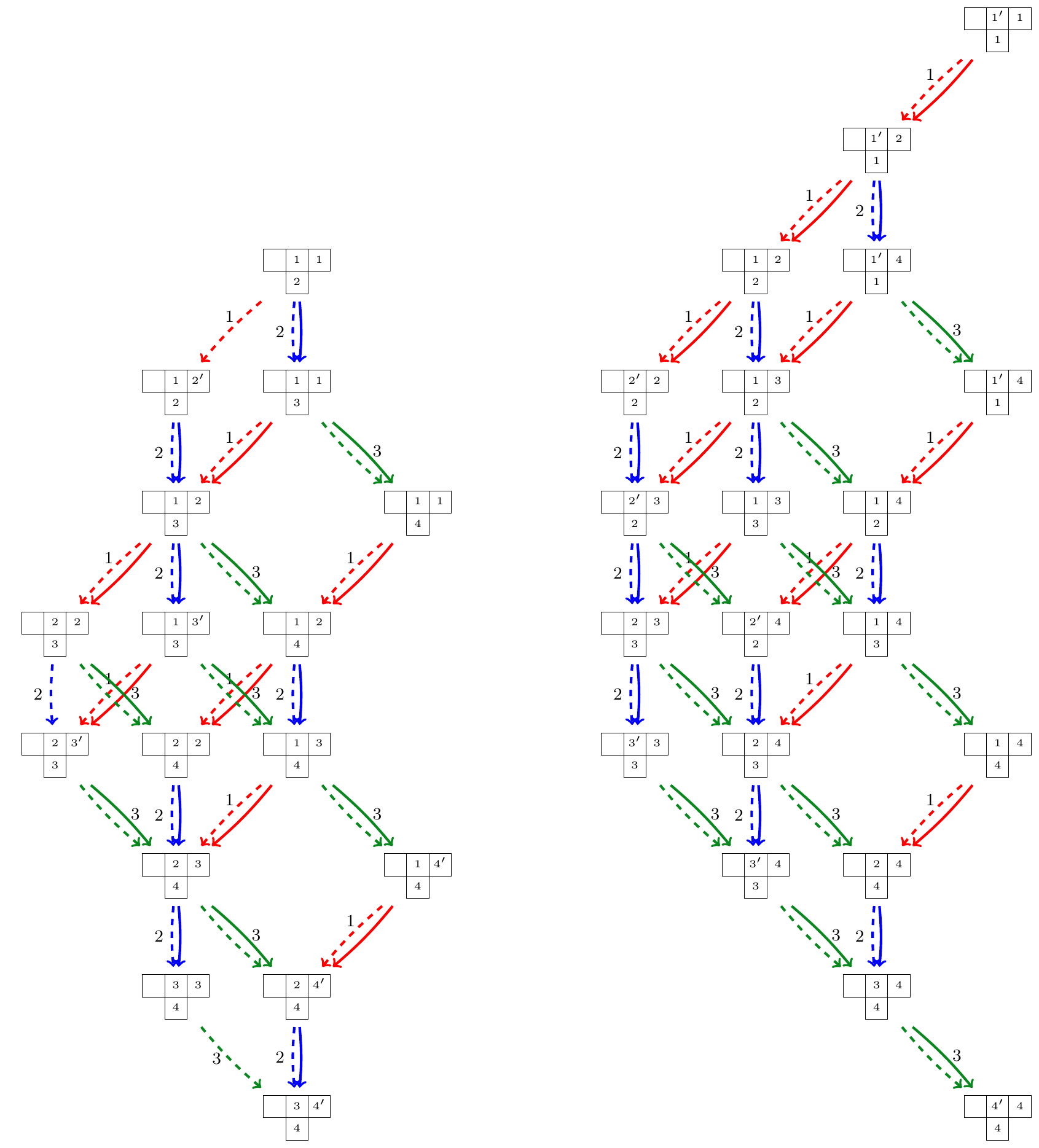}
\end{center}
\caption{On the left, the shifted tableau crystal graph $\mathcal{B}(\nu,4)$, for $\nu=(2,1)$. On the right, the shifted tableau crystal graph $\mathcal{B}(\lambda/\mu,4)$, for $\lambda=(3,1)$ and $\mu = (1)$. The operators $F_1,F_1'$ are in red, the $F_2,F_2'$ in blue, and $F_3,F_3'$ in green. Note that $\mathcal{B} (\lambda/\mu,4)$ has two connected components, one of them being isomorphic to $\mathcal{B}(\nu,4)$.}
\label{fig:crystal}
\end{figure}

The set $\mathsf{ShST}(\lambda/\mu,n)$, together with primed and unprimed operators, partial and total length functions, and weight function, is called a \emph{shifted tableau crystal} and denoted by $\mathcal{B}(\lambda/\mu,n)$. It may be regarded as a directed acyclic graph with weighted vertices, and $i$-coloured labelled double edges, the solid ones being labelled with $i$ ($x \overset{i}{\rightarrow} y$ if $F_i(x)=y$) and the dashed ones with $i'$ ($x\overset{i'}{\dashrightarrow} y$ if $F'_i(x)=y$), for $i\in I$. Examples are shown in Figures \ref{fig:crystal}, and in Figures \ref{fig:crystal521} and \ref{fig:crystal531} of Appendix \ref{appendix}. The connected components of $\mathcal{B}(\lambda/\mu,n)$ are the connected components of the underlying undirected and non-labelled graph. We also remark that the set $\mathsf{ShST}(\lambda/\mu,n)$ together with only the primed (or the unprimed) operators and with the same auxiliary functions, $\mathsf{wt}$, and total length functions $\varphi_i$, $\varepsilon_i$, satisfies the axioms of a Kashiwara crystal in type $A$ \cite[Proposition 6.9]{GLP17}.

\begin{prop}[\cite{GLP17}, Corollary 6.5]\label{isomhigh}
Each connected component of $\mathcal{B}(\lambda/\mu,n)$ has a unique highest weight element $T^{\mathsf{high}}$, which is a LRS tableau, and is isomorphic, as a weighted edge-labelled graph, to the shifted tableau crystal $\mathcal{B}(\nu,n)$, where $\nu = \mathsf{wt}(T^{\mathsf{high}})$.  Then, every $T$ in a connected component of $ \mathsf{ShST}(\lambda/\mu,n)$ may be obtained from every other by a sequence of primed and unprimed lowering and raising operators.
\end{prop}

\begin{prop}[\cite{GLP17}, Corollary 6.6]\label{compdual}
Each connected component of $\mathcal{B}(\lambda/\mu,n)$ forms a shifted dual equivalence class.
\end{prop}

Therefore, by decomposing $\mathcal{B}(\lambda/\mu,n)$ into connected components, we have the crystal isomorphism 
$$\mathcal{B}(\lambda/\mu,n)\simeq\bigsqcup_\nu \mathcal{B}(\nu,n)^{f_{\mu,\nu}^\lambda}$$ 
where $f_{\mu,\nu}^{\lambda}$ is the shifted Littlewood-Richardson coefficient, which yields the well known decomposition of skew Schur $Q$-function $Q_{\lambda/\mu} = \sum\limits_{\nu} f_{\mu,\nu}^{\lambda} Q_{\nu}$ (for details, see \cite[Section 7]{GLP17}).

\subsection{Decomposition into strings}

Erasing all arrows of colours $j, j'\notin \{i,i'\}$ in $\mathcal{B}(\lambda/\mu,n)$, for each $i \in I$, one obtains the $\{i',i\}$-connected components of $\mathcal{B}(\lambda/\mu,n)$. This may also be done by defining an equivalence relation on $\mathcal{B}(\lambda/\mu,n)$ as a set, in which $x \sim y$ if $x$ and $y$  are related  by a sequence of $F_i$,  $F'_i$, $E_i$ or $E'_i$. The equivalence classes are called the $i$-\emph{strings}, which are the underlying subsets of the $\{i',i\}$-connected components.

Hence, $\mathcal{B}(\lambda/\mu,n)$ may be partitioned, as a set, into $i$-\emph{strings} (see Figure \ref{figcrystalpart}). The $i$-strings, as crystal graphs, are the $\{i',i\}$-connected components of $\mathcal{B}(\lambda/\mu,n)$, and have  two possible arrangements \cite[Section 3.1]{GL19} \cite[Section 8]{GLP17}, as shown in Figure \ref{figstringsepcol}. A string consisting of two $i$-labelled chains of equal lenght, connected by $i'$-labelled edges is called a \emph{separated $i$-string}. The smallest separated string is formed by two vertices connected by a $i'$-labelled edge. A string formed by a double chain of both $i$- and $i'$-labelled strings is called a \emph{collapsed $i$-string}. A single vertex (without edges) is the smallest collapsed string.

\begin{figure}[H]
\begin{center}
\includegraphics[scale=0.9]{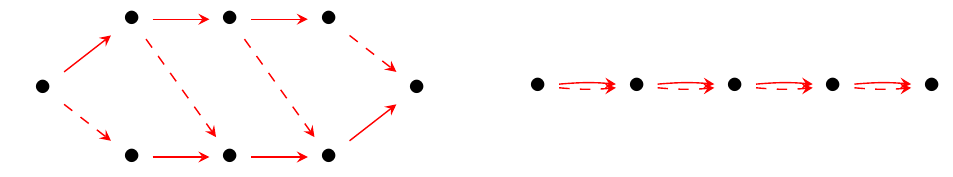}
\end{center}
\caption{A separated $i$-string (left) and a collapsed $i$-string (right).}
\label{figstringsepcol}
\end{figure}

\begin{figure}[h]
\begin{center}
\includegraphics[scale=0.6]{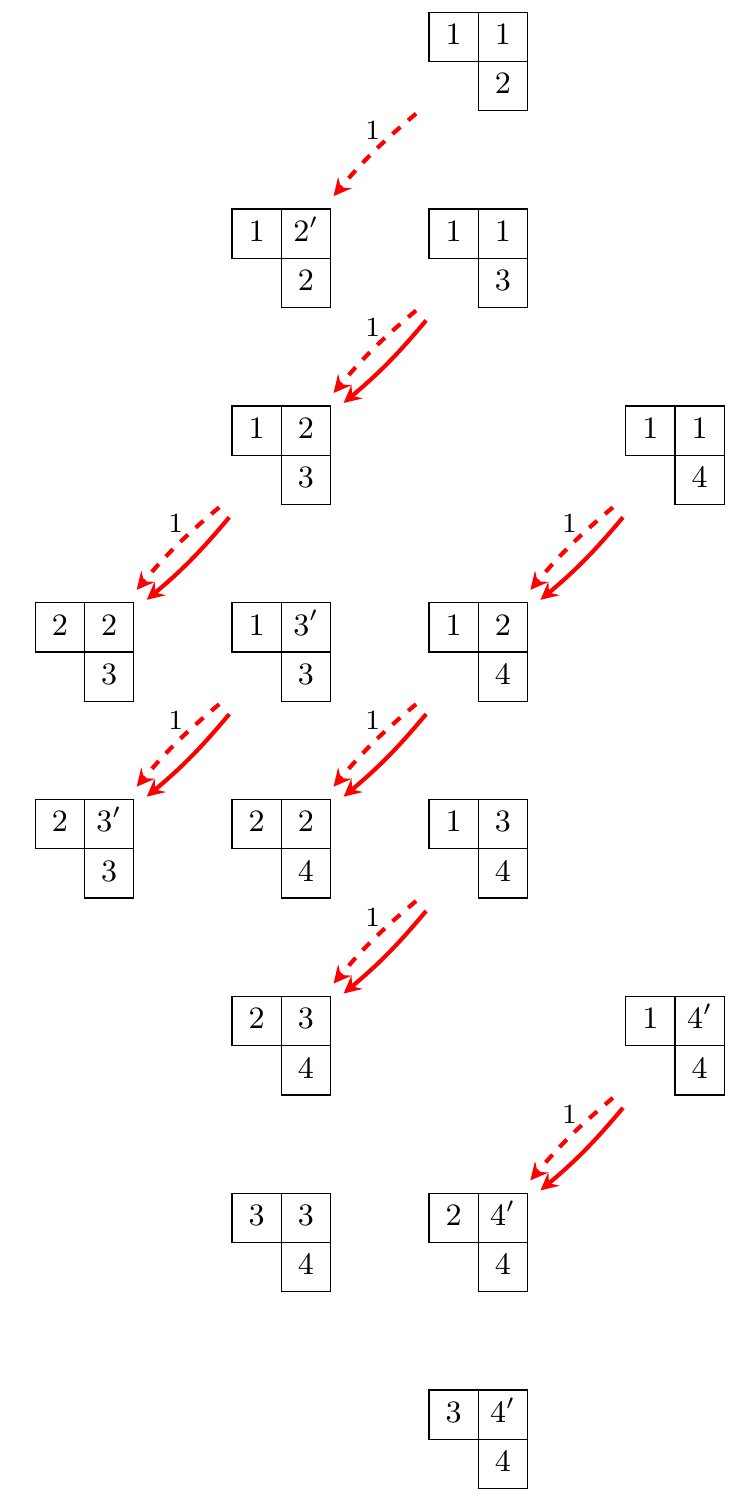}\qquad
\includegraphics[scale=0.6]{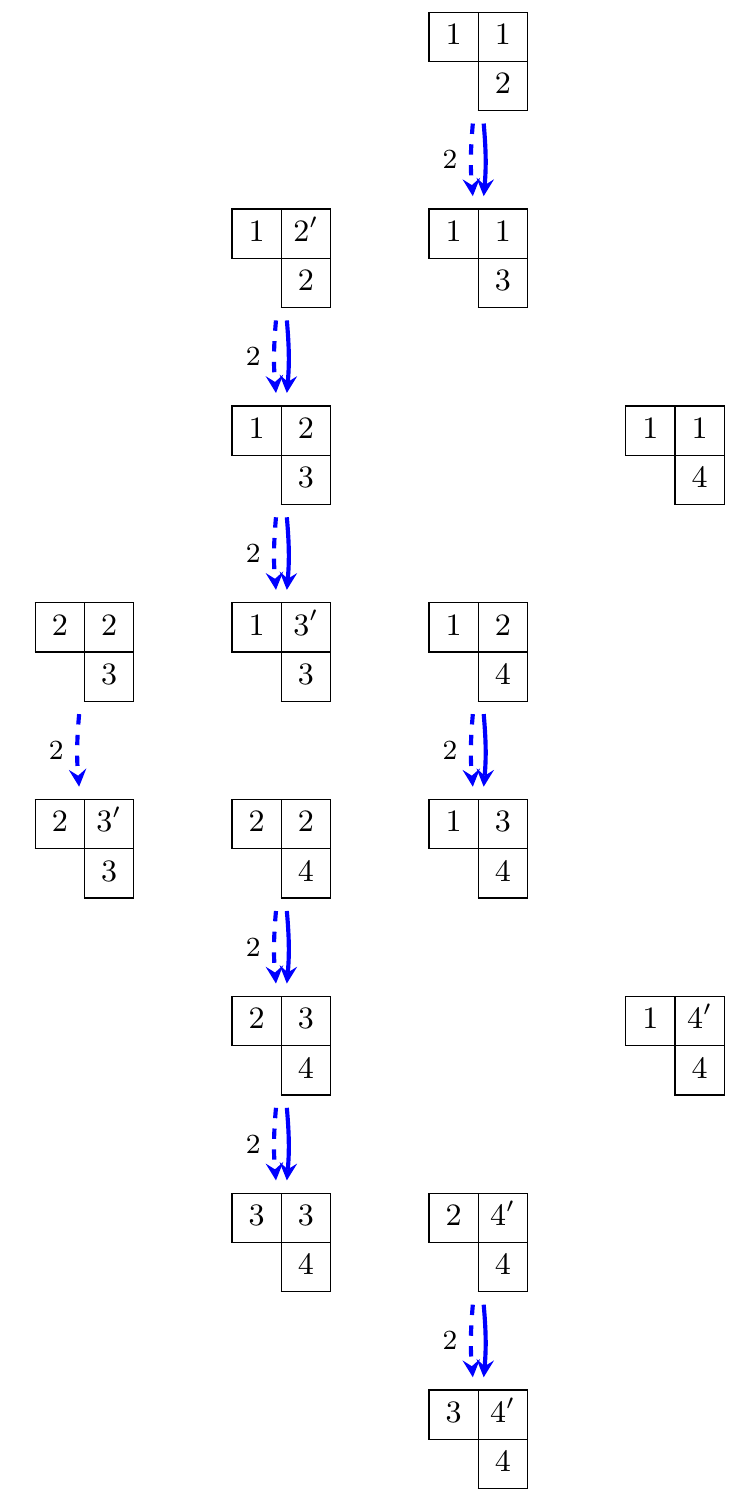}\qquad
\includegraphics[scale=0.6]{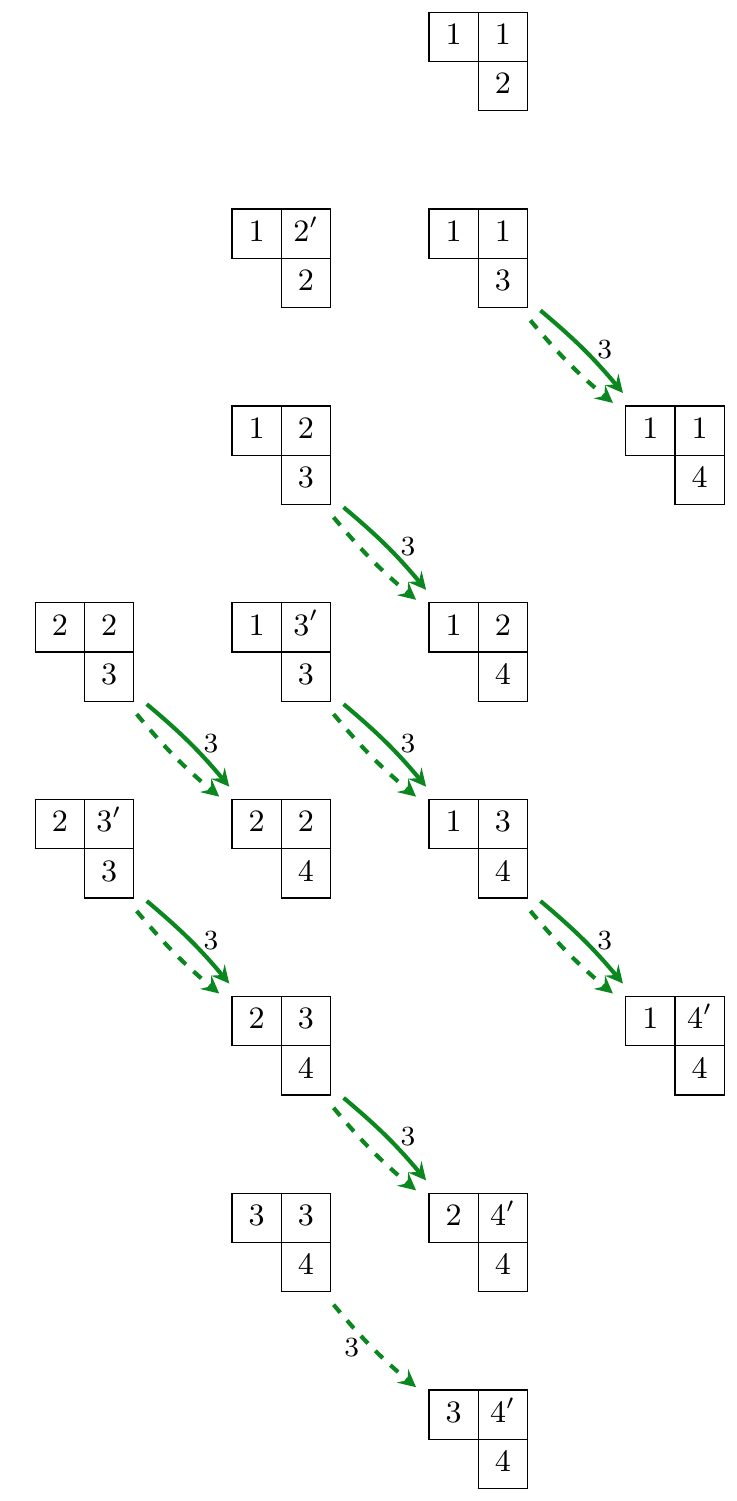}
\end{center}
\caption{The crystal graph $\mathcal{B}(\nu,4)$, for $\nu=(2,1)$, partitioned into 1-strings (left), 2-strings (middle) and 3-strings (right).}
\label{figcrystalpart}
\end{figure}

The following propositions are intended to detail the possible arrangements for an $i$-string. This corresponds to the details of the axiom (B1) in \cite{GL19}. The next result provides a condition for a $i$-string to be separated or collapsed in terms of its highest weight element. The highest weight element of an $i$-string is the unique tableau $T$ in that string such that $E_i (T) = E_i' (T) = \emptyset$, and the lowest weight element is defined similarly.
 
\begin{prop}\label{axb1}
Let $T \in \mathcal{B}(\lambda/\mu,n)$ and $i \in I$. Suppose that $T$ is the highest weight element of its $i$-string. Then, $F_i(T) = F_i'(T)$ if and only if $\mathsf{wt}(T)_{i+1} = 0$. In this case, the $i$-string containing $T$ is collapsed.
\end{prop}

\begin{proof}
We consider the alphabet $\{1,2\}'$, to simplify notation. Suppose $\mathsf{wt}(T)_{2}=0$. If $\mathsf{wt}(T)_1=0$, then $F_1(T)=F_1'(T)=\emptyset$ and the 1-string is a trivial collapsed string. Thus, we may assume, without loss of generality, that $\mathsf{wt}(T)_1 > 0$. Since $\mathsf{wt}(T)_{2} = 0$, there are no occurrences of $2'$ to the right of the last $1$. Therefore, $F'_1(T)$ is defined and obtained by changing the last $1$ into $2'$. Moreover, since $\mathsf{wt}(F'_1(T))_2=1$, this $2'$ is identified with $2$ in the canonical form.
We claim that $F_1 (T)$ is defined, i.e., there exists a final $F_1$-critical substring that is not of type $5F$. Since $\mathsf{wt}(T)_2=0$, the location of a possible substring is $y=0$, excluding the type 5F, and since $\mathsf{wt}(T)_1>1$, there is necessarily a substring of type 2F (with $x=0$) or 3F (4F would be the case where either $w(T) = 1'$, which is equal to $1$ in canonical form, or $w(T) = 1' (\mathbf{1})^*$, and again the first $1'$ would be identified with $1$, and the substring would be of type 2F and 3F).
If $w(T)$ has a final $F_1$-critical substring of type 2F, then $F_1$ changes the substring $11'$ into $2'1$, which is identified with $21$ since this is the only occurrence of $2$. If it is of type 3F, then $F_1$ changes the substring $1$ into $2$. In both cases, $F_1$ changes the last $1$ into $2$, coinciding with $F_1'$.

Now suppose that $F_1(T) = F'_1(T)$. If they are both undefined, then $\mathsf{wt}(T)_1 = 0$, which is a trivial case. Thus, we may assume that $F_1(T) = F'_1(T) \neq \emptyset$. Since $F'_1(T) \neq \emptyset$, we have $\mathsf{wt}(T)_1 > 0$ and there is no occurrences of $2'$ to the right of the last $1$ in $T$. Suppose that $\mathsf{wt}(T)_2 > 0$. We have the following cases:

\begin{description}
\item[Case 1] We assume there are no occurrences of $2$ after the mentioned $1$. Therefore, the occurrences of $\mathbf{2}$ to the left of $1$, since we are assuming that $\mathsf{wt}(T)_2 >0$. Moreover, $F_1(T)$ must coincide with $F'_1(T)$, so we either have:
	\begin{itemize}
	\item $F_1$ changes $1$ into $2$, which implies that the $2'$ resulting from $F'_1$ must be identified with $2$ in canonical form. For this to happen, this $2$ must be the only occurrence of $\mathbf{2}$ in $F_1(T)$. Hence, $\mathsf{wt}(F_{1}(T))_2 =1$ and necessarily $\mathsf{wt}(T)_2 = 0$, contradicting the hypothesis.
	\item $F_1$ changes $1$ into $2'$. For this to happen, $w(T)$ must have a final $F_1$-critical substring of type 1F or 2F. If it is 1F, there would be some $2'$ to the right of the last $1$ and $F'_1$ would not be defined. If it is 2F, and since there are no occurrences of $\mathbf{2}$ to the right of the last $1$, by assumption, it must be the case $11' \longmapsto 2' 1$. So, since $F_1(T) = F'_1(T)$, we have that the canonical form of $2'1'$ is  $2'1$. For $1$ and $1'$ to be identified, there must be no occurrences of $\mathbf{1}$ to the right of the last $1$. By hypothesis, $E_1(T) = E'_1(T) = \emptyset$. Clearly $E'_1$ is undefined since the last $2'$ is not right to the last $1$. Hence, it must be the case where $T$ has no final $E_1$-critical substring or has some of type 5E. If it is 5E, then it is either $2'$ at $x\geq 1, y = 1$. Since $y = 1$, there must be some $\mathbf{2}$ before the $2'$ of the substring, and since $x\geq 1$, there must be some $\mathbf{1}$ after it, which contradicts the non-existence of $\mathbf{1}$ to the right of the $1$ to be changed. Therefore, it must be the case where there is no final $E_1$-critical substring. Since we are assuming that $\mathsf{wt}(T)_2>0$, some $\mathbf{2}$ must appear before the last $1$, yielding at least some final $E_1$-critical substring of type 3E or 4E, which is a contradiction.
	\end{itemize}

\item[Case 2] Assume there are some $2$ after the referred $1$. Then, we must have the substring $12$ (the $1$ appearing is the one to be changed) at either one of these locations:
	\begin{itemize}
	\item At $x \geq 0$ and $y=0$. In this case $E_1$ would be defined, being type $4E$.
	\item At $x=0$ and $y > 0$. But then, $y > 0$ implies that there are some $\mathbf{2}$ before this string, placing it at location $y=0$ and yielding a 3E or 4E type.
	\item At $y=1$. Then, the $2$ is located at $y=0$ yielding a $4E$ type.
	\item At $y > 1$. In this case necessarily $x >1$, otherwise this would be a final $F_1$-critical substring of type 5F and $F_1(T) = \emptyset$. But then, the location obtained is not a valid one for this string to be the final critical substring. Therefore, the $1$ to be changed by $F_1$ is not the same as the one changed by $F'_1$, which contradicts their equality.
	\end{itemize}
	
\end{description}
\end{proof}

The next lemmas concern the total length functions, which are the total distances from a vertex to the highest and
lowest weight vertices of its $i$-string.

\begin{lema}\label{lemalenfun}
Let $T \in \mathcal{B}(\lambda/\mu,n)$ and let $i \in I$. Then,

$$\varepsilon_i(T) =
\begin{cases*}
\widehat{\varepsilon_i}(T) =\varepsilon'_i(T) & if $T$ is in a collapsed $i$-string\\
\widehat{\varepsilon_i}(T) + \varepsilon_i' (T) & if $T$ is in a separated $i$-string.
\end{cases*}$$
The result is also valid for $\varphi_i$ with the adequate changes.
\end{lema}

\begin{proof}Suppose that $T$ is in a collapsed $i$-string. Then, by Proposition \ref{isomhigh}, that collapsed string has an highest weight element $T_0$, and so $T_0 = E_i^k (T)$, for some $k \geq 0$. Since $T_0$ is an highest weight, then $E_i (T_0) = \emptyset$, hence $E_i^{k+1} (T) = \emptyset$. Consequently, $\widehat{\varepsilon}(T) = k$. On the other hand, $T_0$ is a LRS tableau (for the alphabet $\{i',i,(i+1)',i+1\}$), thus, by Proposition \ref{cor4.5}, the endpoint of the $i$-th lattice walk of its word has the $y$-coordinate equal to $0$. The operator $F_i$ shifts the endpoint of the $i$-th lattice walk by $(-1,1)$ \cite[Corollary 5.12]{GLP17}. Since $T_0 = E_i^k (T)$, then $T = F_i^k (T_0)$, and so the $i$-th lattice walk of $T$ has the $y$-coordinate equal to $k$. Then, $\varepsilon_i(T) = k = \widehat{\varepsilon_i} (T)$.

Now suppose that $T$ is in a separated $i$-string. This $i$-string has a highest weight element $T_0$, which is a LRS tableau for the mentioned alphabet. Thus, the endpoint of the $i$-th lattice walk of its word has the $y$-coordinate equal to zero. Then, we have two cases.

\begin{itemize}
\item Suppose that $T$ is such that $E_i'(T) \neq \emptyset$. Let $T_1 := E_i' (T)$. By definition of $E_i'$, $T_1$ is obtained from $T$ by replacing the last $(i+1)'$ that was right to the last $i$ with $i$. Hence, $E_i' (T_1) = E_i'^2 (T) = \emptyset$ and so

\begin{equation}\label{varepslinha}
\varepsilon_i' (T) = 1.
\end{equation}

Since $T_0$ is the highest weight element, we have that $E_i^k E_i' (T) = T_0$, for some $k \geq 0$. This is equivalent to $F_i' F_i^{k} (T_0) = T$, and since both $F_i$ and $F_i'$ shift the endpoint of the $i$-th lattice walk by $(-1,1)$ \cite[Propostion 4.9]{GLP17}, the $y$-coordinate of the $i$-th lattice walk of the word of $T$ must be equal to $k+1$. Hence,

\begin{equation}\label{vareps}
\varepsilon_i (T) = k+1.
\end{equation}

Since the operators $E_i$ and $E_i'$ commute when defined, we have $T_0 = E_i' E_i^k (T)$, and so $E_i^k (T) = F_i'(T_0) \neq \emptyset$ (recall that the shortest $i$-string is one with a $i'$-labelled edge). Thus, $E_i^{k+1} (T) = E_i F_i' (T) = \emptyset$, and we have

\begin{equation}\label{varepshat}
\widehat{\varepsilon_i}(T) = k.
\end{equation}

By \eqref{varepslinha}, \eqref{vareps} and \eqref{varepshat}, we have
$$\varepsilon_i (T) = \varepsilon_i' (T) + \widehat{\varepsilon_i}(T).$$

\item Suppose now that $T$ is such that $E_i' (T) = \emptyset$. Then, $\varepsilon_i'(T) = 0$. As in the previous case, there exists a highest weight element $T_0$ in this $i$-string, and the endpoint of the $i$-th lattice walk of its words has $y$-coordinate equal to zero. If $T=T_0$, then $\widehat{\varepsilon_i}(T)=0$ and the proof is done. Otherwise, there exists $k > 0$ such that $T_0 = E_i^k (T)$, and so, $F_i^k (T_0) = T$. Consequently, the endpoint of the $i$-th lattice walk of the word of $T$ has its $y$-coordinate equal to $k$. Hence,
\begin{equation}\label{vareps2}
\varepsilon_i(T) = k.
\end{equation}
Moreover, $E_i^{k+1} (T) = E_i (T_0) = \emptyset$, as $T_0$ is an highest weight element. So,
\begin{equation}\label{varepshat2}
\widehat{\varepsilon_i}(T) = k.
\end{equation}

Hence, by \eqref{vareps2} and \eqref{varepshat2}, and since $\varepsilon_i' (T) = 0$, we have $\varepsilon_i (T) = \varepsilon_i' (T) + \widehat{\varepsilon_i}(T)$.

\end{itemize}
\end{proof}

\begin{lema}\label{lemadasstrings}
Let $T \in \mathcal{B}(\lambda/\mu,n)$ and suppose its $i$-string has highest weight element $T_0^{\mathsf{high}}$ and lowest weight element $T_0^{\mathsf{low}}$, and that $T \neq T_0^{\mathsf{high}}, T_0^{\mathsf{low}}$. The following holds:
	\begin{enumerate}
	\item If the $i$-string is separated, then
		\begin{align*}
		T_0^{\mathsf{low}} &= F_i'^a F_i^b (T) =F_i^bF_i'^a(T) = F_i^{b-k} F_i'^a F_i^k (T),\; \text{for some}\; k\geq 0\\
		T_0^{\mathsf{high}} &= E_i'^c E_i^d (T) =  E_i^d E_i'^c(T) = E_i^{d-k} E_i'^c E_i^{k},\; \text{for some}\; k\geq 0
		\end{align*}
	
	with $a=\varphi_i'(T) \in \{0,1\}$, $b=\hat{\varphi}_i(T) \geq 0$, $c=\varepsilon_i'(T) \in \{0,1\}$, and $d=\hat{\varepsilon}_i(T) \geq 0$.
	\item If the $i$-string is collapsed, then
		\begin{align*}
		T_0^{\mathsf{low}} &= F_i^a (T) = F_i'^a (T)\\
		T_0^{\mathsf{high}} &= E_i^b (T) = E_i'^b (T)
		\end{align*}
	with $a = \varphi_i(T)$ and $b = \varepsilon_i(T)$.
	\end{enumerate}
\end{lema}
\begin{proof}
We prove the case for the separated $i$-string and for the raising operators. For collapsed $i$-string, the proof is similar, noting that by Lemma \ref{lemalenfun} we have $\varepsilon_i (T) = \widehat{\varepsilon_i} (T) = \varepsilon_i' (T)$. Let $T$ be in separated $i$-string. We have $E_i' (E_i'^c E_i^d (T)) = E_i'^{c+1} E_i^d (T) = E_i^d E_i'^{c+1} (T)$, since these operators commute. By definition of $c = \varepsilon_i'(T)$, we have $E_i'^{c+1} (T) = \emptyset$. Hence
$E_i' (E_i'^c E_i^d (T)) = E_i^d (\emptyset) = \emptyset $.
On the other hand, we have $E_i (E_i'^c E_i^d (T)) = E_i'^c E_i^{d+1} (T)$. By definition of $d = \widehat{\varepsilon_i}(T)$, $E_i^{d+1} = \emptyset$. Consequently, $E_i (E_i'^c E_i^d (T)) = E_i'^c (\emptyset) = \emptyset$. Thus, $E_i'^c E_i^d (T)$ must be the highest weight element of this $i$-string.
\end{proof}

\section{The Schützenberger involution and shifted crystal reflection operators}\label{sec4}

The Schützenberger involution, or Lusztig involutin, is defined on the shifted tableau crystal \cite[Section 2.3.1]{GL19} in the same fashion as for type $A$ Young tableau crystal. Similarly, we also realize it through shifted evacuation, for tableaux of straight shape, and through shifted reversal otherwise. For each $i \in I = [n-1]$, we define the shifted crystal reflection operator $\sigma_i$, using the primed and unprimed crystal operators $E_i, E_i', F_i$ and $F_i'$. We also show in Example \ref{exbraidnot} that they do not need to satisfy the braid relations and, therefore, do not yield a natural action of $\mathfrak{S}_n$ on this crystal. Throughout this section $\nu$ will denote a strict partition.

\begin{prop}\label{defSchu}
Let $\mathcal{B}(\nu,n)$ denote the shifted tableau crystal with $T^{\mathsf{high}}=Y_\nu$ as highest weight and $T_{\mathsf{low}}=\mathsf{evac}(Y_\nu)$ as lowest weight. Then, there exists a unique map of sets $\eta: \mathcal{B}(\nu,n) \longrightarrow \mathcal{B}(\nu,n)$ that satisfies the following, for all $T\in \mathcal{B}(\nu,n)$ and for all $i \in I$:
	\begin{enumerate}
	\item $E'_i \eta (T) = \eta F'_{n-i} (T)$.
	\item $E_i \eta (T) = \eta F_{n-i} (T)$.	
	\item $F'_i \eta (T) = \eta E'_{n-i} (T)$.
	\item $F_i \eta (T) = \eta E_{n-i} (T)$.
	\item $\mathsf{wt}(\eta(T)) = \mathsf{wt}(T)^{\mathsf{rev}}$.
	\end{enumerate}
This map may be defined on $\mathcal{B}(\lambda/\mu,n)$ by extending it to its connected components. Moreover, it coincides with the evacuation in $\mathcal{B}(\nu,n)$, and with the reversal on the connected components of $\mathcal{B}(\lambda/\mu,n)$.
\end{prop}

The map $\eta$ is called the \emph{Schützenberger or Lusztig involution}. We use the notation $\eta$ for both straight-shaped and skew tableaux. The map $\eta$ is indeed an involution on the set of vertices of $\mathcal{B}(\nu,n)$, that reverses all arrows and indices, thus sending the highest weight element to the lowest, i.e. $\eta(T^{\mathsf{high}}) = T^{\mathsf{low}}$, and vice versa. It is coplactic and a weight-reversing, shape-preserving involution. Note that the operator $\mathsf{c}_n$ also acts on $\mathcal{B}(\nu,n)$ by reversing arrows and indices, however it does not preserve the shape, although the resulting crystal $\mathsf{c}_n (\mathcal{B}(\nu,n))$ is isomorphic as sets to $\mathcal{B}(\nu,n)$ and to $\mathsf{evac}(\mathcal{B}(\nu,n))$, being isomorphic as crystal graph to the latter.

\begin{proof}[Proof of Proposition \ref{defSchu}]
It suffices to do the proof for $\mathcal{B}(\nu,n)$. We prove that the evacuation $\mathsf{evac}$ satisfies the assertions $(1)$ and $(2)$. Let $T \in \mathcal{B}(\nu,n)$ and let $i \in I$.  Since $\mathsf{evac}$ is an involution then $(3)$ and $(4)$ are satisfied, and we have seen already that $\mathsf{evac}$ satisfies $(5)$. By definition of the primed operators, $\mathsf{std} (E_{i}'\mathsf{evac}(T)) = \mathsf{std} (\mathsf{evac}(T))$. Therefore, since standardization commutes with evacuation, we have
$$\mathsf{std} (\mathsf{evac} E_{i}'\mathsf{evac}(T)) = \mathsf{std} (\mathsf{evac}^2(T)) = \mathsf{std} (T).$$
Moreover, we have
\begin{align*}
\mathsf{wt} (\mathsf{evac}E_{i}'\mathsf{evac}(T) )  &= \mathsf{wt} (E_{i}'\mathsf{evac}(T))^{\mathsf{rev}}\\
&= (\mathsf{wt}(\mathsf{evac}(T))+\alpha_{i})^{\mathsf{rev}}\\
&= (\mathsf{wt}(T)^{\mathsf{rev}} - \alpha_i^{\mathsf{rev}})^{\mathsf{rev}}\\
&=((\mathsf{wt}(T)-\alpha_i)^{\mathsf{rev}})^{\mathsf{rev}}\\
&= \mathsf{wt}(T) - \alpha_i.
\end{align*}
Hence, by the definition of $F_{n-i}'$, we have $\mathsf{evac} E_{i}' \mathsf{evac}(T) = F_{n-i}' (T)$ and consequently, $E_{i}' \mathsf{evac}(T) = \mathsf{evac} F_{n-i}'(T)$. 

To prove that $E_i \mathsf{evac} (T) = \mathsf{evac} F_{n-i} (T)$ we note that $E_i \mathsf{evac} (T)$ and $F_{n-i} (T)$ are in the same connected component of $\mathcal{B}(\nu,n)$, hence they are dual equivalent due to Proposition \ref{compdual}. Thus, it remains to show that $E_i \mathsf{evac} (T)$ and $\mathsf{c}_n (F_{n-i} (T))$ are shifted Knuth equivalent. We have that $w(\mathsf{evac}(T)) \equiv_k \mathsf{c}_n(T)$ and since $E_i$ is coplactic, we have $E_i (w(\mathsf{evac}(T))) \equiv_k E_i(\mathsf{c}_n(T))$. Then, we have

\begin{align}\label{eq:Eievac1}
w (E_i (\mathsf{evac}(T))) &= E_i (w(\mathsf{evac}(T)))\nonumber\\
&\equiv_k E_i (w(\mathsf{c}_n (T)))\\
&= \mathsf{c}_n F_{n-i} \mathsf{c}_n (w(\mathsf{c}_n(T))).\nonumber
\end{align}

By \eqref{wcolwrow}, we have $\mathsf{c}_n w(\mathsf{c}_n(T)) = w_{\mathsf{col}} (\mathsf{c}_n^2(T)) = w_{\mathsf{col}} (T)$. Moreover, the row and column words of a shifted semistandard tableau are shifted Knuth equivalent (see, for instance, \cite[Lemma 6.4.12]{Wor84}). Thus, since $\mathsf{c}_n$ and $F_{n-i}$ are coplactic,

\begin{align}\label{eq:Eievac2}
\mathsf{c}_n F_{n-i} \mathsf{c}_n (w(\mathsf{c}_n(T))) &= \mathsf{c}_n F_{n-i} (w_{\mathsf{col}}(T))\nonumber\\
&\equiv_k \mathsf{c}_n F_{n-i} (w(T))\\
&= \mathsf{c}_n w(F_{n-i} (T)).\nonumber
\end{align}

Finally, by \eqref{wcolwrow}, we have $\mathsf{c}_n w(F_{n-i} (T)) = w_{\mathsf{col}} (\mathsf{c}_n F_{n-i}(T)) \equiv_k w (\mathsf{c}_nF_{n-i}(T))$. Thus, from \eqref{eq:Eievac1} and \eqref{eq:Eievac2} we have 
$$w(E_i \mathsf{evac}(T)) \equiv_k w(\mathsf{c}_n F_{n-i}(T))$$ 
and, consequently, $\mathsf{evac} (F_{n-i}(T)) = E_i (\mathsf{evac}(T))$. Finally, we note that the definition of the shifted evacuation ensures that $\mathsf{wt}(\mathsf{evac}(T))=\mathsf{wt}(\mathsf{evac}(T))^{\mathsf{rev}}$.

For the uniqueness part, suppose that there is another involution $\xi$ on $\mathcal{B}(\nu,n)$ satisfying the previous properties and let $Y_{\nu}$ be the highest weight element of $\mathcal{B}(\nu,n)$. By Proposition \ref{isomhigh}, we have $T = H_{i_1}\ldots H_{i_k} (Y_{\nu})$, where $H_i \in \{F'_i, F_i, E'_i, E_i\}$, with $i_k \in I$. Moreover, let $\tilde{H}_i$ be $E_i'$ (respectively, $E_i$, $F_i'$ and $F_i$) if $H_i$ is $F_i'$ (respectively $F_i$, $E_i'$ and $E_i$). Then,
	\begin{align*}
	\xi(T) &= \xi H_{i_1}\ldots H_{i_k} (Y_{\nu})\\
	&= \tilde{H}_{n-i_1} \ldots \tilde{H}_{n-i_k} \xi(Y_{\nu})\\
	&= \tilde{H}_{n-i_1} \ldots \tilde{H}_{n-i_k} \mathsf{evac} (Y_{\nu})\\
	&=  \mathsf{evac} H_{i_1} \ldots H_{i_k} (Y_{\nu})\\
	&=  \mathsf{evac}(T).
	\end{align*}
\end{proof}

\begin{lema}\label{lemaphieta}
Let $T \in \mathsf{ShST}(\lambda/\mu,n)$ and let $i \in I$. Then,
\begin{enumerate}
\item $\varphi_i (T) = \varepsilon_{n-i} \eta (T)$.
\item $\varepsilon_i (T) = \varphi_{n-i} \eta (T)$.
\end{enumerate}
\end{lema}

\begin{proof}
We prove the first assertion, the second one is analogous. Using Lemma \ref{lemadasstrings}, there are two cases:
	\begin{enumerate}
	\item Suppose that the $i$-component in which $T$ lies is a collapsed $i$-string $\mathcal{S}_i^0$. Then, by Lemma \ref{lemadasstrings}, $F_{i}^{\varphi_i(T)}(T)$ is the lowest weight element of $\mathcal{S}_i^0$. We have that $\eta F_{i}^{\varphi_i(T)}(T)$ is in a $(n-i)$-string $\mathcal{S}_{n-i}^0$ (which is also collapsed) and by Proposition \ref{defSchu},	
	$$\eta F_{i}^{\varphi_i(T)}(T) = E_{n-i}^{\varphi_i(T)} (\eta (T)).$$	
Hence, by the definition of $\varepsilon_{n-i}$, we have that $\varepsilon_{n-i} (\eta (T)) \geq \varphi_i (T)$.

On the other hand, since $E_{n-i}^{\varepsilon_{n-i}(\eta (T))} (\eta(T))$ is in $\mathcal{S}_{n-i}^0$, then $$\eta E_{n-i}^{\varepsilon_{n-i}(\eta (T))} (\eta(T))$$ must be in $\mathcal{S}_i^0$. By Proposition \ref{defSchu}, and since $\eta$ is an involution, we have
$$\eta E_{n-i}^{\varepsilon_{n-i}(\eta (T))} (\eta(T)) = F_i^{\varepsilon_{n-i}(\eta(T))} (T).$$
Consequently, by the definition of $\varphi_i$, we have $\varphi_i (T) \geq \varepsilon_{n-i} (\eta (T))$.

\item Now suppose that $T$ is in $\mathcal{S}_i^0$, a separated $i$-string. Then, by Lemma \ref{lemadasstrings} $F_i'^{\varphi_i'(T)} F_i^{\hat{\varphi}_i (T)} (T)$ is the lowest weight element of $\mathcal{S}_i^0$. Consequently, $\eta F_i'^{\varphi_i(T)} F_i ^{\hat{\varphi}_i (T)} (T)$ is in a $(n-i)$-string $\mathcal{S}_{n-i}^0$ (which is also separated). As before, we have
$$ \eta F_i'^{\varphi_i(T)} F_i ^{\hat{\varphi}_i (T)} (T) =  E_{n-i}'^{\varphi_i'(T)} E_{n-i}^{\hat{\varphi}_i(T)} (\eta(T)),$$
and by definition of $\varepsilon_{n-i}$, we have $\varepsilon_{n-i}(\eta(T)) \geq \varphi_i'(T) + \hat{\varphi}_i(T) = \varphi_i (T)$.

Since $E_{n-i}'^{\varepsilon_{n-i}'(\eta(T))} E_{n-i}^{\hat{\varepsilon}_{n-i} (\eta(T))} (\eta (T))$ is in $\mathcal{S}_{n-i}^0$, we have that  $$\eta E_{n-i}'^{\varepsilon'_{n-i}(\eta(T))} E_{n-i}^{\hat{\varepsilon}_{n-i} (\eta(T))} (\eta (T))$$ is in $\mathcal{S}_{n-i}^0$. By Proposition \ref{defSchu}, and since $\eta$ is an involution, we have
$$\eta E_{n-i}'^{\varepsilon'_{n-i}(\eta(T))} E_{n-i}^{\hat{\varepsilon}_{n-i} (\eta(T))} (\eta (T)) =  F_{i}'^{\varepsilon'_{n-i}(\eta(T))} F_{i}^{\hat{\varepsilon}_{n-i} (\eta(T))} (T),$$
and then, $\varphi_i (T) \geq \varepsilon_{n-i}'(\eta(T)) + \hat{\varepsilon}_{n-i}(\eta(T)) = \varepsilon_{n-i}(\eta(T))$.
\end{enumerate}
\end{proof}

\subsection{The shifted reflection crystal operators}
We now introduce a shifted version of the \textit{crystal reflection operators} $\sigma_i$ (see \cite[Definition 2.35]{BumpSchi17}) on $\mathcal{B}(\nu,n)$, for each $i\in I$. Crystal reflection operators were originally defined by Lascoux and Schutzenberger \cite{LaSchu81} in the Young tableau crystal of type $A$. They are involutions, on type $A$ crystals, so that each $i$-string is sent to itself by reflection over its middle axis, for all $i \in I$. It coincides with the restriction of the Schützenberger involution to the tableaux consisting of the letters  $i,i+1$, ignoring the remaining ones. On $\mathcal{B}(\nu,n)$, collapsed strings are similar to the $i$-strings of type $A$ crystals, hence the shifted reflection operator $\sigma_i$ is expected to resemble the one for Young tableaux. However, for separated strings, a sole reflection of the $i$-string would not coincide with the restriction of the Schützenberger involution to $\{i,i+1\}'$, hence we have the next definition.

\begin{defin}\label{def:crystalreflection}
Let $i\in I$ and $T\in \mathcal{B}(\nu,n)$. Let $k= \langle \mathsf{wt}(T),\alpha_i \rangle$ (usual inner product in $\mathbb{R}^n$). The \emph{shifted crystal reflection operator} $\sigma_i$ is defined as follows

$$\sigma_i (T) =
\begin{cases*}
F_i' F_i^{k-1} (T) & if $k >0$ and $F_i'(T) \neq \emptyset$\\
E_i' F_i^{k+1} (T) & if $k > 0$ and $F_i'(T) = \emptyset$\\
E_i F_i' (T) & if $k=0$ and $F_i'(T)\neq \emptyset$\\
E_i' F_i (T) & if $k=0$ and $F_i' (T) = \emptyset$ and $F_i (T) \neq \emptyset$\\
T & if $k=0$ and $F_i' (T) = F_i (T) = \emptyset$\\
E_i^{-k+1} F_i' (T) & if $k < 0$ and $F_i' (T) \neq \emptyset$\\
E_i^{-k-1} E_i'(T) & if $k<0$ and $F_i' (T) = \emptyset$
\end{cases*}$$
\end{defin}

As the definition suggests, the shifted reflection operator $\sigma_i$ must do a double reflection, by vertical and horizontal middle axes (see Figure \ref{fig:flip}). As we will see in Theorem \ref{sigmarever}, a simple reflection in the same fashion as the type $A$ crystal fails to ensure the coincidence of the shifted crystal reflection operators with the adequate restriction of the Schützenberger involution, on separated strings. By coplacity, the operator $\sigma_i$ is extended to $\mathcal{B}(\lambda/\nu,n)$, for $i\in I$.

\begin{figure}[H]
\centering
\includegraphics[scale=0.8]{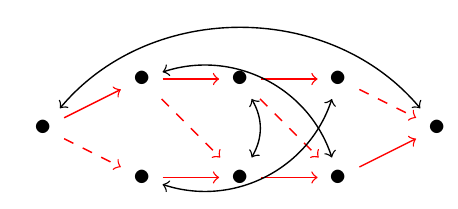}
\includegraphics[scale=0.8]{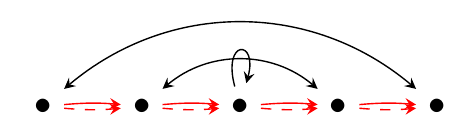}
\caption{The action of a crystal reflection operator in separated and collapsed strings, which corresponds to the Schützenberger involution.}
\label{fig:flip}

\end{figure}

We remark that this definition is the same for both separated or collapsed strings. However, for the latter there is simpler formulation, as stated in the following lemma, since in this case the primed and unprimed operators coincide.

\begin{lema}
Let $i \in I$ and let $T \in \mathsf{ShST}(\lambda/\mu,n)$ be such that $F_i (T) = F_i'(T)$. Let $k= \langle \mathsf{wt}(T),\alpha_i \rangle$. Then,
$$\sigma_i(T) =
\begin{cases*}
F_i^k(T) & if $k > 0$\\
T & if $k=0$\\
E_i^{-k} (T) & if $k <0$
\end{cases*}$$
\end{lema}

\begin{prop}\label{sigmai}
For $i \in I$, the operator $\sigma_i$ satisfies the following:
	\begin{enumerate}
	\item $\sigma_i$ sends each connected component of $\mathcal{B}(\lambda/\mu,n)$ to itself.
	\item $\sigma_i$ takes each $i$-string to itself.
	\item $\sigma_i^2 = id$ and $\sigma_i\sigma_j=\sigma_j\sigma_i$, if $|i-j| >1$.
	\item $\mathsf{wt}(\sigma_i(T)) = \theta_i \cdot \mathsf{wt}(T)$, where $\theta_i = (i,i+1) \in \mathfrak{S}_n$.
	\end{enumerate}
\end{prop}

\begin{proof}
The first two assertions result directly from the definition of the raising and lowering operators.
For the third assertion, it is clear that $\sigma_i\sigma_j=\sigma_j\sigma_i$, for $|i-j| >1$, since each $\sigma_k$ acts only on the primed subinterval of adjacent letters $\{k,k+1\}'$, leaving the remaining ones unchanged. To prove that $\sigma_i$ is an involution, we must analyse various cases according to Definition \ref{def:crystalreflection}. Let $T \in \mathsf{ShST} (\lambda/\mu,n)$.
	\begin{description}
	\item[Case 1.] Suppose that $k>0$ and that $F_i(T)\neq \emptyset$. Let $S = \sigma_i (T) = F_i' F_i^{k-1} (T)$. Then, $F_i' (S) = \emptyset$. By definition of $\sigma_i$, we have
		\begin{align*}
		\mathsf{wt}(S) &= \mathsf{wt} (F_i' F_i^{k-1} (T))\\
		&= \mathsf{wt}(F_i^{k-1}(T))-\alpha_i\\
		&= \mathsf{wt}(T) - (k-1)\alpha_i - \alpha_i\\
		&= \mathsf{wt}(T) - k\alpha_i.
		\end{align*}
		
	Hence, putting $\tilde{k} := \mathsf{wt}(S)_i - \mathsf{wt}(S)_{i+1}$ we have
		\begin{align*}
		\tilde{k} &= (\mathsf{wt}(T)_i - (k\alpha_i)_i) - (\mathsf{wt}(T)_{i+1}-(k\alpha_i)_{i+1})\\
		&= \mathsf{wt}(T)_i - k - \mathsf{wt}(T)_{i+1} - k\\
		&= -k < 0.
		\end{align*}
		
	Consequently,
		\begin{align*}
		\sigma_i(S) &= E_i^{-(-k)-1} E_i' (S)\\
		&= E_i^{k-1} E_i' (S)\\
		&= E_i^{k-1} E_i' F_i' F_i^{k-1} (T)\\
		&= E_i^{k-1} F_i^{k-1} (T) = T.
		\end{align*}
	
	\item[Case 2.] Now suppose $T$ is such that $k > 0$ and $F_i' (T) = \emptyset$. Let $S = \sigma_i(T)$. We have $\mathsf{wt}(S) = \mathsf{wt}(T) + \alpha_i - (k+1)\alpha_i = \mathsf{wt}(T) - k\alpha_i $. Using the same notation as before, we have
		\begin{align*}
		\tilde{k} &= (\mathsf{wt}(T)_i - (k\alpha_i)_i) - (\mathsf{wt}(T)_{i+1} - (k\alpha_i)_{i+1})\\
		&= \mathsf{wt}(T)_i - \mathsf{wt}(T)_{i+1} -2k \\
		&= -k < 0.
		\end{align*}
	
	Moreover, since $E_i' F_i^{k+1} (T)$ is defined, this means that in the word of $F_i^{k+1}(T)$ the last $(i+1)'$ was to the right of the last $i$, and was then changed to $i$, due to Proposition \ref{computEF}. Consequently, in S the last $i$ is to the right of the last $(i+1)'$, which means that $F_i' (S) \neq \emptyset$. Hence,
		\begin{align*}
		\sigma_i(S) &= E_i'^{-(-k)+1} F_i' (S)\\
		&= E_i^{k+1} F_i' E_i' F_i^{k+1} (T)\\
		&= E_i^{k+1} F_i^{k+1} (T) = T.
		\end{align*}
		
	\item[Case 3.] Suppose that $T$ is such that $k = 0$ and $F_i' (T) \neq \emptyset$. Let $S = \sigma_i (T)$. We have $\mathsf{wt}(S) = \mathsf{wt} (T)$, hence $\tilde{k} = k$. Since $F_i'(T) \neq \emptyset$, we have $F_i' (S) = \emptyset$. Then,
		\begin{align*}
		\sigma_i(S) &= E_i F_i' (S)\\
		&=E_i F_i' E_i' F_i (T)\\
		&= E_i F_i (T) = T.
		\end{align*}

	\item[Case 4.] Suppose that $k =0$ and that $F_i'(T) = F_i'(T) = \emptyset$. Then, $\sigma_i^2 (T) = \sigma_i (T) = T$.	
	\end{description}
The remaining cases are dual to the first three, which concludes the proof that $\sigma_i^2 = id$.
Finally, using the same notation as before, for the first case we have
	\begin{align*}
	\mathsf{wt}(\sigma_i(T)) &= \mathsf{wt} (T) - k \alpha_i\\
	&= \mathsf{wt}(T) - (\mathsf{wt}(T)_i - \mathsf{wt}(T)_{i+1})\alpha_i\\
	&= (\mathsf{wt}(T)_1, \ldots, \mathsf{wt}(T)_i - \mathsf{wt}(T)_i+ \mathsf{wt}(T)_{i+1},\\
	&\quad \mathsf{wt}(T)_{i+1} + \mathsf{wt}(T)_i - \mathsf{wt}(T)_{i+1}, \ldots, \mathsf{wt}(T)_n)\\
	&= (\mathsf{wt}(T)_1, \ldots, \mathsf{wt}(T)_{i+1}, \mathsf{wt}(T)_i, \ldots, \mathsf{wt}(T)_n)\\
	&= \theta_i \cdot (\mathsf{wt}(T)).
	\end{align*}
The remaining cases are proved analogously.
\end{proof}

For each $i\in I$, let $\mathcal{B}_{i,i+1}$ be the subgraph of $\mathcal{B}(\lambda/\mu)$ consisting of the $\{i,i'\}$-connected components. Given $T \in \mathsf{ShST}(\lambda/\mu,n)$ and $i \in I$, let $T^{i,i+1} := T^{i} \sqcup T^{i+1}$ be the tableau obtained from $T$ considering only the letters in $\{i,i+1\}'$, and let $\lambda^i/\mu^i$ denote its skew-shape of $T^{i,i+1}$. Let $\mathcal{C}$ be the (unique) connected component of $\mathcal{B}_{i,i+1}$ containing $T$. We may ignore the entries of the tableaux in $\mathcal{C}$ that are not in $\{i,i+1\}'$ and subtract $i-1$ to each entry and to each edge label, thus obtaining a connected component of the shifted tableau crystal $\mathcal{B}(\lambda^i/\mu^i,2)$. Then, when we write $\eta(T^{i,i+1})$, we mean that we first apply the Schützenberger involution $\eta$ to $T^{i,i+1}$ as an element of $\mathcal{B}(\lambda^i/\mu^i,2)$, and then we add $i-1$ to the entries of the resulting tableau, to obtain one in the alphabet $\{i,i+1\}'$, in $\mathcal{C}$. Thus, we define
\begin{equation}\label{eq:etaii1}
\eta_{i,i+1}(T):=T^1\sqcup\cdots \sqcup T^{i-1}\sqcup\eta(T^{i,i+1})\sqcup T^{i+2}\sqcup\cdots\sqcup T^n\in \mathcal{C}, 
\end{equation}
where $\eta$ is the Schützenberger involution on $\mathcal{B}(\lambda^i/\mu^i,2)$. This makes rigorous the restriction of the Schützenberger involution on $\mathcal{B}(\lambda/\mu)$ to the alphabet $\{i,i',i+1,i+1'\}$.

\begin{teo}\label{sigmarever}
Let $T$ be a shifted semistandard tableau on the alphabet $[n]'$. Then, for any $i \in I$,
	$$\sigma_i (T) = \eta_{i,i+1}(T).$$
\end{teo}

It suffices to prove this result for tableaux on the primed alphabet of two adjacent letters and we consider it to be $\{1,2\}'$, to simplify the notation. Moreover, the raising and lowering operators are coplactic, thus $\sigma_1$ is also coplactic. Hence, it suffices to prove the result for tableaux of straight shape. We remark that such tableaux have at most two rows. Furthermore, $T$ and $\sigma_1(T)$ are in the same $1$-string (which, in particular, is a connected component), hence by Proposition \ref{compdual}, $T$ and $\sigma_1(T)$ are shifted dual equivalent. It remains to show that $\mathsf{c}_1(T)$ and $\sigma_1(T)$ are shifted Knuth equivalent. We remark that another proof may be done by directly verifying the conditions on Proposition \ref{defSchu}. The one we choose to present highlights some of the properties of straight-shaped tableaux with at most two rows. First, we introduce some technical results on shifted Knuth equivalence.

\begin{lema}\label{lemam1}
Let $a_1, \ldots, a_n, b_1, \ldots, b_m, c \in [n]'$, with $m,n \geq 1$.
\begin{enumerate}
\item If $b_m < \ldots < b_1 < c < a_1 < \ldots < a_n$ in standardization ordering, then $$c a_1 \ldots a_n b_1 \ldots b_m \equiv_k c b_1 \ldots b_m a_1 \ldots a_n.$$

\item If $a_1 < \ldots < a_n < c < b_m < \ldots < b_1$ in standardization ordering, then $$ a_1 \ldots a_n b_1 \ldots b_m c\equiv_k  b_1 \ldots b_m a_1 \ldots a_n c.$$
\end{enumerate}
\end{lema}

\begin{proof}
We prove the first part by induction on $n$, the second part is proved similarly. If $n=1$, we have
	\begin{alignat*}{2}
	\underline{c a_1 b_1} b_2 \ldots b_m &\equiv_k c \underline{b_1 a_1 b_2} \ldots b_m &\qquad\text{(K1)}\\
	&\equiv_k c b_1 \underline{b_2 a_1 b_3} \ldots b_m &\text{(K1)} \\
	&\ldots\\
	&\equiv_k c b_1 \ldots \underline{b_{m-1} a_1 b_m} &\text{(K1)}\\
	&\equiv_k c b_1 \ldots b_{m-1} b_m a_1 &\text{(K1)}
	\end{alignat*}
	
Now suppose the result is true for some $n \geq 1$ and let $a_{n+1} > a_n$ in standardization ordering. Then,
	\begin{alignat*}{2}
	c a_1 \ldots  \underline{a_n a_{n+1} b_1} b_2 \ldots b_m &\equiv_k c a_1 \ldots a_n \underline{b_1 a_{n+1} b_2} b_3 \ldots b_m &\text{(K1)}\\
	&\ldots\\
	&\equiv_k c a_1 \ldots a_n b_1 \ldots \underline{b_{m-1} a_{n+1} b_m} &\text{(K1)}\\
	&\equiv_k c a_1 \ldots a_n b_1 \ldots b_m a_{n+1} &\text{(K1)}\\
	&\equiv_k c b_1 \ldots b_m a_1 \ldots a_n a_{n+1} &\text{{\footnotesize Induction hypothesis + Lemma \ref{congrdir}}}
	\end{alignat*}
\end{proof}

\begin{lema}\label{lemadosas}
Let $a \in [n]'$. Then, for any $m \geq 1$, $a (a')^m \equiv_k a^{m+1}.$
\end{lema}

\begin{proof}
For $m=1$ the result corresponds to the (S2) relation. Suppose the result holds for some $m \geq 1$. Then,
\begin{alignat*}{2}
a (a')^{m+1} &= a (a')^m a' &\\
&\equiv_k (a)^{m+1} a' &\quad\text{{\footnotesize Induction hypothesis + Lemma \ref{congrdir}}}\\
&= a (a)^{m} a'\\
&\equiv_k \underline{a a'} (a)^m &\text{{\footnotesize Lemma \ref{lemam1}}}\\
&\equiv_k aa (a)^m &\text{(S2)}\\
&= (a)^{m+2}.
\end{alignat*}
\end{proof}

In order to prove that $\mathsf{c}_n (T)$ and $\sigma_1(T)$ are shifted Knuth equivalent, we will present sequences of Knuth moves between their words. We have remarked that $T$ has at most two rows. The case where it has one row is in the following result.

\begin{prop}\label{prop:sigmaevac1}
Let $T$ be a shifted semistandard tableau with one row, filled with in the alphabet $\{1,2\}'$. Then,
$\sigma_1 (T) = \mathsf{evac}(T)$.
\end{prop}

\begin{proof}
As stated before, if suffices to show that $\sigma_1 (T)$ is shifted Knuth equivalent to $\mathsf{c}_n(T)$. Suppose that $\mathsf{wt}(T) = 1^a$, for $a \geq 1$. If $a=1$, then $w(\mathsf{c}_n(T)) = 2 = w(F_1 (T)) = w(\sigma_1 (T))$. If $a > 1$, then $w(\mathsf{c}_n(T)) = 2(2')^{a-1}$ and $w(\sigma_1 (T)) = w(F_1^a (T)) = 2^a$. Hence, by Lemma \ref{lemadosas}, $w(\mathsf{c}_n(T)) \equiv_k w(\sigma_1 (T))$. Now suppose that $w(T) = 1^a 2^b$, with $a,b \geq 1$. Then, $w(\sigma_1(T)) = 1^b 2^a$ and $w(\mathsf{c}_n(T)) = 2(2')^{a-1} 1 (1')^{b-1}$. There are two cases:
	\begin{description}
	\item[Case 1] If $a=1$, we have
		\begin{align*}
 		\underline{21} (1')^{b-1} &\equiv_k 12 (1')^{b-1} &\text{(S1)}\\
 		&\equiv_k 1 (1')^{b-1} 2 & \text{Lemma \ref{congrdir}}\\
 		&\equiv_k 1^b 2. &\text{Lemmas \ref{congrdir} and \ref{lemadosas}}
		\end{align*}
		
	\item[Case 2] If $a > 1$, then we have
		\begin{align*}
		2(2')^{a-1} 1 (1')^{b-1} &\equiv_k 2^a 1 (1')^{b-1} &\text{Lemmas \ref{lemadosas} and \ref{congrdir}}\\
		&= 2 2^{a-1} 1 (1')^{b-1} &\\
		&\equiv_k \underline{2 1} (1')^{b-1} 2^{a-1} &\text{Lemma \ref{lemam1}}\\
		&\equiv_k 1 2 (1')^{b-1} 2^{a-1} & \text{(S1)}\\
		&\equiv_k 1 (1')^{b-1} 2 2^{a-1} & \text{Lemmas \ref{lemam1} and \ref{congrdir}}\\
		&\equiv 1^b 2^a. & \text{Lemmas \ref{lemadosas} and \ref{congrdir}}
		\end{align*}
	\end{description}
\end{proof}

If $T$ has two rows, we remark that it suffices to verify the case where the second row has only one box. To make this statement rigorous, we need to introduce some notation. A shifted semistandard tableau $T$ is called \emph{detached} if its main diagonal has exactly one box. Then, we may define the following operator on shifted semistandard tableaux:

$$\mathsf{r}(T) =
\begin{cases*}
T & if $T$ is detached\\
\widehat{T} & otherwise
\end{cases*}$$
where $\widehat{T}$ is obtained from $T$ removing its main diagonal and shifting every box one unit to the left (so that its second diagonal becomes the main diagonal).

\begin{ex}
If $T = \begin{ytableau}
1 & 1 & 1 & 1 & 2' & 2\\
\none & 2 & 2 & 2
\end{ytableau}$, then $\mathsf{r}(T) = \begin{ytableau}
1 & 1 & 1 & 2' & 2\\
\none & 2 & 2
\end{ytableau}$,
 $\mathsf{r} ^2(T) = \begin{ytableau}
1 & 1 & 2' & 2\\
\none & 2
\end{ytableau}$, and 
$\mathsf{r}^m (T) = \begin{ytableau}
1 & 2 & 2
\end{ytableau}$, for $m \geq 3$.
\end{ex}

The following lemma states that, if $T$ is not detached and its $(l+1)$-th diagonal is the first with one box, then $\sigma_1 (T)$ is determined by $\mathsf{r}^{l-1} (T)$, i.e., one may temporarily remove the first $l-1$ diagonals with two elements, compute $\sigma_1$ on the remaining tableau, and then place the diagonals back.

\begin{lema}\label{sigmadetached}
Let $T$ be a shifted semistandard tableau of straight shape, with two rows, filled in the alphabet $\{1,2\}'$. Let $l$ be such that $\{(1,l),(2,l+1)\}$ and $\{(1,l+1)\}$ are adjacent diagonals of $T$ with two and one box, respectively. Then,
$$\mathsf{r}^{l-1} \sigma_1 (T) = \sigma_1 \mathsf{r}^{l-1}(T).$$
\end{lema}

\begin{proof}
If $T=Y_{\nu}$, for $\nu = (\nu_1, \nu_2)$, then $wt (\sigma_1 (Y_{\nu})) = (\nu_2,\nu_1)$. Consequently, $\sigma_1 (Y_{\nu}) = \mathsf{evac}(Y_{\nu})$. Then, $\mathsf{r}^{\nu_2 -1} \mathsf{evac} (Y_{\nu}) = \mathsf{evac}(Y_{\nu^0})$, where $\nu^0 = (\nu_1-\nu_2,1)$. Similarly, $\mathsf{r}^{\nu_2-1} (Y_{\nu}) = Y_{\nu^0}$, and using the same argument with the weight, $\sigma_1 \mathsf{r}^{\nu_2-1} (Y_{\nu}) = \mathsf{evac}(Y_{\nu^0})$. The proof for $\mathsf{evac}(Y_{\nu})$ is similar.

Suppose now that $T$ is neither $Y_{\nu}$ nor $\mathsf{evac}(Y_{\nu})$. Suppose that the word of $T$ is given by $w(T)=2^a 1^{a+1} 1^b \mathbf{2} 2^c$, with $a \geq 1$ and $b, c \geq 0$. Then, $\mathsf{wt}(T)=(a+b+1,a+c+1)$ and considering Definition \ref{def:crystalreflection}, we have $k = (a+b+1)-(a+c+1) = b-c$ (note that it does not depend on $a$). We show the case when $k > 0$ and $\mathbf{2} =2$. The proof for the other cases is analogous. If $\mathbf{2} = 2$, then $F_1 (T) \neq \emptyset$ and we have
	\begin{align*}
	\sigma_i(T) &= F_1' F_1^{b-c-1} (T)\\
	&= F_1' F_1^{b-c-1} (2^a 1^{a+1} 1^b 2^{c+1})\\
	&= F_1' F_1^{b-c-1} (2^a 1^{a+1} 1^{b -(b-c-1)} 2^{(c+1)+(b-c-1)})\\
	&= F_1' (2^a 1^{a+1} 1^{c+1} 2^b)\\
	&= 2^a 1^{a+1} 1^{c} 2' 2^b
	\end{align*}
	and so, $\mathsf{r}^{a-1} \sigma_1(T) = 2 1^2 1^c 2' 2^b$.
		
	On the other hand, we have $\mathsf{r}^{a-1}(T) = 2 1^2 1^b 2 2^c$ and so
	\begin{align*}
	\sigma_1 \mathsf{r}^{a-1}(T) &= F_1' F_1^{b-c-1} (2 1^2 1^b 2^{c+1})\\
	&= F_1' (2 1^2 1^{c+1} 2^{b})\\
	&= 2 1^2 1^c 2' 2^b.
	\end{align*}

\end{proof}

In what follows, we consider $T$ to be of shape $\nu = (m, 1)$, i.e., such that its second row has only one box. To show that $\mathsf{c}_1 (T)$ is shifted Knuth equivalent to $\sigma_1(T)$ is equivalent to show that $\mathsf{rect}(\mathsf{c}_1 (T)) = \sigma_1(T)$, since $T$ is of straight shape. Moreover, we ask for $T$ to be neither $Y_{\nu}$ nor $\mathsf{evac}(Y_{\nu})$, since the result for those cases is already proved. We have the following lemma, which is easy to prove.

\begin{lema}
Let $\nu = (m,1)$, for $m \geq 3$. Let $T \in \mathsf{ShST}(\nu,2)$ such that $T \neq Y_{\nu}, \mathsf{evac}(Y_{\nu})$ and let $k=\varepsilon_1(T)$. Then,
	\begin{enumerate}
	\item If $T = \begin{ytableau}
	1 & 1 & 1 & {\ldots} & 1 & {2'} & 2 & {\ldots} & 2\\
	\none & 2
	\end{ytableau}$, with $\mathsf{wt}(T) = (m-k,k+1)$, then
	$$\sigma_1 (T) = \begin{ytableau}
	1 & 1 & 1 & {\ldots} & 1 & 2 & 2 & {\ldots} & 2\\
	\none & 2
	\end{ytableau}$$
	with $\mathsf{wt} (\sigma_1(T)) = (k+1,m-k)$.
	
		\item If $T = \begin{ytableau}
	1 & 1 & 1 & {\ldots} & 1 & {2} & 2 & {\ldots} & 2\\
	\none & 2
	\end{ytableau}$, with $\mathsf{wt}(T) = (m-k,k+1)$, then
	$$\sigma_1 (T) = \begin{ytableau}
	1 & 1 & 1 & {\ldots} & 1 & {2'} & 2 & {\ldots} & 2\\
	\none & 2
	\end{ytableau}$$
	with $\mathsf{wt} (\sigma_1(T)) = (k+1,m-k)$.
	\end{enumerate}
\end{lema}

The rectification process does not depend on the sequence of inner corners, so, for simplicity, we may fix that we always choose the rightmost inner corner in the highest-index row. Thus, we apply \textit{jeu de taquin} slides on $\mathsf{c}_1 (T)$, following this sequence, to the point where an occurrence of $2$ or $2'$ on $T$ (which correspond to $1'$ or $1$ in $\mathsf{c}_1 (T)$) will determine different slides on the next move. For instance, consider the following tableaux:

\begin{align*}
T_1 &= \begin{ytableau}
1 & 1 & 1 & 2 & 2\\
\none & 2
\end{ytableau} \longrightarrow
\mathsf{c}_1 (T_1) = \begin{ytableau}
{} & {} & {} & {} & {1'}\\
\none & {} & {} & {} & {\color{red} 1'}\\
\none & \none & {} & {} & {2'}\\
\none & \none & \none & 1 & {2'}\\
\none & \none & \none & \none & 2
\end{ytableau}
\equiv_k
\begin{ytableau}
{} & {} & {} & {} & {1'}\\
\none & {} & *(gray){} & {\color{red} 1'} & 2\\
\none & \none & 1 & {2'}\\
\none & \none & \none & 2
\end{ytableau}
& \text{{\footnotesize the box $(2,4)$ with {\color{red}$1'$} will go \textbf{left} on the next slide.}}
\\
T_2 &= \begin{ytableau}
1 & 1 & 1 & {2'} & 2\\
\none & 2
\end{ytableau} \longrightarrow
\mathsf{c}_1 (T_2) = \begin{ytableau}
{} & {} & {} & {} & {1'}\\
\none & {} & {} & {} & {\color{red}1}\\
\none & \none & {} & {} & {2'}\\
\none & \none & \none & 1 & {2'}\\
\none & \none & \none & \none & 2
\end{ytableau}
\equiv_k
\begin{ytableau}
{} & {} & {} & {} & {1'}\\
\none & {} & *(gray){} & {\color{red} 1} & 2\\
\none & \none & 1 & {2'}\\
\none & \none & \none & 2
\end{ytableau}
& \text{{\footnotesize the box $(3,3)$ with 1 will go \textbf{up} on the next slide.}}
\end{align*}

Continuing the rectification process we obtain, respectively:

$$
\mathsf{rect}(\mathsf{c}_1 (T_1)) = \begin{ytableau}
1 & 1 & 1 & 2' & 2\\
\none & 2
\end{ytableau}\qquad \quad
\mathsf{rect}(\mathsf{c}_1 (T_2)) = \begin{ytableau}
1 & 1 & 1 & 2 & 2\\
\none & 2
\end{ytableau}.
$$

We begin by stating some auxiliary results.

\begin{lema}\label{lemacasomaissimples}
Let $T=Y_{\nu}$ be the highest weight of its $1$-string, with $m \geq 3$, for $\nu = (m,1)$. Then,
	\begin{enumerate}
	\item $w (\mathsf{c}_1 F_1(T)) = 21(2')^{m-2} 1 \equiv_k 21 2' 1 (2)^{m-3}$.
	\item $w (\mathsf{c}_1 F_1'(T)) = 21(2')^{m-2} 1' \equiv_k 21 2' 1' (2)^{m-3}$.
	\end{enumerate}
\end{lema}

\begin{proof}
We prove a more general claim that, if $a \geq 1$, then,
	\begin{align*}
	21(2')^a 1 &\equiv_k 21 2' 1 (2)^{a-1}\\
	21(2')^a 1' &\equiv_k 21 2' 1' (2)^{a-1}.
	\end{align*}
Then, the result follows, observing that $m \geq 3$ ensures that $a:=m-2 \geq 1$. 
If $a=1$, the claim is trivial. Suppose this is true for some $a \geq 1$. We have,
$$ 21(2')^{a+1} 1 = 21 (2')^a 2' 1.$$

The word $21 (2')^a 2'$ has the same standardization of $21 (2')^a 1$. Therefore, by induction hypothesis $
21 (2')^a 2' \equiv_k 21 2' 2' (2)^{a-1}$. Then, we have
\begin{align*}
21 (2')^a 2' 1  &\equiv_k 2\underline{ 1 2' 2'} (2)^{a-1} 1 & \text{Lemma \ref{congresq}}\\
&\equiv_k \underline{22'} 1 2' (2)^{a-1} 1 &\text{(K2)} \\
&\equiv_k \underline{221}2' (2)^{a-1} 1 &\text{(S2)}\\
&\equiv_k \underline{21} 22' (2)^{a-1} 1 &\text{(K1)}\\
&\equiv_k 1 \underline{2 22'} (2)^{a-1} 1 &\text{(S1)}\\
&\equiv_k \underline{12} 2'2 (2)^{a-1} 1 & \text{(K1)}\\
&\equiv_k 21 2'2 (2)^{a-1} 1 & \text{(S1)}\\
&= 21 2' (2)^{a} 1.
\end{align*}

Moreover, we have $2'(2)^a 1 \equiv_k 2' 1 (2)^a$, by Lemma \ref{lemam1}, using only (K1) Knuth moves. Hence, by Lemma \ref{congrdir}, we have
$$21 2' (2)^a 1 \equiv_k 21 2'1 (2)^a.$$
Consequentely, $21(2')^a 1 \equiv_k 21 2' 1 (2)^{a-1}$. The proof that $21(2')^a 1' \equiv_k 21 2' 1' (2)^{a-1}$ is done similarly, since Lemma \ref{lemam1} also ensures that $2'(2)^a 1' \equiv_k 2' 1' (2)^a$, using only (K1) Knuth moves.
\end{proof}

\begin{cor}\label{equivrev1}
Let $T$ be a shifted semistandard tableaux in the $1$-string of $Y_{\nu}$, with $\nu = (m,1)$ and $m \geq3$, such that $T$ is not $Y_{\nu}$ neither $\mathsf{evac}(Y_{\nu})$. Let $a = \varepsilon_1 (T)$.

\begin{enumerate}
\item If $T= F_1^a(Y_{\nu})$, then $w (\mathsf{c}_1 (T)) = 21 (2')^{m-a-1} 1 (1')^{a-1} \equiv_k 21 2' 1 (2)^{m-a-2} (1')^{a-1}$.
\item If $T = F_1' F_1^{a-1} (Y_{\nu})$ $w (\mathsf{c}_1 (T)) = 21 (2')^{m-a-1} 1' (1')^{a-1} \equiv_k 21 2' 1' (2)^{m-a-2} (1')^{a-1}$.
\end{enumerate}
\end{cor}

\begin{proof}
Since $T \neq \mathsf{evac}(Y_{\nu})$ and $\mathsf{evac}(Y_{\nu})$ is a lowest weight, then $a = \varepsilon_1(T) < \varepsilon_1 (\mathsf{evac}(Y_{\nu})) = m-1$. Then, $a \leq m - 2$ and so we have that $m-a-1 \geq 1$. Therefore, using Lemma \ref{lemacasomaissimples}, we have
$$21 (2')^{m-a-1} 1 \equiv_k 21 2' 1 (2)^{m-a-2}.$$

Consequentely, by Lemma \ref{congresq}, we have $21 (2')^{m-a-1} 1 (1')^{a-1} \equiv_k 21 2' 1 (2)^{m-a-2} (1')^{a-1}$. The proof for the second case is similar.
\end{proof}

\begin{prop}\label{equivrev2}
Let $T$ be a shifted semistandard tableau in the $1$-string of $Y_{\nu}$, with $\nu = (m,1)$ and $m \geq 3$, such that $T$ is neither $Y_{\nu}$ or $\mathsf{evac}(Y_{\nu})$. Let $a = \varepsilon_1 (T)$.
	\begin{enumerate}
	\item If $T = F_1^a (Y_{\nu})$, then $21 2' 1 (2)^{m-a-2} (1')^{a-1} \equiv_k 2 (1)^{a+1} 2 (2)^{m-a-2} = \sigma_1 (T)$.
	\item If $T = F_1' F_1^a(Y_{\nu})$, then $21 2' 1' (2)^{m-a-2} (1')^{a-1} \equiv_k 2 (1)^{a+1} 2' (2)^{m-a-2} = \sigma_1 (T)$.
	\end{enumerate}
\end{prop}

\begin{proof}
We first prove the first assertion. We have

\begin{align*}
212' 1 (2)^{m-a-2} (1')^{a-1} &\equiv_k \underline{212'}1 (1')^{k-1}(2)^{m-a-2} &\text{Lemmas \ref{lemam1} and \ref{congrdir}}\\
&\equiv_k \underline{22'} 11  (1')^{a-1}(2)^{m-a-2} &\text{(K2)}\\
&\equiv_k 2\underline{211} (1')^{a-1}(2)^{m-a-2} &\text{(S2)}\\
&\equiv_k \underline{21}21 (1')^{a-1}(2)^{m-a-2} &\text{(K2)}\\
&\equiv_k 1\underline{221} (1')^{a-1}(2)^{m-a-2} &\text{(S1)}\\
&\equiv_k 1212(1')^{a-1} (2)^{m-a-2} &\text{(K1)}\\
&\equiv_k \underline{12}1 (1')^{a-1} 2 (2)^{m-a-2} &\text{Lemma \ref{congrdir}}\\
&\equiv_k 211 (1')^{a-1} 2 (2)^{m-a-2} &\text{(S1)}\\
&\equiv_k \underline{21} (1')^{a-1} 12 (2)^{m-a-2} &\text{Lemmas \ref{lemam1}, \ref{congresq}, and \ref{congrdir}}  \\
&\equiv_k 12 (1')^{a-1} 12 (2)^{m-a-2} &\text{(S1)}\\
&\equiv_k 1(1')^{a-1} 2 12 (2)^{m-a-2} &\text{Lemmas \ref{lemam1} and \ref{congresq}}\\
&\equiv_k (1)^a 212 (2)^{m-a-2} &\text{Lemmas \ref{lemadosas} and \ref{congrdir}}\\
&\equiv_k 2 (1)^a 1 2 (2)^{m-a-2} &\text{Lemmas \ref{lemam1} and \ref{congrdir}}\\
&= 2 (1)^{a+1} 2 (2)^{m-a-2}.
\end{align*}

For the second assertion, we have

\begin{align*}
21 2' 1' (2)^{m-a-2} (1')^{a-1} &\equiv_k 212'1' (1')^{a-1} (2)^{m-a-2} &\text{Lemma \ref{lemam1}}\\
&= 212' (1')^a (2)^{m-a-2}\\
&\equiv_k \underline{21} (1')^a 2' (2)^{m-a-2} &\text{Lemma \ref{lemam1}}\\
&\equiv_k 12 (1')^a 2' (2)^{m-a-2} &\text{(S1)}\\
&\equiv_k 1 (1')^a  22' (2)^{m-a-2} &\text{Lemma \ref{lemam1}}\\
&\equiv_k (1)^{a+1} 2 2' (2)^{m-a-2} &\text{Lemma \ref{lemadosas}}\\
&\equiv_k 2 (1)^{a+1} 2' (2)^{m-a-2}. &\text{Lemma \ref{lemam1}}
\end{align*}
\end{proof}

We are now able to prove Theorem \ref{sigmarever}.

\begin{proof}[Proof of Theorem \ref{sigmarever}]
It suffices to show the result for $T$ of straight shape with two rows. Corollary \ref{equivrev1} and Proposition \ref{equivrev2} ensure that the words $w(\mathsf{c}_1 (T))$ and $w(\sigma_1 (T))$ are shifted Knuth equivalent, thus $\mathsf{c}_1 (T) \equiv_k \sigma_1(T)$. Since $T$ and $\sigma_1(T)$ are dual equivalent, this concludes the proof that $\sigma_1(T) = \mathsf{evac}(T)$.
\end{proof}

As a direct consequence of Proposition \ref{highuni} and Proposition \ref{defSchu}, we have that the evacuation of a Yamanouchi tableau is the lowest weight of its crystal, i.e., $F_i'(\mathsf{evac}(T)) = F_i(\mathsf{evac}(T)) = \emptyset$, for all $i \leq I$.

Unlike the type $A$ crystals, the reflection operators $\sigma_i$, for $i \in I$, do not define an action of the symmetric group $\mathfrak{S}_n$ on $\mathcal{B}(\lambda/\mu,n)$. In particular, the braid relations $\sigma_i \sigma_{i+1} \sigma_i = \sigma_{i+1} \sigma_i \sigma_{i+1}$ do not need to hold, as shown in the next example.

\begin{ex}\label{exbraidnot}
Let $\mathcal{B}(\nu,3)$ where $\nu=(5,3,1)$, and consider the shifted semistandard tableau
$$T = \begin{ytableau}
1 & 1 & 1 & 1 & {3'}\\
\none & 2 & 2 & {3'}\\
\none & \none & 3\end{ytableau}$$

The weight of $T$ is given by $\mathsf{wt}(T)=(4,2,3)$. Then, since $\langle \mathsf{wt}(T) , (1,-1,0) \rangle = 4 -2 = 2 > 0$ and $F_1'(T) \neq \emptyset$, we have

$$ \sigma_1(T) = F_1'F_1 (T) = \begin{ytableau}
1 & 1 & {2'} & 2 & {3'}\\
\none & 2 & 2 & {3'}\\
\none & \none & 3
\end{ytableau}$$

Putting $T_1 := \sigma_1 (T)$, we have that $\langle \mathsf{wt}(T_1), (0,1,-1) \rangle = 4 - 3 = 1 > 0$. As $F_2' (T_1) = \emptyset$, we have

$$\sigma_2 \sigma_1 (T) = E_2' F_2^2 (T_1)
	=E_2' (\begin{ytableau}
1 & 1 & {2'} & {3'} & 3\\
\none & 2 & {3'} & 3\\
\none & \none & 3\end{ytableau})
= \begin{ytableau}
1 & 1 & {2'} & 2 & 3\\
\none & 2 & {3'} & 3\\
\none & \none & 3\end{ytableau}$$

And putting $T_2 := \sigma_2 \sigma_1 (T)$, we have $\langle \mathsf{wt}(T_2), (1,-1,0) \rangle = 2 - 3 = -1 <0$ and $F_1' (T_2) = \emptyset$. Thus,

\begin{equation}\label{sigma121}
\sigma_1 \sigma_2 \sigma_1 (T) = \sigma_1 (T_2)
	= E_1' (T_2)
	= \begin{ytableau}
1 & 1 & 1 & 2 & 3\\
\none & 2 & {3'} & 3\\
\none & \none & 3\end{ytableau}
\end{equation}

On the other hand, we have that $\langle \mathsf{wt}(T), (0,1,-1) \rangle = 2-3 = -1 <0$ and $F_2' (T) = \emptyset$, hence

$$ \sigma_2 (T) = E_2' (T) = \begin{ytableau}
1 & 1 & 1 & 1 & 2\\
\none & 2 & 2 & {3'}\\
\none & \none & 3\end{ytableau}$$

We put $T_3 := \sigma_2 (T)$, thus we have that $\langle \mathsf{wt}(T_2), (1,-1,0) \rangle = 4-3 = 1 >0$ and $F_1' (T_3) \neq \emptyset$ and consequently

$$	\sigma_1 \sigma_2 (T) = \sigma_1 (T_3)
	= F_1' (T_3)
	= \begin{ytableau}
1 & 1 & 1 & {2'} & 2\\
\none & 2 & 2 & {3'}\\
\none & \none & 3\end{ytableau}$$

Finally, putting $T_4 := \sigma_1 \sigma_2 (T)$, we have that $\langle \mathsf{wt}(T_4), (0,1,-1) \rangle = 3 -2 = 1 > 0$ and $F_2' (T_4) \neq \emptyset$, thus

\begin{equation}\label{sigma212}
\sigma_2 \sigma_1 \sigma_2 (T) = \sigma_2 (T_4)
	= F_2' F_2 (T_4)
	= \begin{ytableau}
1 & 1 & 1 & {2'} & {3'}\\
\none & 2 & {3'} & 3\\
\none & \none & 3\end{ytableau}
\end{equation}

Then, by \eqref{sigma121} and \eqref{sigma212}, we have that $\sigma_1 \sigma_2 \sigma_1 (T) \neq \sigma_2 \sigma_1 \sigma_2 (T)$.

\end{ex}

However, we have the following result, as in \cite[Section 3.2]{AzMaCo20, AzMaCo09} for ordinary LR tableaux, ensuring that the longest permutation of $\mathfrak{S}_n$ acts on a connected component of $\mathcal{B}(\lambda/\mu,n)$ by sending the highest weight element to the lowest weight element.

\begin{teo}\label{sigmalongperm}
Let $T$ be a LRS tableau in $\mathcal{B}(\lambda/\mu,n)$. Let $\omega_0 = \theta_{i_1} \cdots \theta_{i_k}$ be the longest permutation in $\mathfrak{S}_n$. Then, $\omega_0$ acts on a connected component of $\mathcal{B}(\lambda/\mu,n)$ by sending the highest weight element $T$ to the lowest,  $\sigma_{i_1} \cdots \sigma_{i_k} (T) = T^e$.
\end{teo}

\begin{proof}
Since the operators $\sigma_i$ are coplactic, we may consider $Y_{\nu} = \mathsf{rect}(T)$, $\nu=\mathsf{wt}(T)$. By Proposition \ref{sigmai}, $\sigma_j$ permutes the entries $j$ and $j+1$ on the weight and keeps the shape $\nu$, and as $\omega_0$ is the longest permutation, $\sigma_{i_1} \ldots \sigma_{i_k}$ reverts the weight of $T$. Then, from Proposition \ref{Yevanunique}, we have $\sigma_{i_1} \ldots \sigma_{i_k}Y_{\nu}=\mathsf{evac}(Y_{\nu})$.
\end{proof}

\begin{obs}\label{remarkdih}
Let $G_n := \langle \sigma_1, \ldots, \sigma_{n-1} \rangle$ be the free group generated by the shifted crystal reflection operators $\sigma_i$, for $i \in I$, modulo the relations satisfied by them on shifted semistandard tableaux. We know from Proposition \ref{sigmai} that the relations $\sigma_i^2 = id$ and $\sigma_i \sigma_j = \sigma_j \sigma_i$, for $|i-j|>1$, hold on $G_n$, but not the braid relations of $\mathfrak{S}_n$, $(\sigma_i \sigma_{i+1})^3 = id$, for $i \in [n-2]$. However, since $\mathcal{B}(\nu,n)$ is finite, we have that, given $T \in \mathcal{B}(\nu,n)$, there must exist some $m > 3$ such that $(\sigma_i \sigma_{i+1})^m (T)=T$, for $i \in [n-2]$. We computed some examples on the alphabet $\{1,2,3\}'$ (see Appendix \ref{appendix}), which show that $(\sigma_1 \sigma_2)^m = id$, for $m$ a multiple of at least 90, but we do not know if an upper bound valid for any shape $\nu$ exists.
\end{obs}

\section{A cactus group action on the shifted tableau crystal}\label{sec5}

In this section we show that the adequate restrictions of the shifted Shützenberger involution, which include the crystal reflection operators, define an action of the cactus group $J_n$ on the shifted tableau crystal $\mathcal{B}(\lambda/\mu,n)$. This follows a similar approach to Halacheva \cite{Hala16, Hala20}, where this action is shown for any $\mathfrak{g}$-crystal, for $\mathfrak{g}$ a complex reductive finite-dimensional Lie algebra, which includes the classical Lie algebras.

Let $1 \leq p < q \leq n$ and define $[p,q]:=\{p<\cdots<q\}$. Let $\theta_{p,q \shortminus 1}$ denote the longest permutation in $\mathfrak{S}_{[p,q \shortminus 1]}$ embedded in $\mathfrak{S}_{n-1}$, that is, $\theta_{p, q \shortminus 1}$ sends $i$ to $p+q-i-1$ if $i\in[p,q-1]$, leaving $i$ unchanged otherwise, and put $\theta := \theta_{1,n \shortminus 1}$. The lemma below is straightforward.

\begin{lema}\label{lemacompridito}
Let $1 \leq p < q \leq n$. Then,
\begin{enumerate}
\item $\theta_{1,q \shortminus 1} [p, q-1] = [1,q-p]$.
\item $\theta_{1,q \shortminus p} = \theta_{1,q \shortminus 1} \theta_{p,q \shortminus 1} \theta_{q \shortminus 1}$.
\end{enumerate}
\end{lema}

In what follows, and without loss of generality, we will consider the shifted tableau crystal $\mathcal{B}(\nu,n)$ of straight shape. Given $T \in \mathcal{B}(\nu,n)$ and $1 \leq p < q \leq n$, let $T^{p,q}:=T^p\sqcup T^{p+1}\sqcup\cdots\sqcup T^q$, the tableau obtained from $T$ considering only the entries in $[p,q]'$. In particular, we have $T^{1,n} = T$. By convention, we set $T^{p,p}:=T^p$ and $T^{1,0}=T^{n+1,n}  := \emptyset$.

For each $1 \leq p < q \leq n$, let $\mathcal{B}_{p,q}$ be the subgraph of $\mathcal{B}(\nu,n)$ where the edges coloured in $I\setminus [p,q - 1]$ were removed, and the vertices are the same as in $\mathcal{B}(\nu,n)$ but the letters in $[n]' \setminus [p, q]'$ are ignored. In particular, $\mathcal{B}_{p,p+1}$ is the collection of all the $p$-strings. For each $1 \leq p < q \leq n$, the set $\mathcal{B}(\nu,n)$ is partitioned into classes consisting of the underlying sets of the connected components $\mathcal{B}_{p,q}$. By imposing an equivalence relation on the set $\mathcal{B}(\nu,n)$, defined by $x \sim y$ if $x$ and $y$ are related by a sequence of any lowering or raising
operators coloured in $[p, q-1]$. The equivalence classes are the underlying subsets of the $[p,q]'$-connected
components. An highest weight element of $\mathcal{B}_{p,q}$ is a tableau $T$ such that $E_i (T) = E_i'(T) = \emptyset$, for any $i \in [p,q-1]$ and a lowest weight element is defined similarly.

\begin{lema}\label{highestpq}
Let  $1 \leq p < q \leq n$. Each connected component of $\mathcal{B}_{p,q}$ has unique highest and lowest weight elements. 
\end{lema}

\begin{proof}
If $\{p, \ldots, q\} = \{i,i+1\}$, then $\mathcal{B}_{i,i+1}$ is an union of $i$-strings, each one having unique highest and lowest weight elements. Suppose now that the cardinality of $\{p, \ldots, q\}$ is larger than 2. Each connected component $\mathcal{B}^0$ of $\mathcal{B}_{p,q}$ has vertices the shifted tableaux of shape $\nu$ filled with letters $[p,q]'$, ignoring the letters in $[n]' \setminus [p,q]'$. Thus, given $T \in \mathcal{B}^0$, we may regard it as shifted tableaux of skew shape $\lambda_0^1 / \lambda_0^2$, for strict partitions $\lambda_0^2 \subseteq \lambda_0^1 \subseteq \lambda$, where $\lambda_0^2$ corresponds to the part of tableaux corresponding to the letters $[1,p-1]'$ and $\lambda/\lambda_0^1$ corresponding to $[q+1,n]'$. These parts indeed define the same shapes within each connected components, since the operators corresponding to the edges coloured in $\{p, \ldots,q\}$ leave them unchanged. Then, similarly to the construction of $\eta_{i,i+1}$ in \eqref{eq:etaii1}, we may relabel the filling of these tableaux by
\begin{align}\label{rhopq}
\rho_{p,q} \colon & [p,q]' \to [1,q-p+1]' \nonumber\\
 &i \mapsto i-p+1\\
 &i' \mapsto (i-p+1)'\nonumber
\end{align}
obtaining tableaux of shape $\lambda_0^1/\lambda_0^2$ in the alphabet $[1, q-p+1]'$. Hence, after rectification, each connected component $\mathcal{B}^0$ may be identified with $\mathcal{B}(\lambda_0^1 / \lambda_0^2, q-p+1)$, which is isomorphic to $\mathcal{B}(\nu_0,q-p+1)$, and consequently, by Proposition \ref{isomhigh} it has unique highest and lowest weight elements.
\end{proof}

Recall that, given a shifted semistandard tableau $T \in \mathcal{B}(\nu,n)$, $\eta(T)$ is the unique tableau shifted dual equivalent to $T$ and shifted Knuth equivalent to $\mathsf{c}_n(T)$ \cite{Haim92}, and coincides with the evacuation if $T$ is of straight shape. The following definition formalizes the restriction of the Schützenberger involution to an interval $[p,q]'$.

Given $T \in \mathcal{B}(\nu,n)$, let $\mathcal{B}^0$ be the connected component of $\mathcal{B}_{p,q}$ containing $T$. Similarly to \eqref{eq:etaii1}, we write $\eta (T^{p,q})$ to mean that we first apply the Schützenberger involution $\eta$ to $T^{p,q}$ as an element of $\mathcal{B}(\lambda_0^1 / \lambda_0^2, q-p+1)$ and then add $q-p+1$ to the entries of the obtained tableau, obtaining one in the alphabet $[p,q]'$, in $\mathcal{B}^0$. Thus, we have the following definition.

\begin{defin}\label{def:schupq}
For $1 \leq p < q \leq n$, let $\eta_{p,q}:\mathcal{B}(\nu,n)\rightarrow \mathcal{B}(\nu,n)$ be such that
$$\eta_{p,q} (T) := T^{1,p \shortminus 1} \sqcup \eta(T^{p,q}) \sqcup T^{q+1, n}.$$
In particular, we have $\eta_{1,n} (T) = \eta(T)$, and, from Theorem \ref{sigmarever}, $\eta_{i,i+1}=\sigma_i$.
\end{defin}

\begin{lema}\label{lemautil}
Let $1 \leq p < q \leq n$ and let $i \in [p, q]$. Given $T \in \mathcal{B}(\nu,n)$, we have, whenever the operators are defined:
	\begin{enumerate}
	\item $E_i' \eta_{p,q} (T) = \eta_{p,q} F_{\theta_{p,q \shortminus 1 }(i)}' (T)$.
	\item $E_i \eta_{p,q} (T) = \eta_{p,q} F_{\theta_{p,q \shortminus 1}(i)} (T)$.
	\item $F_i' \eta_{p,q} (T) = \eta_{p,q} E_{\theta_{p,q \shortminus 1}(i)}' (T)$.
	\item $F_i \eta_{p,q} (T) = \eta_{p,q} E_{\theta_{p,q \shortminus 1}(i)} (T)$.
	\item $\mathsf{wt}(\eta_{p,q} (T)) = \theta_{p,q \shortminus 1} \cdot \mathsf{wt} (T)$.
	\end{enumerate}
\end{lema}

\begin{proof}
If $[p,q] = [1,n]$, this corresponds to Proposition \ref{defSchu}. Otherwise, we remark that since the crystal operators are coplactic, we may consider the subinterval $[p,q]$ of $[1,n]$.
\end{proof}

\begin{defin}[\cite{HenKam06}]\label{def:cactus}
Let $J_n$ be the free group with generators $s_{p,q}$, for $1 \leq p < q \leq n$, subject to the relations:
	\begin{enumerate}
	\item $s_{p,q}^2 = id$.
	\item $s_{p,q} s_{k,l} = s_{k,l} s_{p,q}$ for $[p,q] \cap [k,l] = \emptyset$.
	\item $s_{p,q}s_{k,l} = s_{p+q-l,p+q-k} s_{p,q}$ for $[k,l] \subseteq [p,q]$.
	\end{enumerate}
This is called the \emph{$n$-fruit cactus group}.
\end{defin}

\begin{ex}
For $n=2$, we have $J_2 = \langle s_{1,2} | s_{1,2}^2 = 1 \rangle = C_2$, the cyclic group of order 2.
For $n=3$, we have $J_3 = \langle s_{1,2}, s_{1,3}, s_{2,3} | s_{1,2}^2= s_{1,3}^2= s_{2,3}^2=1, s_{1,3} s_{1,2} = s_{2,3}s_{1,3}  \rangle$, which is infinite.
\end{ex}

Note that there exists an epimorphism $J_n \longrightarrow \mathfrak{S}_n$, that sends $s_{p,q}$ to $\theta_{p,q}$. The kernel of this surjection is known as the \emph{pure cactus group} and denoted by $PJ_n$ (see \cite[Section 3.4]{HenKam06}). In the next result, we show that the cactus group $J_n$ acts on the shifted tableau crystal $\mathcal{B}(\lambda/\mu,n)$ via the shifted Scützenberger involutions $\eta_{p,q}$ of Definition \ref{def:schupq}. An example is shown in Figure \ref{fig:crystalcactus}.

\begin{figure}[h]
\begin{center}
\includegraphics[scale=0.7]{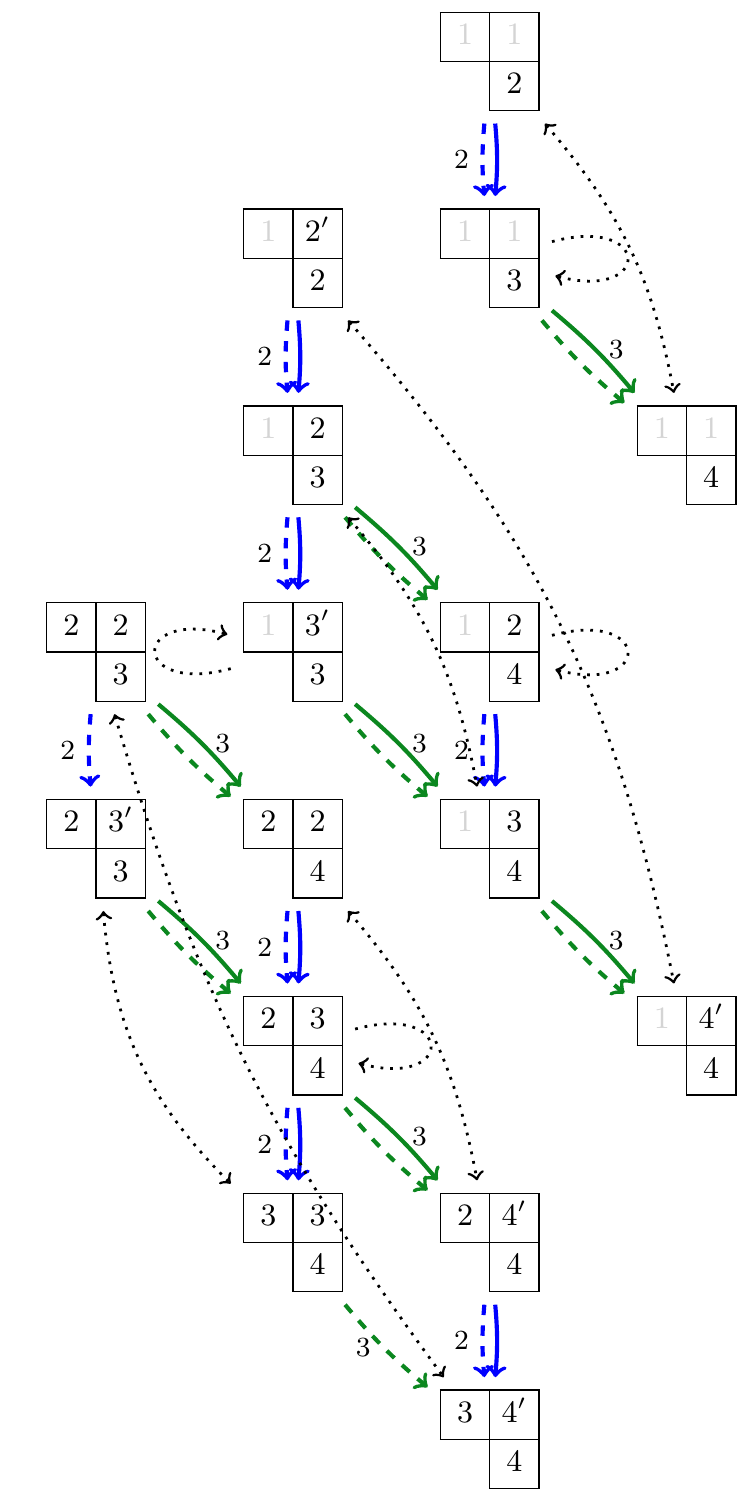}\qquad
\includegraphics[scale=0.7]{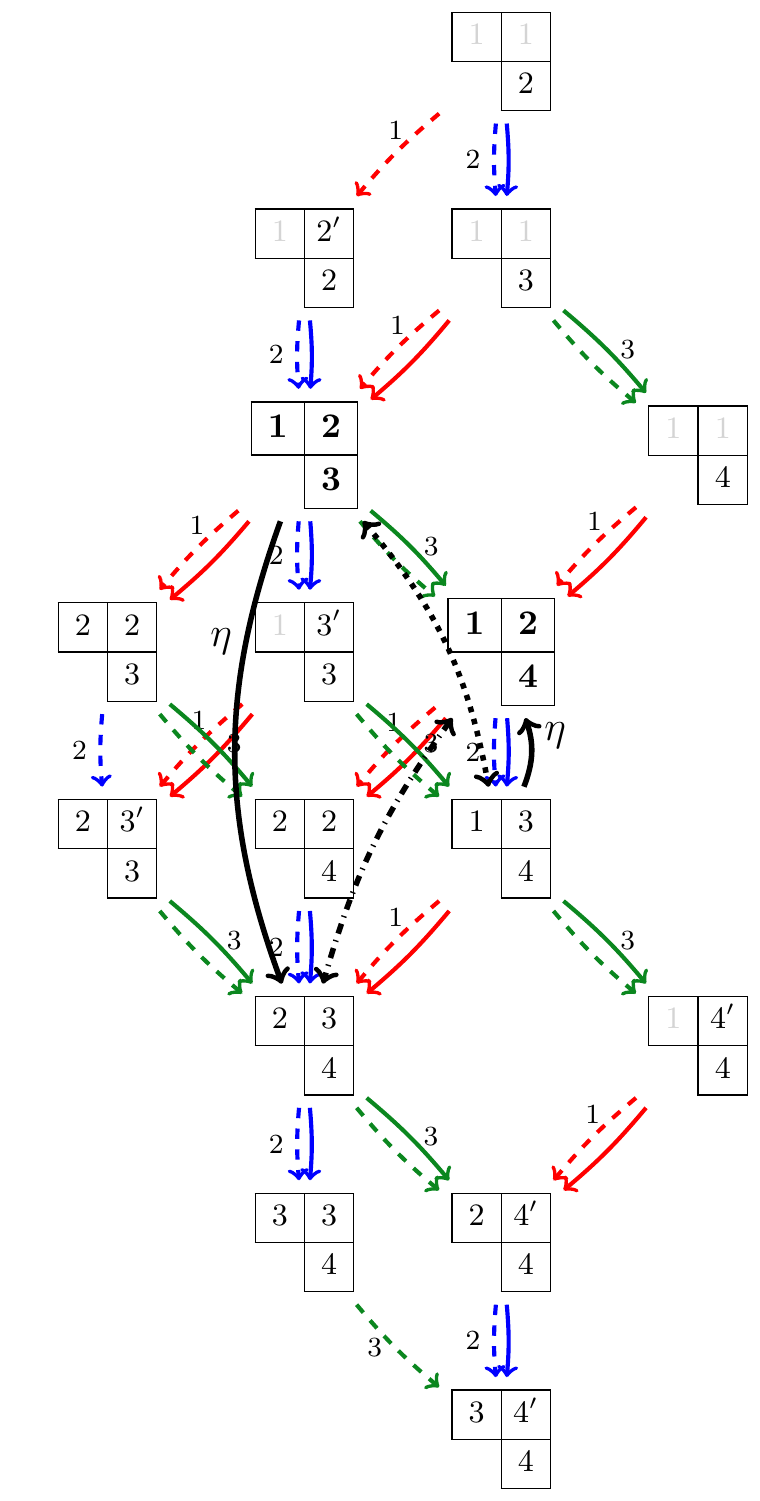}
\end{center}
\caption{On the left, the action of $s_{2,4}$ on $\mathcal{B}(\nu,4)$, with $\nu = (2,1)$. On the right, an illustration of $s_{1,3} s_{1,4} = s_{1,4} s_{2,4}$.}
\label{fig:crystalcactus}
\end{figure}

\begin{teo}\label{teo:cactusaction}
There is a natural action of the $n$-fruit cactus group $J_n$ on the shifted tableau crystal $\mathcal{B}(\lambda/\mu,n)$ given by the group homomorphism:

\begin{align*}
\phi: J_n & \longrightarrow \mathfrak{S}_{\mathcal{B}(\lambda/\mu,n)}\\
s_{p,q} & \longmapsto \eta_{p,q}
\end{align*}
for $1 \leq p < q \leq n$.
\end{teo}

\begin{proof}
To show that $\phi$ is a group homomorphism, we show that the operators $\eta_{p,q}$ satisfy the cactus group relations of Definition \ref{def:cactus}, for any $1 \leq p < q \leq n$. Due to Proposition \ref{isomhigh}, it suffices to do this for shifted crystals of the form $\mathcal{B}(\nu,n)$, as every connected component of $\mathcal{B}(\lambda/\mu,n)$ is isomorphic to these, via rectification. Since the operator $\eta$ is an involution on each connected component of $\mathcal{B}_{p,q}$, it follows that $s_{p,q}^2 = id$ for all $1\leq p < q \leq n$.
The second relation is a direct consequence of the definition. Suppose that $[p,q] \cap [k,l] = \emptyset$, and without loss of generality assume that $1 \leq p < q < k < l \leq n$. We remark that each $\eta_{p,q}$ acts on the entries of $[p,q]'$, leaving the entries in $[n]' \setminus [p,q]'$ unchanged. Thus,
	\begin{align*}
	\eta_{p,q} \eta_{k,l} (T)	&= T^{1,p-1} \sqcup \eta(T^{p,q}) \sqcup T^{q+1,k-1} \sqcup \eta(T^{k,l}) \sqcup T^{l+1,n}\\
	&= \eta_{k,l}\eta_{p,q} (T^{1,p-1} \sqcup T^{p,q} \sqcup T^{q+1,k-1} \sqcup T^{k,l} \sqcup T^{l+1,n}).
	\end{align*}	
	
For the third relation, we claim that it suffices to show that, for any $[p,q] \subseteq [1,n]$,
\begin{equation}\label{eq:3cact_1}
\eta_{1,q} \eta_{p,q} = \eta_{1,1+q-p} \eta_{1,q}.
\end{equation}

Given $[k,l] \subseteq [p,q]$, we show that \eqref{eq:3cact_1} implies the third relation $\eta_{p,q} \eta_{k,l} = \eta_{p+q-l,p+q-k} \eta_{p,q}$. Since, in particular, $[k,l] \subseteq [1,q]$, then \eqref{eq:3cact_1} ensures that
\begin{equation}\label{eq:3cact_2}
\eta_{1,q} \eta_{k,l} = \eta_{1+q-l,1+q-k} \eta_{1,q}.
\end{equation}

Moreover, $[k,l] \subseteq [p,q]$ implies that $[k-p+1,l-p+1] \subseteq [1,1+q-p]$, and consequently, by \eqref{eq:3cact_1},
\begin{equation}\label{eq:3cact_3}
\eta_{1,1+q-p} \eta_{k-p+1,l-p+1} = \eta_{q-l+1,q-k+1} \eta_{1,1+q-p}.
\end{equation} 

We also remark that $[k,l] \subseteq [p,q]$ implies that $[q+p-l,q+p-k] \subseteq [1,q]$, and thus,
\begin{equation}\label{eq:3cact_4}
\eta_{1,q} \eta_{q+p-l,q+p-k} = \eta_{k-p+1,l-p+1} \eta_{1,q}.
\end{equation}

Then, for $[k,l] \subseteq [p,q]$, and using the fact that $\eta_{p,q}$ is an involution, for any $[p,q]$, we have
\begin{alignat*}{2}
\eta_{p,q} \eta_{k,l} &= \eta_{1,q} \eta_{1,1+q-p} \eta_{1,q} \eta_{k,l} \eta_{1,q} \eta_{1,q} & \qquad \text{\eqref{eq:3cact_1}}\\
&= \eta_{1,q} \eta_{1,1+q-p} \eta_{1+q-l,1+q-k} \eta_{1,1+q-p} \eta_{1,1+q-p} \eta_{1,q} & \qquad \text{\eqref{eq:3cact_2}}\\
&= \eta_{1,q} \eta_{k-p+1,l-p+1} \eta_{1,q} \eta_{1,q} \eta_{1,1+q-p} \eta_{1,q} & \qquad \text{\eqref{eq:3cact_3}}\\
&= \eta_{q+p-l,q+p-l} \eta_{p,q} & \qquad \text{\eqref{eq:3cact_4} and \eqref{eq:3cact_1}}
\end{alignat*}

We will now prove \eqref{eq:3cact_1}. Given $T \in \mathcal{B}(\nu,n)$, if $T$ is an isolated vertex of $\mathcal{B}_{p,q}$, then the result is trivially true. Hence, we assume that $T$ is in a connected component $\mathcal{B}_0$ of $\mathcal{B}_{p,q}$ with at least two vertices. This component $\mathcal{B}_0$ is contained in $\mathcal{B}(\nu,q) \subseteq \mathcal{B}(\nu,n)$. We may regard $\mathcal{B}(\nu,q)$ as the shifted tableau crystal obtained from $\mathcal{B}(\nu,n)$ by considering only the arrows labelled in $[1,q-1]$ and the vertices of $\mathcal{B}(\nu,n)$, deleting the letters in $[q+1,n]'$, which yield straight-shaped tableaux. By Lemma \ref{highestpq}, this connected component $\mathcal{B}_0$ has a unique highest weight element, $T_{0}^{\mathsf{high}}$, as well as a lowest weight

\begin{equation}\label{Tlowhigh}
T_{0}^{\mathsf{low}} = \eta_{p,q} (T_{0}^{\mathsf{high}}),
\end{equation}
which is also unique by Lemma \ref{lemautil}. Since we assume the cardinality of $\mathcal{B}_0$ to be greater or equal to 2, we have that $T^{\mathsf{high}}_0 \neq T^{\mathsf{low}}_0$.
Moreover, $\mathcal{B}(\nu,q)$ has a unique highest weight element $Y_{\nu}$ (which is the same as $\mathcal{B}(\nu,n)$, ignoring the letters in $[q+1,n]'$) and a lowest weight element $\eta_{1,q} (Y_{\nu})$. Then, we have:
	\begin{equation}\label{TTlow}
	\begin{aligned}	
	T &= F_{i_1}'^{m_1} F_{i_1}^{n_1} \ldots F_{i_k}'^{m_k} F_{i_k}^{n_k} (T_{0}^{\mathsf{high}})\\
	T_{0}^{\mathsf{low}} &= E_{j_1}'^{a_1} E_{j_1}^{b_1} \ldots E_{j_l}'^{a_l} E_{j_l}^{b_l} \eta_{1,q} (Y_{\nu})
	\end{aligned}
	\end{equation}
for some $i_1, \ldots, i_k \in [p,q-1]$, $j_1, \ldots, j_l \in [1, q-1]$, $m_i, a_j \in \{0,1\}, n_i, b_j \geq 0$. Consequently,
	\begin{align*}
	\eta_{1,q} \eta_{p,q} (T) &= \eta \eta_{p,q} F_{i_1}'^{m_1} F_{i_1}^{n_1} \ldots F_{i_k}'^{m_k} F_{i_k}^{n_k} (T_{0}^{\mathsf{high}}) &\text{\eqref{TTlow}}\\
	&= \eta_{1,q} E_{\theta_{p,q \shortminus 1} (i_1)}'^{m_1} E_{\theta_{p,q \shortminus 1}(i_1)}^{n_1} \ldots E_{\theta_{p,q \shortminus 1}(i_k)}'^{m_k} E_{\theta_{p,q \shortminus 1}(i_k)}^{n_k} \eta_{p,q} (T_{0}^{\mathsf{high}}) &\text{Lemma \ref{lemautil}}\\
	&= \eta_{1,q} E_{\theta_{p,q \shortminus 1} (i_1)}'^{m_1} E_{\theta_{p,q \shortminus 1}(i_1)}^{n_1} \ldots E_{\theta_{p,q \shortminus 1}(i_k)}'^{m_k} E_{\theta_{p,q \shortminus 1 }(i_k)}^{n_k} (T_{0}^{\mathsf{low}}) &\text{\eqref{Tlowhigh}}\\
	&= \eta_{1,q} E_{\theta_{p,q \shortminus 1} (i_1)}'^{m_1} E_{\theta_{p,q \shortminus 1 }(i_1)}^{n_1} \ldots E_{\theta_{p,q \shortminus 1}(i_k)}'^{m_k} E_{\theta_{p,q \shortminus 1 }(i_k)}^{n_k} \\
	&\qquad E_{j_1}'^{a_1} E_{j_1}^{b_1} \ldots E_{j_l}'^{a_l} E_{j_l}^{b_l} \eta_{1,q} (Y_{\nu}) &\text{\eqref{TTlow}}\\
	&= F_{\theta_{1,q \shortminus 1} \theta_{p,q \shortminus 1} (i_1)}'^{m_1} F_{\theta_{1,q \shortminus 1}\theta_{p,q \shortminus 1}(i_1)}^{n_1} \ldots F_{\theta_{1,q \shortminus 1}\theta_{p,q \shortminus 1}(i_k)}'^{m_k} F_{\theta_{1,q \shortminus 1}\theta_{p,q \shortminus 1}(i_k)}^{n_k}\\
	&\qquad F_{\theta_{1,q \shortminus 1}(j_1)}'^{a_1} F_{\theta_{1,q \shortminus 1}(j_1)}^{b_1} \ldots F_{\theta_{1,q \shortminus 1}(j_l)}'^{a_l} F_{\theta_{1,q \shortminus 1}(j_l)}^{b_l} (\eta_{1,q}^2 Y_{\nu}) &\text{Lemma \ref{lemautil}}\\
	&= F_{\theta_{1,q \shortminus 1}\theta_{p,q \shortminus 1} (i_1)}'^{m_1} F_{\theta_{1,q \shortminus 1}\theta_{p,q \shortminus 1}(i_1)}^{n_1} \ldots F_{\theta_{1,q \shortminus 1}\theta_{p,q \shortminus 1}(i_k)}'^{m_k} F_{\theta_{1,q \shortminus 1}\theta_{p,q \shortminus 1}(i_k)}^{n_k}\\
	&\qquad F_{\theta_{1,q \shortminus 1}(j_1)}'^{a_1} F_{\theta_{1,q \shortminus 1}(j_1)}^{b_1} \ldots F_{\theta_{1,q \shortminus 1}(j_l)}'^{a_l} F_{\theta_{1,q \shortminus 1}(j_l)}^{b_l} (Y_{\nu}).
	\end{align*}		
Thus, we have
\begin{equation}\label{eq:cact1}
\begin{aligned}
\eta_{1,q} \eta_{p,q} (T) = &F_{\theta_{1,q \shortminus 1}\theta_{p,q \shortminus 1} (i_1)}'^{m_1} F_{\theta_{1,q \shortminus 1}\theta_{p,q \shortminus 1}(i_1)}^{n_1} \ldots F_{\theta_{1,q \shortminus 1}\theta_{p,q \shortminus 1}(i_k)}'^{m_k} F_{\theta_{1,q \shortminus 1}\theta_{p,q \shortminus 1}(i_k)}^{n_k}\\ 
&F_{\theta_{1,q \shortminus 1}(j_1)}'^{a_1} F_{\theta_{1,q \shortminus 1}(j_1)}^{b_1} \ldots F_{\theta_{1,q \shortminus 1}(j_l)}'^{a_l} F_{\theta_{1,q \shortminus 1}(j_l)}^{b_l} (Y_{\nu}).
\end{aligned}
\end{equation}

If $Q$ is a tableau in a connected component of $\mathcal{B}_{p,q}$, then either $Q$ is an isolated vertex, meaning that $F_i,F_i', E_i$ and $E_i'$ are undefined on $Q$, for $i \in [p,q-1]$, or else $Q$ is adjacent to some other element, meaning that $F_i,F_i', E_i$ and $E_i'$ are not simultaneously undefined on $Q$, for $i \in [p,q-1]$. Proposition \ref{defSchu} ensures that, in the former case, the operators $F_i,F_i', E_i$ and $E_i'$ are undefined on $\eta_{1,q}(Q)$, for $i \in [1,q-p] = \theta_{1,q \shortminus 1} [p,q-1]$, and in the latter case, that the operators are not simultaneously undefined for $i \in [1,q-p]$. In either cases, $\eta_{1,q}(Q)$ is in some connected component of $\mathcal{B}_{1,q-p+1}$. Hence, $\eta_{1,q}$ takes the connected component $\mathcal{B}_0$ to another connected component $\mathcal{B}_1$ of $\mathcal{B}_{1,q-p+1}$. Lemma \ref{lemaphieta} implies that $\varphi_j(T) = \varepsilon_{\theta_{1, q \shortminus 1}} \eta_{1,q} (T)$, for any $j \in [p,q-1]$, and thus $\eta_{1,q}$ exchanges highest and lowest weights of $\mathcal{B}_0$ and $\mathcal{B}_1$. Thus, $\eta_{1,q}(T_0^{\mathsf{low}})$ and $\eta_{1,q} (T_0^{\mathsf{high}})$ are, respectively, the highest and lowest weight elements of $\mathcal{B}_1$. Since $\mathcal{B}_1$ is a component of $\mathcal{B}_{1,q-p+1}$, then $\eta_{1,q-p+1}$ maps its lowest weight element to its highest weight element, hence $\eta_{1,q-p+1} \eta_{1,q} (T_0^{\mathsf{high}})$ is the highest weight in $\mathcal{B}_1$. Then, we have
\begin{equation}\label{etalowzero}
\eta_{1,q-p+1} \eta_{1,q} (T_{0}^{\mathsf{high}}) = \eta_{1,q} (T_{0}^{\mathsf{low}}).
\end{equation}
and thus we may write
\begin{align*}
\eta_{1,q-p+1} \eta_{1,q} (T) &= \eta_{1,q-p+1} F_{i_1}'^{m_1} F_{i_1}^{n_1} \ldots F_{i_k}'^{m_k} F_{i_k}^{n_k} (T_{0}^{\mathsf{high}}) &\text{\eqref{TTlow}}\\
&= \eta_{1,q-p+1} E_{\theta_{1,q \shortminus 1}(i_1)}'^{m_1} E_{\theta_{1,q \shortminus 1}(i_1)}^{n_1} \ldots E_{\theta_{1,q \shortminus 1}(i_k)}'^{m_k} E_{\theta_{1,q \shortminus 1}(i_k)}^{n_k} (\eta_{1,q}(T_{0}^{\mathsf{high}})) &\text{Proposition \ref{defSchu}}\\
&= F_{\theta_{1,q-p}\theta_{1,q \shortminus 1}(i_1)}'^{m_1} F_{\theta_{1,q-p}\theta_{1,q \shortminus 1}(i_1)}^{n_1} \ldots F_{\theta_{1,q-p}\theta_{1,q \shortminus 1}(i_k)}'^{m_k}\\
&\phantom{==} F_{\theta_{1,q-p}\theta_{1,q \shortminus 1}(i_k)}^{n_k} \eta_{1,q-p+1} \eta_{1,q} (T_{0}^{\mathsf{high}}) &\text{Lemma \ref{lemautil}}\\
&= F_{\theta_{1,q \shortminus 1}\theta_{p,q \shortminus 1}(i_1)}'^{m_1} F_{\theta_{1,q \shortminus 1}\theta_{p,q \shortminus 1}(i_1)}^{n_1} \ldots\\
&\qquad \ldots F_{\theta_{1,q \shortminus 1}\theta_{p,q \shortminus 1}(i_k)}'^{m_k} F_{\theta_{1,q \shortminus 1}\theta_{p,q \shortminus 1}(i_k)}^{n_k} \eta_{1,q}(T_{0}^{\mathsf{low}}) &\text{\eqref{etalowzero} and Lemma \ref{lemacompridito}}\\
&=  F_{\theta_{1,q \shortminus 1}\theta_{p,q \shortminus 1}(i_1)}'^{m_1} F_{\theta_{1,q \shortminus 1}\theta_{p,q \shortminus 1}(i_1)}^{n_1} \ldots\\
&\qquad \ldots F_{\theta_{1,q \shortminus 1}\theta_{p,q \shortminus 1}(i_k)}'^{m_k} F_{\theta_{1,q \shortminus 1}\theta_{p,q \shortminus 1}(i_k)}^{n_k} \eta_{1,q}(E_{j_1}'^{a_1} E_{j_1}^{b_1} \ldots E_{j_l}'^{a_l} E_{j_l}^{b_l} \eta_{1,q}(Y_{\nu})) &\text{\eqref{TTlow}}\\
&= F_{\theta_{1,q \shortminus 1}\theta_{p,q \shortminus 1}(i_1)}'^{m_1} F_{\theta_{1,q \shortminus 1}\theta_{p,q \shortminus 1}(i_1)}^{n_1} \ldots F_{\theta_{1,q \shortminus 1}\theta_{p,q \shortminus 1}(i_k)}'^{m_k} F_{\theta_{1,q \shortminus 1}\theta_{p,q \shortminus 1}(i_k)}^{n_k}\\
&\qquad F_{\theta_{1,q \shortminus 1}(j_1)}'^{a_1} F_{\theta_{1,q \shortminus 1}(j_1)}^{b_1} \ldots F_{\theta_{1,q \shortminus 1}(j_l)}'^{a_l} F_{\theta_{1,q \shortminus 1}(j_l)}^{b_l} (Y_{\nu}). &\text{Proposition \ref{defSchu}}
\end{align*}
and so
\begin{equation}\label{eq:cact2}
\begin{aligned}
\eta_{1,q-p+1} \eta_{1,q} (T) = &F_{\theta_{1,q \shortminus 1}\theta_{p,q \shortminus 1}(i_1)}'^{m_1} F_{\theta_{1,q \shortminus 1}\theta_{p,q \shortminus 1}(i_1)}^{n_1} \ldots F_{\theta_{1,q \shortminus 1}\theta_{p,q \shortminus 1}(i_k)}'^{m_k} F_{\theta_{1,q \shortminus 1}\theta_{p,q \shortminus 1}(i_k)}^{n_k}\\ 
&F_{\theta_{1,q \shortminus 1}(j_1)}'^{a_1} F_{\theta_{1,q \shortminus 1}(j_1)}^{b_1} \ldots F_{\theta_{1,q \shortminus 1}(j_l)}'^{a_l} F_{\theta_{1,q \shortminus 1}(j_l)}^{b_l} (Y_{\nu}).
\end{aligned}
\end{equation}
Then, comparing \eqref{eq:cact1} and \eqref{eq:cact2}, we get $\eta_{1,q}\eta_{p,q}(T) = \eta_{1,q-p+1} \eta_{1,q} (T)$, which concludes the proof.
\end{proof}

\begin{cor}
Let $T \in \mathsf{ShST}(\lambda/\mu,n)$ and $1 \leq p < q \leq n$. Then,
$$\eta_{p,q} (T) = \eta_{1,q} \eta_{1,q-p+1} \eta_{1,q} (T).$$
In particular, for $T \in \mathsf{ShST}(\nu,n)$, we have
$$\eta_{p,q}(T) = \mathsf{evac}_{q} \mathsf{evac}_{q-p+1} \mathsf{evac}_{q} (T),$$
where $\mathsf{evac}_{q} (T) := \mathsf{evac} (T^{1,q}) \sqcup T^{q+1,n}$.
\end{cor}

\begin{proof}
Since $[p,q] \subseteq [1,q]$, the third relation of the cactus group ensures that $s_{p,q} = s_{1,q} s_{1,q-p+q} s_{1,q}$ in $J_n$. Then, since $\phi$ from Theorem \ref{teo:cactusaction} is a group homomorphism, we have
$$\eta_{p,q} = \eta_{1,q} \eta_{1,q-p+1} \eta_{1,q}.$$
The second identity follows from Theorem \ref{defSchu} which ensures that $\eta_{1,q}$ coincides with $\mathsf{evac}_q$ in $\mathcal{B}(\nu,n)$.
\end{proof}

As shown by Halacheva \cite{Hala16,Hala20}, the cactus group $J_n$ acts on the crystal of Young tableaux of a given shape filled in $[n]$, via the type $A$ Schützenberger involutions $\eta_{p,q}$, for $1 \leq p > q \leq n$. The symmetric group $\mathfrak{S}_n$ also acts on this crystal via the type $A$ crystal reflection operators $\sigma_i$, for $1 \leq i < n$. Hence, we have that the action of $J_n$ on type $A$ crystals factorizes in the quotient of this group by the normal subgroup $\langle (s_{i,i+1} s_{i+1,i+2})^3, i \in [n-2] \rangle$ generated by the braid relations, which is contained in the kernel of this action. This is not the case for the action $\phi$ in Theorem \ref{teo:cactusaction} on the shifted tableau crystal $\mathcal{B}(\nu,n)$, since the braid relations of the symmetric group do not need to hold (see Example \ref{exbraidnot}), hence $\langle (s_{i,i+1} s_{i+1,i+2})^3, i \in [n-2] \rangle \nsubseteq ker \phi$. 

\appendix
\section{Shifted tableau crystals for $n=3$}\label{appendix}
In this appendix we present some larger examples of the shifted tableau crystal $\mathcal{B}(\nu,n)$, for $n=3$. Shifted tableaux will be enumerated according to their list enumeration in SageMath (this enumeration starts on 0 by default, we choose to start on 1). We remark that the definiton of shifted tableaux used in SageMath is the one in \cite{AsOg18}, which is not the same as in \cite{GLP17}. Here, shifted tableaux are not required to be in canonical form and the main diagonal must have no have primed entries. However, we remark that both definitions coincide for straight-shaped tableaux that are filled with $[n]'$, where $n$ is equal to the number of rows, as this ensures that the first occurrence of each letter $i$ or $i'$ appears in the main diagonal, thus being unprimed in canonical form. For instance, the tableau $T \in \mathcal{B}(\nu,3)$, with $\nu_1 = (5,3,1)$ (see Figure \ref{fig:crystal531}), in Example \ref{exbraidnot} is $T_{61}$.

\begin{figure}[h!]
\centering
\begin{BVerbatim}
sage: SPT=ShiftedPrimedTableaux([5,3,1], max_entry=3)
sage: L=SPT.list()
sage: T=SPT([[1,1,1,1,'3p'],[2,2,'3p'],[3]])
sage: L.index(T)+1
61
\end{BVerbatim}
\caption{A SageMath code listing the shifted semistandard tableaux of shape $\nu = (5,3,1)$, in the alphabet $\{1,2,3\}'$.}
\end{figure}

We have seen in Example \ref{exbraidnot} that $(\sigma_1 \sigma_2)^3 (T_{61}) \neq T_{61}$. However, we have that $(\sigma_1 \sigma_2)^9 (T_{61})= T_{61}$. If we set $m_i := min \{m: (\sigma_1 \sigma_2)^m (T_i) = T_i\}$, for $i \in [|\mathcal{B}(\nu_1,3)|]=[64]$, then, we have the following in $\mathcal{B}(\nu,n)$:

$$m_i = \begin{cases}
3 & \text{for}\; i \in \{1,2,3,9,10,11,15,18,19,22,23,27,28,31,37,38,44,47\}\\
5 & \text{for}\; i \in \{6,13,25,33,42,50,54,58,62,64\}\\
9 & \text{otherwise}
\end{cases}$$

Therefore, taking $m = lcm (3,5,9) = 45 $, we have that $(\sigma_1 \sigma_2)^m (T) = T$, for all $T \in \mathcal{B}(\nu_1,3)$.

Similarly, for $\nu_2 = (5,2,1)$ (see Figure \ref{fig:crystal521}), we have that, for all $i \in [|\mathcal{B}(\nu_2,3)|] = [48]$,

$$m_i = \begin{cases}
3 & \text{for}\; i \in \{1,2,3,4,5,7,10,11,12,13,14,15,16,19,20,\\
 & \qquad \qquad 21,23,24,25,26,30,31,32,33,35,36,39,40,42,45\}\\
9 & \text{otherwise}
\end{cases}$$

Hence, putting $m = lcm (3,9) = 9$, we have that $(\sigma_1 \sigma_2)^m (T) = T$ for all $T \in \mathcal{B}(\nu_2,3)$. The following table summarizes these and other computations we did. We remark that Lemma \ref{sigmadetached} ensures that the effect of $\sigma_i$ on $\mathsf{rect}(T^{i,i+1})$ (which has, at most, two rows) does not depend on the first diagonals, except for one, with two elements. Thus, it suffices to check strict partitions whose last part is equal to one. This means that the results obtained for $(3,2,1)$ are the same for $(3+k,2+k,1+k)$, for $k \geq 1$.

\begin{table}[H]
\begin{center}
\begin{tabular}{|c|c|c|}
\hline
\multirow{2}{*}{$\nu$} & \multirow{2}{*}{$|\mathcal{B}(\nu,3)|$} & {\footnotesize least $m$ such that $(\sigma_1 \sigma_2)^m(T)=T$}\\
& & {\footnotesize for all $T \in \mathcal{B}(\nu,3)$}\\
\hline
$(3,2,1)$ & 8 & 3 \\
\hline
$(4,2,1)$ & 24 & 3 \\
\hline
$(4,3,1)$ & 24 & 3 \\
\hline
$(5,2,1)$ & 48 & 9 \\
\hline
$(5,3,1)$ & 64 & 45 \\
\hline
$(5,4,1)$ & 48 & 9 \\
\hline
$(6,2,1)$ & 80 & 18\\
\hline
$(6,3,1)$ & 120 & 18\\
\hline
$(6,4,1)$ & 120 & 18
 \\
\hline
\end{tabular}
\end{center}
\end{table}

Given $\nu$ a strict partition, since $\mathcal{B}(\nu,3)$ is finite, we know that there exists a $m > 3$ such that $(\sigma_1 \sigma_2)^m (T)=T$ for all $T \in \mathcal{B}(\nu,3)$ (see Remark \ref{remarkdih}). These computations show that, if there exits an $m$ such that $(\sigma_1 \sigma_2)^m = id$, for any $\nu$, then it should be greater or equal to $lcm(3,9,18,45) = 90$. However, an upper bound for any $\nu$ is not known.
 
\begin{figure}
\includegraphics[scale=0.5]{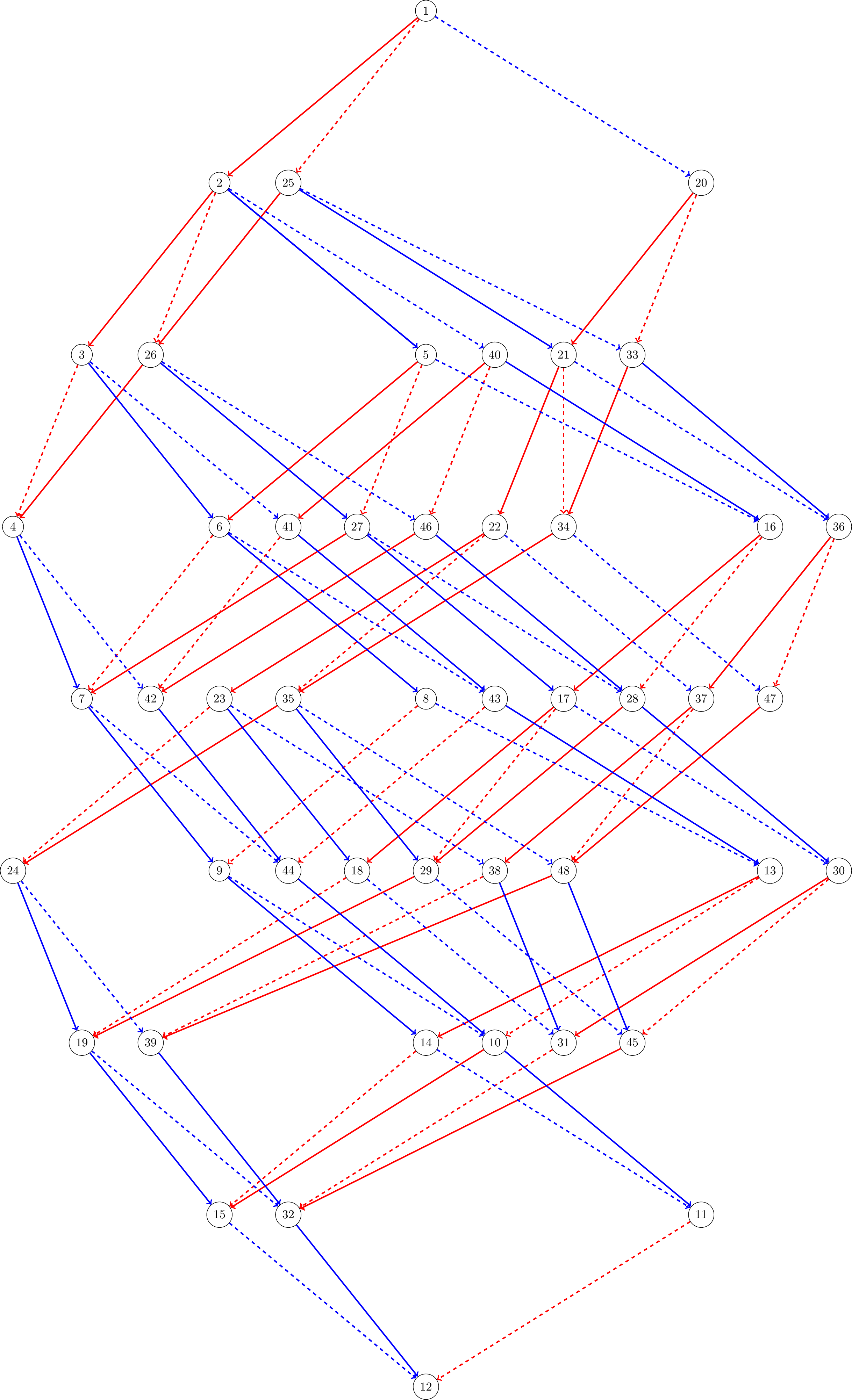}
\caption{Shifted tableau crystal graph $\mathcal{B}(\nu,3)$, with $\nu = (5,2,1)$. The operators $F_1, F_1'$ are in red and the $F_2, F_2'$ are in blue. Vertices with the same weight are grouped together.}
\label{fig:crystal521}
\clearpage
\end{figure}

\begin{figure}
\includegraphics[scale=0.5]{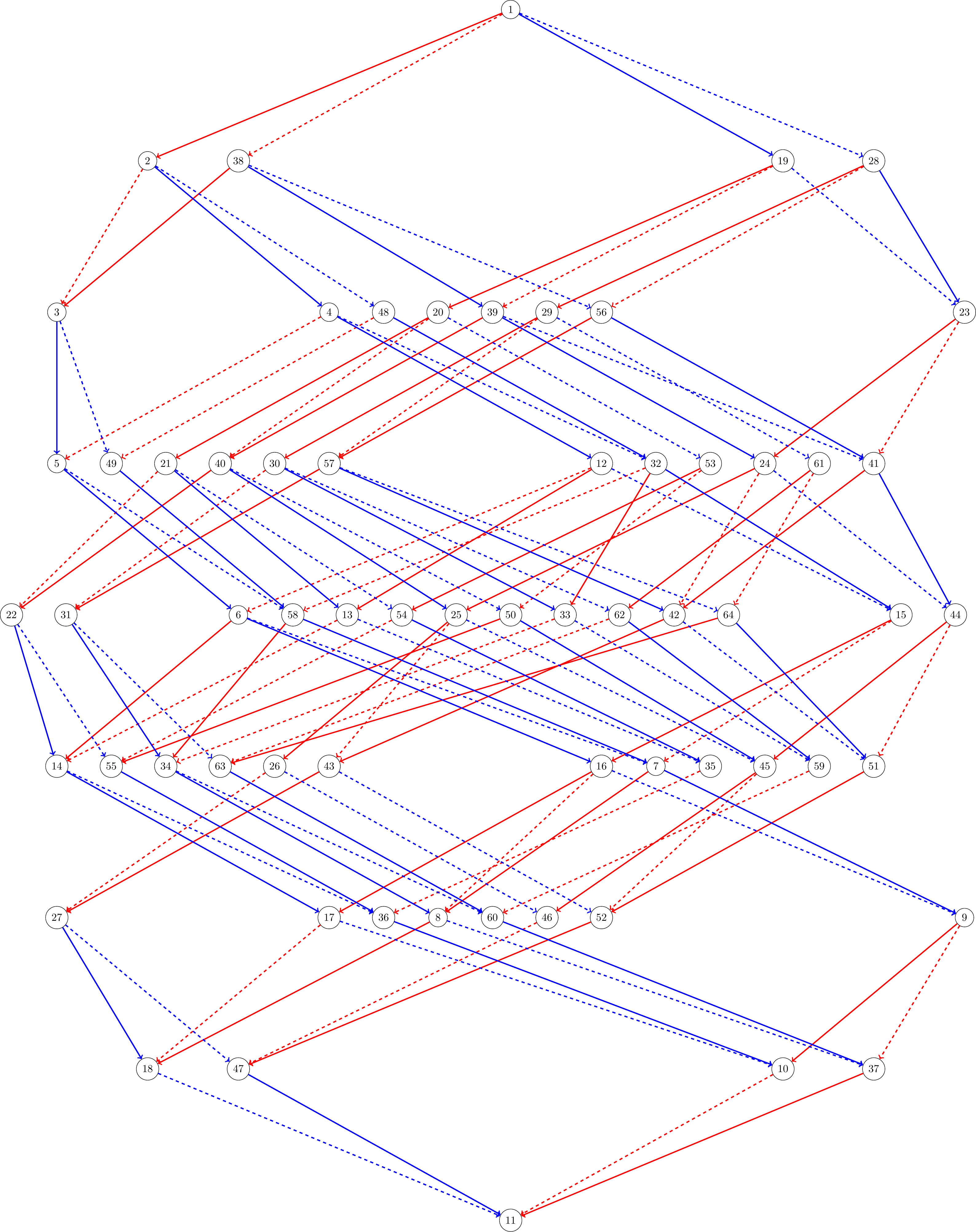}
\caption{Shifted tableau crystal graph $\mathcal{B}(\nu,3)$, with $\nu = (5,3,1)$. The operators $F_1, F_1'$ are in red and the $F_2, F_2'$ are in blue. Vertices with the same weight are grouped together.}
\label{fig:crystal531}
\clearpage
\end{figure}
\clearpage

\section*{Acknowledgements}
The author wishes to express her gratitude to her supervisors Olga Azenhas and Maria Manuel Torres and to acknowledge the hospitality of the Department of Mathematics of University of Coimbra. We also thank the organizers of the Sage Days 105 and OpenDreamKit, for the financial support to participate in the workshop, which was particularly useful for learning more about SageMath. 

%\printbibliography

\bibliographystyle{siam}
\bibliography{bibliography}

\end{document}